\documentclass[12pt]{amsart}
\usepackage[margin=1in]{geometry}
\usepackage[foot]{amsaddr}

\usepackage{amsthm, amsmath, amssymb}
\usepackage[mathscr]{euscript}
\usepackage{tikz}
\usepackage{dsfont}
\usepackage{mathtools}
\usepackage{hyperref}
\hypersetup{
    colorlinks=true,
    linkcolor=blue,
    filecolor=magenta,      
    urlcolor=cyan,
}
\usepackage{stmaryrd}
\usepackage{enumitem}

\newcommand\bord{\partial\Omega}
\newcommand\B{\mathcal{B}}

\newcommand\Hpo{\dot{\B}(\bord)}
\newcommand\Hmo{\dot{\B}'(\bord)}
\newcommand\Scal{\mathcal{S}}
\newcommand\D{\mathcal{D}}
\newcommand\DO{\mathscr{D}}
\newcommand\K{\mathcal{K}}
\newcommand\Sd{\dot{\Scal}}
\newcommand\Dd{\dot{\D}}
\newcommand\DOd{\dot{\DO}}
\newcommand\Kd{\dot{\K}}
\newcommand\Vd{\dot{\mathcal{V}}}
\newcommand\Wd{\dot{\mathcal{W}}}
\newcommand\R{\mathbb{R}}
\newcommand\Z{\mathbb{Z}}
\newcommand\N{\mathbb{N}}
\newcommand\C{\mathbb{C}}
\newcommand\Lrm{\mathrm{L}}
\newcommand\dx{{\mathrm d}x}
\newcommand\dy{{\mathrm d}y}
\newcommand\nbord{\R^n\backslash\bord}
\newcommand\cbord{\C\backslash\bord}
\newcommand\Trd{\dot{\mathrm T}\mathrm r}
\newcommand\trd{\dot{\mathrm t}\mathrm r}

\newcommand{\ddno}[2]{\frac{\dot{\partial}_{#2}#1}{\partial\nu}}

\theoremstyle{plain}
\newtheorem{theorem}{Theorem}[section]
\newtheorem{lemma}[theorem]{Lemma}
\newtheorem{proposition}[theorem]{Proposition}
\newtheorem{corollary}[theorem]{Corollary}

\theoremstyle{definition}
\newtheorem{definition}[theorem]{Definition}
\newtheorem{remark}[theorem]{Remark}
\numberwithin{equation}{section}

\begin{document}
\title[Convergence of layer potentials on extension domains]{Convergence of layer potentials and Riemann-Hilbert problem on extension domains}

\author[G. Claret]{Gabriel Claret}
\author[A. Rozanova-Pierrat]{Anna Rozanova-Pierrat}
\address{CentraleSupélec, Université Paris-Saclay, 9 rue Joliot Curie, 91190 Gif-sur-Yvette, France}
\email{gabriel.claret@centralesupelec.fr, anna.rozanova-pierrat@centralesupelec.fr}
\author[A. Teplyaev]{Alexander Teplyaev}
\address{University of Connecticut, 352 Mansfield Rd, CT 06269 Storrs, Connecticut, USA}
\email{teplyaev@uconn.edu}


\begin{abstract}
We prove the convergence of layer potential operators for the harmonic transmission problem over a sequence of converging two-sided extension domains. Consequently, the Neumann-Poincaré operators, Calderón projectors, and associated Neumann series converge in this setting. As a result, we generalize the notion of Cauchy integrals and, in a sense, of Hilbert transforms for a class of extension domains. Our approach relies on dyadic approximations of arbitrary open sets, considering convergence in terms of characteristic functions, Hausdorff distance, and compact sets.
\end{abstract}

\maketitle

\section{Introduction}
We prove the convergence of single and double layer potential operators for the harmonic transmission problem over a converging sequence of two-sided extension domains.
In particular, we establish the convergence of Neumann-Poincaré operators, Calderón projectors, and associated Neumann series.
With respect to a certain trace norm, the speed of convergence of the Neumann series can be controlled by the norms of the inner and outer Dirichlet harmonic extension operators, while otherwise being independent of the domain.
As a consequence of our main results, we generalize the notion of the Cauchy integral for a class of extension domains.
This can provide a way to study Hilbert transforms on two-sided extension domains that are not rectifiable. 
Another important consequence is stability for both Lipschitz and non-Lipschitz domains, 
meaning that, in some sense, small domain perturbations lead to small perturbations in the single and double layer potential operators.
Those results have  applications in shape optimization problems involving single and double layer potential operators for both Lipschitz and non-Lipschitz domains, and can be applied to a wider array of problems. 
For convenience, we often concentrate on the case where the sequence of domains is monotone increasing and the norms of the inner and outer extension operators are uniformly bounded, though many results hold under more relaxed conditions.
Using dyadic approximations of arbitrary open sets, we analyze convergence in  terms of characteristic functions, Hausdorff distance, and compact sets. 
Our work aims at extending certain results related to the Radon-Carleman problem \cite[Chapter 6]{mitrea_geometric_2023-2} and the invertibility of layer potential operators~\cite{AuscherEgertNystrom} to non-rectifiable domains. 

It is natural to use irregular shapes such as fractals to model the real world.
As one considers smaller and smaller scales, natural objects display new irregularities~\cite{mandelbrot_how_1967, mandelbrot_fractal_1983}.
For that matter, to fully grasp the richness of their geometries, it appears necessary to adopt a model which allows non-smoothness at all scales.
However, considering such shapes can cause difficulties when studying boundary value problems.
On the one hand, many theoretical results rely on the smoothness of domains to -- among many other things -- obtain regularity estimates for weak solutions~\cite{gilbarg_elliptic_2001}, determine conductivities from boundary measurements~\cite{alessandrini_singular_1990} or establish the boundedness of singular boundary operators~\cite{coifman_two_1989}.
Should one wish to go beyond the smooth case, it would become necessary to consider weaker notions of trace, or of normal derivative, for there is no guarantee that the normal vector can be defined anywhere on the boundary (think of a Minkowski curve for example, which is nowhere smooth).
From a numerical perspective as well -- albeit not the main focus of this paper --, since the computational power of any machine can only allow for a given, finite number of irregularities.
To overcome this regularity gap, one natural approach is to try and approximate the irregular shape using a sequence of more regular geometries.
In the case of self-similar fractals such as the Von Koch snowflake, their definition gives an intrinsic approximation procedure, simply by considering the associated prefractal sequences. However, this procedure depends heavily on the geometry of the fractal shape, and cannot be followed for any non-Lipschitz shape.

In this paper, we generalize the notion of Cauchy integral (which is related to the  Hilbert transform via the classical analysis) to a class of extension domains by means of the convergence of the harmonic layer potential operators along a converging sequence of domains. We also define and study an approximation procedure of arbitrary open sets using dyadic shapes.
The most fundamental definition of the Hilbert transform is a convolution on $\R$ with the kernel $y\mapsto\frac1{\pi y}$ -- or rather the principal value of that convolution, the kernel being singular~\cite[Prologue]{stein_harmonic_1993}.
In the complex plane however, an alternative definition of the Hilbert transform can be given~\cite{khavinson_poincares_2007, muskhelishvili_singular_1977, olver_computing_2011} where the convolution is on a Lipschitz curve instead.
That notion can notably be used in the study of gravitational waves, see~\cite{chapman_exponential_2006, olver_computing_2011}, and allows to express the solution to the Riemann-Hilbert problem consisting in finding a holomorphic function away from the curve satisfying a given jump in trace across the latter~\cite{muskhelishvili_singular_1977}.
If the curve is the boundary of a bounded Lipschitz domain, it coincides with the Cauchy integral~\cite{olver_computing_2011} up to a multiplicative constant.
Mapping properties and boundedness of the latter integral have been widely studied in the case of Lipschitz curves~\cite{calderon_cauchy_1977, coifman_two_1989, coifman_integrale_1982}, as well as on rectifiable Ahlfors regular curves~\cite{eiderman_estimate_2005}.
In addition, and given the connection between the Hilbert transform and the Riemann-Hilbert problem, it appears that the Cauchy integral (and the Hilbert transform in the case of a compact Lipschitz boundary) is connected to the harmonic layer potential operators.
Indeed, it is well known that holomorphic functions are harmonic (see for instance~\cite[Theorem 11.4]{rudin_real_1987}), and the trace jump condition of the problem can be understood in terms of a transmission problem.
The connection between the Cauchy integral and the Neumann-Poincaré operator, which determines the boundary values of the layer potentials, was notably pointed out in~\cite{khavinson_poincares_2007} for $C^2$ curves.
Being defined as singular integrals, the Cauchy integral and Hilbert transform rely heavily on the boundary measure -- in this case, Lebesgue's measure, which is not adapted for non-Lipschitz boundaries. We generalize those notions on non-Lipschitz boundaries and without having to specify a boundary measure, using their connection to the layer potential operators. Then, the convergence of the latter operators allows to prove holomorphism of the generalized Cauchy integral for a class of boundaries beyond Lipschitz.

In the Lipschitz case, the classical definitions of the single and double layer potential operators involve convolutions with either Green's functions or their normal derivatives respectively~\cite{fabes_potential_1978, verchota_layer_1984}.
Solutions to inverse boundary value problems can then be described explicitly using those operators, which constitute an essential tool in the study of inverse problems, in numerical analysis and in areas of spectral theory.
We refer to~\cite{mitrea_geometric_2022-2023, verchota_layer_1984} for the study of the layer potential operators defined in terms of Green's function and to~\cite{barton2017, claret_layer_2025} for a variational approach relying on the well-posedness of transmission problems.
See also the references therein for a more thorough overview of the works on the topic.
The boundary values of the layer potential operators, expressed notably in terms of the Neumann-Poincaré operator, allow to construct the so-called Calder\'on projector~\cite[Theorem 3.1.3]{nedelec_acoustic_2011}, which enables to recover the boundary values of a transmission solution, reconstructed by the layer potential operators from its boundary jumps.
The Neumann-Poincaré operator can also be used to express the solutions to boundary integral equations of the second kind in terms of Neumann series~\cite{perfekt_essential_2017}, allowing notably to recover the jump in trace of a transmission solution with no jump in normal derivative from the knowledge of its interior trace alone. 

In~\cite[Chapter 5]{mazya_boundary_1991}, the boundary layer potential operators on piecewise smooth domains -- however not necessarily Lipschitz -- were studied. Results for layer potentials in the context of Riemannian manifolds were obtained in~\cite{mitrea_boundary_1999}, while layer potentials on half-spaces with boundary data in Besov spaces were studied in \cite{barton_layer_2016}.
The exhaustive investigation carried out in~\cite{mitrea_geometric_2022-2023} of singular integral operators on uniformly rectifiable sets~\cite{david_singular_1991} appeared as the peak of the study of the boundedness of such operators~\cite{calderon_cauchy_1977, coifman_integrale_1982, verchota_layer_1984}.
In $\R^n$, those sets are Ahlfors $(n-1)$-regular closed subsets having `big pieces of Lipschitz images'~\cite[Definitions 5.10.1 and 5.10.2]{mitrea_geometric_2022}, and are characterised by the boundedness of singular integral operators in $L^2$~\cite[Theorem 5.10.2]{mitrea_geometric_2022}.
Another key notion is that of Non-Tangentially-Accessible (NTA) domains~\cite[Section 3]{jerison_boundary_1982}, which arises as essential in the study of the behavior of harmonic functions and measures on the boundary, as in~\cite{david_lipschitz_1990, dahlberg_estimates_1977, jerison_boundary_1982}.
The class of NTA domains includes that of Lipschitz domains, as well as domains with fractal boundaries such as quasidisks~\cite{jerison_boundary_1982, nystrom_integrability_1996}, and domains with boundaries made up of parts of different Hausdorff dimensions. 
As a matter of fact, it is proved in~\cite{azzam_new_2017} that the notions of NTA domains and uniform domains are deeply connected, allowing to unify both approaches.
A short introduction to harmonic measures and NTA domains can be found in~\cite{toro_analysis_2018}.
A class of domains which will be of particular interest in this paper is that of Sobolev extension domains, that is domains on which Sobolev functions can be linearly and continuously extended to the whole space without loss of regularity.
It was shown in~\cite{jones_quasiconformal_1981,rogers_degree-independent_2006} that uniform domains and, more generally, $(\varepsilon,\delta)$-domains are Sobolev extension domains.

Since the layer potential and Neumann-Poincaré operators are defined on the boundary of a domain, their convergence along a sequence of domains must be understood using a notion which goes beyond a single Hilbert space.
A first notion is that of Mosco convergence of closed subspaces, following that of convex sets from~\cite{mosco_convergence_1969} and used notably in~\cite{menegatti_stability_2013} to study the stability of the Neumann problem for the Helmholtz equation, and in~\cite{chandler-wilde_boundary_2021} in the case of a scattering problem by a fractal screen.
This convergence relies on the existence of a larger Hilbert space which contains all others.
Without making that assumption, another method is to use a (measured) Gromov-Hausdorff topology on (measured) pointed metric spaces~\cite{fukaya_collapsing_1987, gromov_metric_2007}.
The general idea is to approximate sets containing the distinguished points in the limit space with sets containing those of the sequence.
This notion has been widely used in a variety of stability problems for optimal transport~\cite{villani_stability_2008, villani_optimal_2009}, and more generally its property of preserving lower curvature bounds for closed Riemannian manifolds~\cite{sturm_geometry_2006} has been used thoroughly, see for instance~\cite{rajala_local_2012, ambrosio_riemannian_2015}.
The spectral distance introduced in~\cite{kasue_spectral_1994, kasue_spectral_1996} allows to quantify how close closed Riemannian manifolds are by comparing the associated heat kernels. It is most useful to study the convergence of the analytic structures~\cite{shioya_convergence_2001}; for instance, the eigenvalues of the Laplacian are continuous with respect to the topology that distance induces.
Here, we use a more general notion of convergence (not relying on the existence of a heat kernel), introduced in~\cite{kuwae_convergence_2003} (see also \cite{Post2012}).
The general idea is to create representatives of elements of the limit space in the approximating sequence.
This framework allows to generalize the notion of Mosco convergence of quadratic forms~\cite{mosco_convergence_1969}, which has many applications in the calculus of variations:
in~\cite{creo_m-convergence_2021} is proved the energy forms associated to the $p$-Laplacian converge along prefractal sequences, see also~\cite{creo_convergence_2018}.
In~\cite{belhadjali_construction_2019}, it is used to prove the convergence of trace energy forms; we also refer to~\cite{creo_non-local_2023}. The notion of convergence we use heavily relies on the way the representatives of the limit space are chosen, for that choice has a significant impact on the subsequent notions of convergence of vectors, operators, and so on (see Figure~\ref{Fig:CV-Hilb}). For that matter, the notion in itself can be deemed rather weak. In this work, we prove the convergence of the layer potential operators in the sense of~\cite{kuwae_convergence_2003} along a sequence of extension domains with uniformly bounded extensions, before strengthening that notion into an $\dot H^1(\R^n)$-convergence (in the sense of~\cite[Chapter 5]{triebel_theory_1983}) of the single layer potentials and of the (harmonic extensions of the) double layer potentials.

The idea of approximating a domain using cubes has been explored in the past.
Although it is not an approximation \textit{per se}, the notion of Whitney covering~\cite{grafakos_classical_2014, whitney_analytic_1934}, using cubes of side comparable to their distance to the boundary, is probably the first to come to mind.
One may also think of~\cite[p. 452]{courant_methods_1989} where a domain is placed on a fine grid to approximate the boundary.
In~\cite{christ_tb_1990} is defined a notion of dyadic cubes in so-called spaces of homogeneous type while in~\cite{david_singular_1991}, those cubes are defined for sets of integer dimension inferior to that of the ambient space.
In~\cite{hofmann_uniform_2016}, dyadic cubes are used to approximate Alhfors-Davies regular sets of dimension $n$ in $\R^{n+1}$.
We also refer to~\cite{david_morceaux_1988, hytonen_non-homogeneous_2012, nazarov_tb-theorem_2003} for similar dyadic constructions.
In the examples mentioned above, the main focus seemed to be the size of the cubes composing the approximation compared to the distance to the boundary, which allows to recover regularity properties of solutions on the approximated domain, see~\cite{bortz_carleson_2023}.
Here, we are more concerned with the modes of convergence of the approximation to the set, and no assumption other than having an open set is made.
In that sense, our approach is closer to that of~\cite{rosler_computing_2024, rosler_computing_2024-1} and their notion of pixelated domains in $\R^2$.
The modes of convergence of sets considered (and holding) with our approximation procedure are those frequently used in shape optimization (see for instance~\cite{claret_existence_2023, hinz_existence_2021, hinz_non-lipschitz_2021}) and described in~\cite{henrot_shape_2018}, namely the convergence in the sense of characteristic functions, of Hausdorff (for the set and its boundary) and of compact sets.

This paper is organised as follows: in Section~\ref{Sec:DyadApprox}, we introduce dyadic approximations of arbitrary open sets. We prove those approximations converge in the sense of characteristic functions, in the sense of Hausdorff (both the domains and the boundaries, on every open ball) and compact sets, Theorem~\ref{Th:DyadApproxCV}.
In Section~\ref{Sec:Framework}, we specify the functional framework in which our study is carried out.
In Subsection~\ref{Subsec:AdDom-Ext-LP}, we define the class of (two-sided) admissible domains, which contains notably Lipschitz, fractal and multifractal domains, on which our study will be carried out, and specify the notions of trace and weak normal derivative which will be used.
We recall the definitions of the layer potential operators for the harmonic transmission problem for two-sided admissible domains from~\cite{claret_layer_2025}.
In Subsection~\ref{Subsec:CV-KS}, we recall the framework of convergence along a sequence of Hilbert spaces from~\cite{kuwae_convergence_2003} and prove several results to gain a better understanding of the notion, and which will be used throughout the rest of the study.
In Section~\ref{Sec:CV-LP}, we prove the convergence of the layer potential operators along a converging sequence of domains: Subsection~\ref{Subsec:CV-Int} focuses on the convergence inside the domain, Subsection~\ref{Subsec:CV-Ext} on the convergence outside, and Subsection~\ref{Subsec:CV-Rn} links those two parts and considers the transmission problem as a whole.
The main result of that section is Theorem~\ref{Th:CV-LP-H1-Rn} which states the convergence of the single and (the harmonic extensions of the) double layer potentials in $\dot H^1(\R^n)$, strengthening Theorem~\ref{Th:CV-LP} which proved the convergence in the framework of moving Hilbert spaces only.
From there, we prove the convergence of the Neumann-Poincaré operators for the harmonic transmission problem and of the associated Calder\'on projectors in Subsection~\ref{Subsec:CV-K-Calderon}, which allows to state the convergence of the Neumann series -- which are uniformly converging sums in this case -- associated to the Neumann-Poincaré operator along a converging sequence of domains in Subsection~\ref{Subsec:CV-Neumann-Series}, Theorem~\ref{Th:CV-Neumann-Series}.
In Section~\ref{Sec:HT-CI}, we express the Cauchy integral on Lipschitz domains in terms of the layer potential operators (Proposition~\ref{Prop:Link-CI-LP}) in Subsection~\ref{Subsec:HT-LP-Lip}, which we use in Subsection~\ref{Subsec:CI-Ext} to define a Cauchy integral on two-sided admissible domains.
We prove that, in a class of domains containing non-Lipschitz and fractal shapes, that Cauchy integral is holomorphic away from the boundary, Theorem~\ref{Th:CI-Holom}.
Then, in Section~\ref{Sec:EquivNorm}, we define a trace norm adapted to the space $\dot H^1(\R^n)$ and prove its equivalence to the norms from Section~\ref{Sec:Framework}. Finally, in Section~\ref{Sec:Monotonicity}, we discuss how to weaken an assumption made throughout Sections~\ref{Sec:CV-LP} which is that of a monotone domain convergence, and give generalizations of Theorem~\ref{Th:CV-LP-H1-Rn} under the assumption of a convergence in the sense of compact sets alone.

Throughout this paper, $n\ge2$ denotes the dimension of the ambient space.
If $x\in \R^n$ and $r>0$, then $B_r(x)$ denotes the open ball of center $x$ and radius $r$.
A set is referred to as a domain if it is nonempty, open and connected.
All domains considered are equal to the interiors of their closures.
If a Hilbert space is decomposed into $H=H_1\oplus H_2$ and $\Lrm_{1,2}:H_{1,2}\to H$ are linear operators, we will denote by $\Lrm_1\oplus \Lrm_2:H\to H$ the linear operator characterised by $(\Lrm_1\oplus \Lrm_2)|_{H_{1,2}}=\Lrm_{1,2}$.
If $u=u_1+u_2\in H$ with $u_{1,2}\in H_{1,2}$, we denote $u=u_1\oplus u_2$.
If $A\subset B$ are Borel sets and $H(B)$ is a set of (equivalence classes of) functions defined on $B$, we will denote by $H(B)|_A$ the set of their restrictions to $A$.
If $\mathrm L$ is an operator defined on a space $H$, we will denote its range by $\mathrm L(H)$.
If $U$ is a subset of $\R^n$, we will denote by $U^c:=\R^n\backslash U$ its complement, $\mathring U$ its interior and $\overline U$ its closure.
If $(\Omega_k)_{k\in\N}$ is a non-decreasing sequence of sets of union $\Omega$, we will denote $\Omega_k\nearrow\Omega$.

\section{Dyadic approximation of open sets}\label{Sec:DyadApprox}

In this section, we give a general method for the approximation of open sets with arbitrary regularity.

Consider the dyadic grid $\Pi_k:=\big\{\pi^k_j\;\mid\;j\in\Z^n\big\}$ on $\R^n$, where
\begin{equation*}
\pi^k_j:=\prod_{m=1}^n\big[2^{-k}j_m,2^{-k}(j_m+1)\big],
\end{equation*}
denoting $j=(j_m)_{m\in\llbracket1,n\rrbracket}\in\Z^n$.
We say that a set $Q\subset\R^n$ is \textit{drawn on} $\Pi_k$ if $Q=\bigcup_{\pi\in\mathcal{Q}}\pi$ for some $\mathcal{Q}\subset\Pi_k$.
Let $\Omega$ be an arbitrary domain of $\R^n$. There exist $k_0\in\N$ and $\pi_0\in \Pi_{k_0}$ such that $\pi_0\subset\Omega$. Define, for all $k\in\N$,
\begin{equation*}
\overline{\Omega}^{\scriptscriptstyle\square}_k:=\max\big\{A\mbox{ drawn on }\Pi_k\;\mid\;\pi_0\subset A\subset\Omega,\; \mathring{A}\mbox{ is connected}\big\},
\end{equation*}
where the maximum is in terms of inclusion. Note that $\overline{\Omega}^{\scriptscriptstyle\square}_k=\varnothing$ for $k<k_0$. Denote its interior by $\Omega^{\scriptscriptstyle\square}_k:=\big(\overline{\Omega}^{\scriptscriptstyle\square}_k\big)^\circ$. This connected open set is referred to as the \textit{dyadic approximation of $\Omega$ of order $k$ rooted in $\pi_0$}. It is a domain for all $k\ge k_0$, see Figure~\ref{Fig:DyadVK} for an example in the case of the Von Koch snowflake.

\begin{figure}[t]
\centering
\hspace{-2.3cm}
\begin{tikzpicture}
\begin{tabular}{c}
\\[-4.148cm]
\includegraphics[height = 6cm]{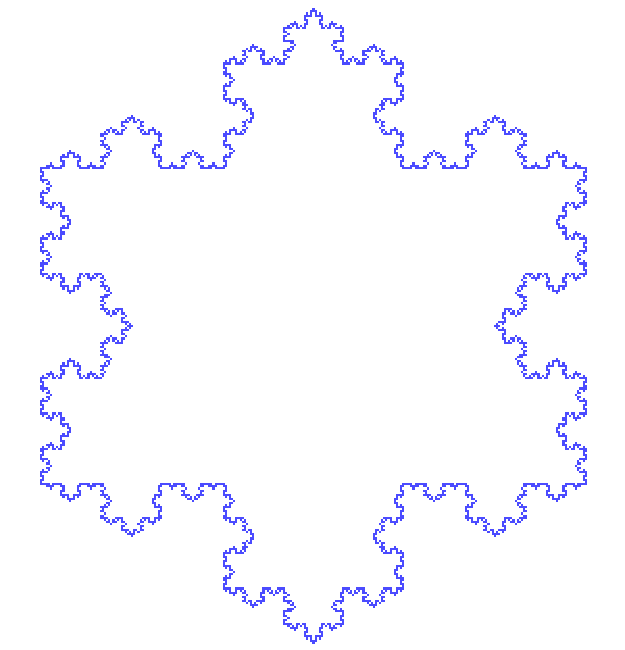}
\end{tabular}
\hspace{-5.08cm}
\draw[step=0.25cm,gray!30!white,very thin] (-1.25,-1.25) grid (5.25,5.25);
\draw (4.5,4.5) node{\color{red} $(\overline{\Omega}^c)^{\scriptscriptstyle\square}_k$};
\draw[thick, densely dotted, red] (1,-0.5) -- (1.25,-0.5) -- (1.25,-0.75) -- (1.5,-0.75) -- (1.5,-1) -- (2.5,-1) -- (2.5,-0.75) -- (2.75,-0.75) -- (2.75,-0.5) -- (3,-0.5) -- (3,0.25) -- (3.25,0.25) -- (3.25,0) -- (4,0) -- (4,0.25) --(4.5,0.25) -- (4.5,0.5) -- (4.75,0.5) -- (4.75,1.75) -- (4,1.75) -- (4,2.25) -- (4.75,2.25) -- (4.75, 2.75) -- (4.75,3.5) -- (4.5,3.5) -- (4.5,3.75) -- (4,3.75) -- (4,4) -- (3.25,4) -- (3.25,3.75) -- (3,3.75) -- (3,4.5) -- (2.75,4.5) -- (2.75,4.75) -- (2.5,4.75) -- (2.5,5) -- (1.5,5) -- (1.5,4.75) -- (1.25,4.75) -- (1.25,4.5) -- (1,4.5) -- (1,4.5) -- (1,3.75) -- (0.75,3.75) -- (0.75,4) -- (0,4) -- (0,3.75) -- (-0.5,3.75) -- (-0.5,3.5) -- (-0.75,3.5) -- (-0.75,2.25) -- (0,2.25) -- (0,1.75) -- (-0.75,1.75) -- (-0.75,0.5) -- (-0.5,0.5) -- (-0.5,0.25) -- (0,0.25) -- (0,0) -- (0.75,0) -- (0.75,0.25) -- (1,0.25) -- cycle;
\draw (2.125,2) node{\color{green!80!black} $\Omega^{\scriptscriptstyle\square}_k$};
\draw[thick, densely dashed, green!80!black] (-0.25,3.25) -- (-0.25,3) -- (0,3) -- (0,2.75) -- (-0.25,2.75) -- (-0.25,2.5) -- (0.25,2.5) -- (0.25,2.25) -- (0.5,2.25) -- (0.5,1.75) -- (0.25,1.75) -- (0.25,1.5) -- (-0.25,1.5) -- (-0.25,1.25) -- (0,1.25) -- (0,1) -- (-0.25,1) -- (-0.25,0.75) -- (0.25,0.75) -- (0.25,0.25) -- (0.5,0.25) -- (0.5,0.75) -- (1.25,0.75) -- (1.25,0.5) -- (1.5,0.5) -- (1.5,-0.25) -- (2.5,-0.25) -- (2.5,0.5) -- (2.75, 0.5) -- (2.75,0.75) -- (3.5,0.75) -- (3.5,0.25) -- (3.75,0.25) -- (3.75,0.75) -- (4.25,0.75) -- (4.25,1) -- (4,1) -- (4,1.25) -- (4.25,1.25) -- (4.25,1.5) -- (3.75,1.5) -- (3.75,1.75) -- (3.5, 1.75) -- (3.5,2.25) -- (3.75,2.25) -- (3.75,2.5) -- (4.25,2.5) -- (4.25,2.75) -- (4,2.75) -- (4,3) -- (4.25,3) -- (4.25,3.25) -- (3.75,3.25) -- (3.75,3.75) -- (3.5,3.75) -- (3.5,3.25) -- (2.75,3.25) -- (2.75,3.5) -- (2.5,3.5) -- (2.5,4.25) -- (1.5,4.25) -- (1.5,3.5) -- (1.25,3.5) -- (1.25,3.25) -- (0.5,3.25) -- (0.5,3.75) -- (0.25,3.75) -- (0.25,3.25) -- cycle;
\draw[blue, ->] (-0.25,-0.5) parabola (0.25,0.125);
\draw (-0.625,-0.5) node{\color{blue} $\bord$};
\draw[gray!70!black] (4,-1) rectangle ++ (0.5,0.5) node[midway]{$\pi_2$};
\draw[gray!70!black] (0.5, 2.5) rectangle ++ (0.5,0.5) node[midway]{$\pi_1$};
\end{tikzpicture}
\caption{Dyadic approximations of a Von Koch snowflake $\Omega$ (lying inside $\bord$, in blue) in $\R^2$ and of its complementary open set $\overline{\Omega}^c$. The dyadic approximation of $\Omega$ rooted at $\pi_1$, $\Omega^{\scriptscriptstyle\square}_k$, lies inside the green dashed line. The dyadic approximation of $\overline{\Omega}^c$ rooted at $\pi_2$, $(\overline{\Omega}^c)^{\scriptscriptstyle\square}_k$, lies outside the red dotted line.}
\label{Fig:DyadVK}
\end{figure}
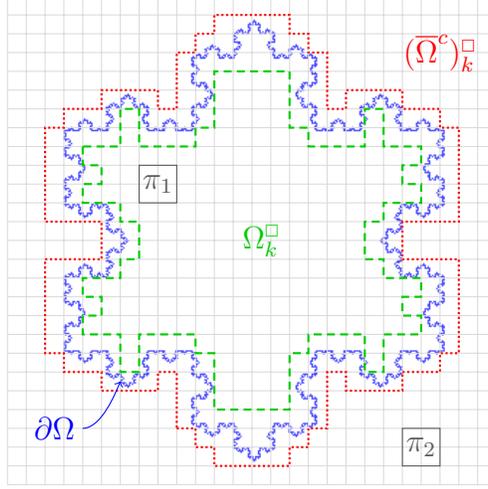

If $\Omega$ is no longer assumed to be connected, denote by $(C_m)_{m\in M}$ its connected components, with $M=\llbracket 0,N\rrbracket$ for some $N\in\N$ or $M=\N$.
Since all $C_m$ are nonempty open sets, there exists a sequence $(k_m)_{m\in M}\in \N^M$, which we may assume to be non-decreasing, and a sequence $(\pi_m)_{m\in M}$ such that
\begin{equation}\label{Eq:Root}
\forall m\in M,\qquad \pi_m\in \Pi_{k_m}\quad\mbox{and}\quad \pi_m\subset C_m.
\end{equation}
Then, for all $k\in\N$,
denote
\begin{equation*}
\overline{\Omega}^{\scriptscriptstyle\square}_k:=\bigcup_{\substack{m\in M\\k_m\le k}}\overline{(C_m)}^{\scriptscriptstyle\square}_k,
\end{equation*}
and define the \textit{dyadic approximation of order $k$ of $\Omega$ rooted in $(\pi_m)_{m\in M}$} as $\Omega^{\scriptscriptstyle\square}_k:=\big(\overline{\Omega}^{\scriptscriptstyle\square}_k\big)^\circ$ once again.
Note that removing the condition $k_m\le k$ in the union would only lead to adding empty sets, and therefore would not have any impact on the definition of $\overline{\Omega}^{\scriptscriptstyle\square}_k$; we specify that condition for clarity. We prove that the dyadic construction presented above defines an approximation procedure for arbitrary open sets, for the dyadic approximations converge as the grid becomes finer.

\begin{theorem}[Convergence of the dyadic approximations]\label{Th:DyadApproxCV}
Let $\Omega$ be an arbitrary open set in $\R^n$. For any $(\pi_m)_{m\in M}$ as in~\eqref{Eq:Root}, denoting by $(\Omega^{\scriptscriptstyle\square}_k)_{k\in\N}$ the dyadic approximations of $\Omega$ rooted in $(\pi_m)$, it holds:
\begin{enumerate}
\item[(i)] $(\Omega^{\scriptscriptstyle\square}_k)_{k\in\N}$ is a non-decreasing sequence of open subsets of $\Omega$;
\item[(ii)] $\Omega^{\scriptscriptstyle\square}_k\to\Omega$ pointwise and in the sense of characteristic functions;
\item[(iii)] $\Omega^{\scriptscriptstyle\square}_k\to\Omega$ in the sense of compact sets.
\end{enumerate}
In addition, if $B$ is an open ball, then
\begin{enumerate}
\item[(iv)] $\Omega^{\scriptscriptstyle\square}_k\cap B\to\Omega\cap B$ in the sense of Hausdorff;
\item[(v)] $\partial(\Omega^{\scriptscriptstyle\square}_k\cap B)\to \partial(\Omega\cap B)$ for the Hausdorff distance on compact sets.
\end{enumerate}
\end{theorem}

\begin{proof}
Point (i) follows from the definition. Let $x\in\Omega$. Then $x\in C_m$ for a unique $m\in M$. Since $C_m$ is connected, there exists a continuous path $\gamma\subset C_m$ joining $x$ and $\pi_m$. Since $C_m$ is open, it holds:
\begin{equation}\label{Eq:PathPositive}
\forall y\in\gamma,\quad d(y,\bord)>0,\qquad\mbox{i.e.,}\qquad\min_{y\in\gamma}\,d(y,\bord)>0,
\end{equation}
by continuity. Therefore, for $k\ge k_m$ large enough, there exists a set $Q$ drawn on $\Pi_k$ and included in $C_m$ joining $x$ and $\pi_m$ in the sense that $x\in \mathring{Q}$ and $Q\cup \pi_m$ is of connected interior.
Consequently, $x\in (C_m)^{\scriptscriptstyle\square}_k$, which implies $x\in \Omega^{\scriptscriptstyle\square}_k$. Point (ii) follows, as well as Point (iv) by~\cite[p. 33]{henrot_shape_2018}.\\
Let $B$ be an open ball. For all $x\in\bord^{\scriptscriptstyle\square}_k$, it holds $d(x,\bord)\le\sqrt n2^{-k}$ (the diagonal of a cube of side $2^{-k}$ in $\R^n$), otherwise $\Omega^{\scriptscriptstyle\square}_k$ would not be maximal. Consequently,
\begin{equation*}
\rho\big(\partial(\Omega^{\scriptscriptstyle\square}_k\cap B),\partial(\Omega\cap B)\big):=\max_{x\in\partial(\Omega^{\scriptscriptstyle\square}_k\cap B)} d\big(x,\partial(\Omega\cap B)\big) \le \sqrt n2^{-k},
\end{equation*}
where we have used the same notations as in~\cite[Definition 2.2.7]{henrot_shape_2018}. On the other hand, for all $y\in\bord$, the sequence $\big(d\big(y,\partial(\Omega^{\scriptscriptstyle\square}_k\cap B)\big)\big)_{k\in\N}$ is non-increasing and must go to $0$ by Point (i). Therefore, $\big(\rho\big(\partial(\Omega\cap B),\partial(\Omega^{\scriptscriptstyle\square}_k\cap B)\big)\big)_{k\in\N}$ is also non-increasing. Since it is non-negative and finite, it converges to some $\ell\ge 0$. Hence,
\begin{equation*}
\forall k\in\N,\; \exists y_k\in \partial(\Omega\cap B),\quad d\big(y_k,\partial(\Omega^{\scriptscriptstyle\square}_k\cap B)\big)\ge \ell.
\end{equation*}
By compactness of $\partial(\Omega\cap B)$, we may assume there exists $y_\infty\in \partial(\Omega\cap B)$ such that $y_k\to y_\infty$. Then, for all $k\in\N$, it holds
\begin{equation*}
d\big(y_\infty,\partial(\Omega^{\scriptscriptstyle\square}_k\cap B)\big)\ge d\big(y_k,\partial(\Omega^{\scriptscriptstyle\square}_k\cap B)\big) - d(y_\infty,y_k),
\end{equation*}
hence, by taking the limit in $k$, $0\ge\ell$ and Point (v) holds.
For $\delta\ge \varepsilon>0$, consider the compact set
\begin{equation*}
K^\varepsilon_\delta:=\big\{x\in\Omega\cap B_\delta(0)\;\mid\; d\big(x,\partial(\Omega\cap B_\delta(0)\big)\ge \varepsilon\big\}.
\end{equation*}
If $k$ is large enough, then
\begin{equation*}
d^H\big(\partial(\Omega\cap B_\delta(0)),\partial(\Omega^{\scriptscriptstyle\square}_k\cap B_\delta(0))\big)<\varepsilon,
\end{equation*}
with $d^H$ the Hausdorff distance on compact sets, so that
\begin{equation*}
K^\varepsilon_\delta\subset (\Omega^{\scriptscriptstyle\square}_k\cap B_\delta(0))\subset \Omega^{\scriptscriptstyle\square}_k.
\end{equation*}
Since any compact set included in $\Omega$ is also included in $K^\varepsilon_\delta$ for some $\delta\ge\varepsilon>0$, it is also included in $\Omega^{\scriptscriptstyle\square}_k$ for $k$ large enough. Since every $\Omega^{\scriptscriptstyle\square}_k$ is a subset of $\Omega$, Point (iii) follows.
\end{proof}

Theorem~\ref{Th:DyadApproxCV} states that any open set, no matter how irregular, can be approximated by a non-decreasing sequence of dyadic shapes.
This idea is at the core of Section~\ref{Sec:CV-LP}, in which we consider convergence along a non-decreasing sequence of domains.

\section{Functional framework}\label{Sec:Framework}

In this section, we introduce the functional framework in which the rest of the study will be carried out, both in terms of boundary value problems and associated operators, and of convergence along sequences of Hilbert spaces.

\subsection{Admissible domains, extensions and layer potentials}\label{Subsec:AdDom-Ext-LP}

Let us begin with defining the class of domains considered in this paper. This class relies on the notion of homogeneous space of Sobolev functions $\dot H^1$, which is the quotient space of the usual Sobolev space $H^1$ by the equality modulo locally constant functions, see~\cite{deny_espaces_1954},~\cite[Sections 1.1.2 and 1.1.13]{mazya_sobolev_2011} and~\cite[Section II.6]{galdi_introduction_2011}. If $\Omega$ is a domain of $\R^n$, then $\dot H^1(\Omega)$ is a Hilbert space for the norm
\begin{equation*}
\lVert\cdot\rVert^2_{\dot H^1(\Omega)}:=\int_\Omega\lvert\nabla\cdot\rvert^2\,\dx.
\end{equation*}

\begin{definition}[Admissible domain]\label{Def:AdmissibleDomain}
A domain $\Omega$ of $\R^n$ is said to be an admissible domain if
\begin{enumerate}
\item[(i)] there exists a bounded linear extension operator $\mathrm{Ext}_\Omega:\dot H^1(\Omega)\to \dot H^1(\R^n)$;
\item[(ii)] $\bord$ is compact and has positive capacity.
\end{enumerate}
It is said to be two-sided admissible if, in addition,
\begin{enumerate}
\item[(iii)] $\overline{\Omega}^c$ is an admissible domain;
\item[(iv)] $\bord$ is a null set with respect to Lebesgue's measure on $\R^n$.
\end{enumerate}
\end{definition}

A domain satisfying condition (i) is called a $\dot H^1$-extension domain~\cite{claret_layer_2025} for it is, in the context of homogeneous spaces, the counterpart of the notion of $H^1$-extension domain~\cite{hajlasz_sobolev_2008, jones_quasiconformal_1981, rogers_degree-independent_2006}.
The notion of capacity intervening in (ii) refers to the capacity with respect to $H^1(\R^n)$, see~\cite[Section 2.1]{fukushima_dirichlet_2010},~\cite[Section 7.2]{mazya_boundary_1991} and~\cite[Section 2]{biegert_traces_2009}; the notions `quasi-everywhere' (q.e.) and `quasi-continuous' will be understood in the same sense. Any Lipschitz domain is an $\dot H^1$-extension domain; more generally, any $(\varepsilon,\infty)$-domain is an $\dot H^1$-extension domain~\cite[Theorem 2]{jones_quasiconformal_1981}. If $\Omega$ is a $(\varepsilon,\infty)$-domain and either $\Omega$ or $\overline{\Omega}^c$ is bounded, then $\Omega$ is two-sided admissible. In dimension $n\ge 3$, the embedding~\cite[Theorem 1.43]{bahouri_fourier_2011} allows to adapt~\cite[Theorems 2 and 5]{hajlasz_sobolev_2008} and prove that any $\dot H^1$-extension domain $\Omega$ is an $n$-set:
\begin{equation*}
\exists c>0,\;\forall x\in\Omega,\;\forall r\in]0,1],\quad \lambda^{(n)}(\Omega\cap B_r(x))\ge cr^n,
\end{equation*}
where $\lambda^{(n)}$ is Lebesgue's measure on $\R^n$. Consequently, a domain with an outward cusp cannot satisfy the $\dot H^1$-extension property, see Figure~\ref{Fig:AdmissibleDomains}.

\begin{figure}[t]
\centering
\includegraphics[height = 5.7 cm, width = 6 cm]{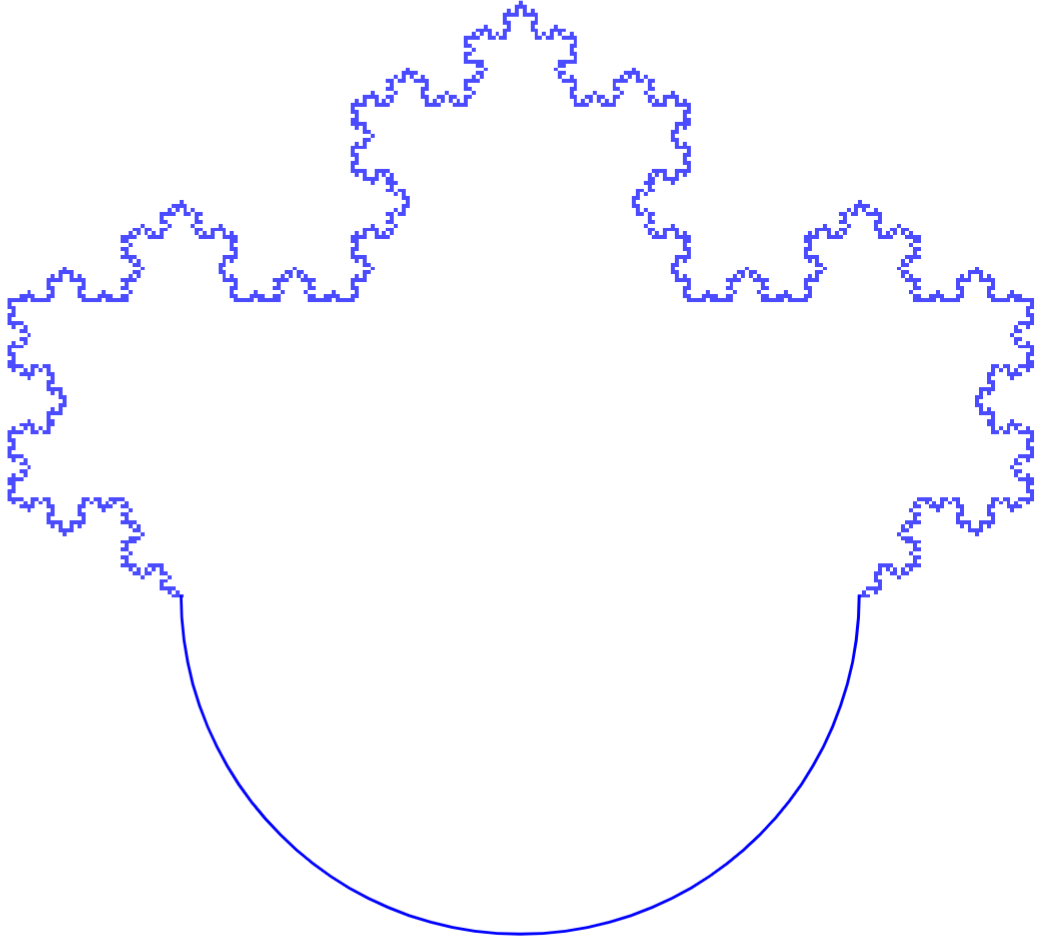}
\hfill
\includegraphics[height = 5.7 cm, width = 5.5 cm]{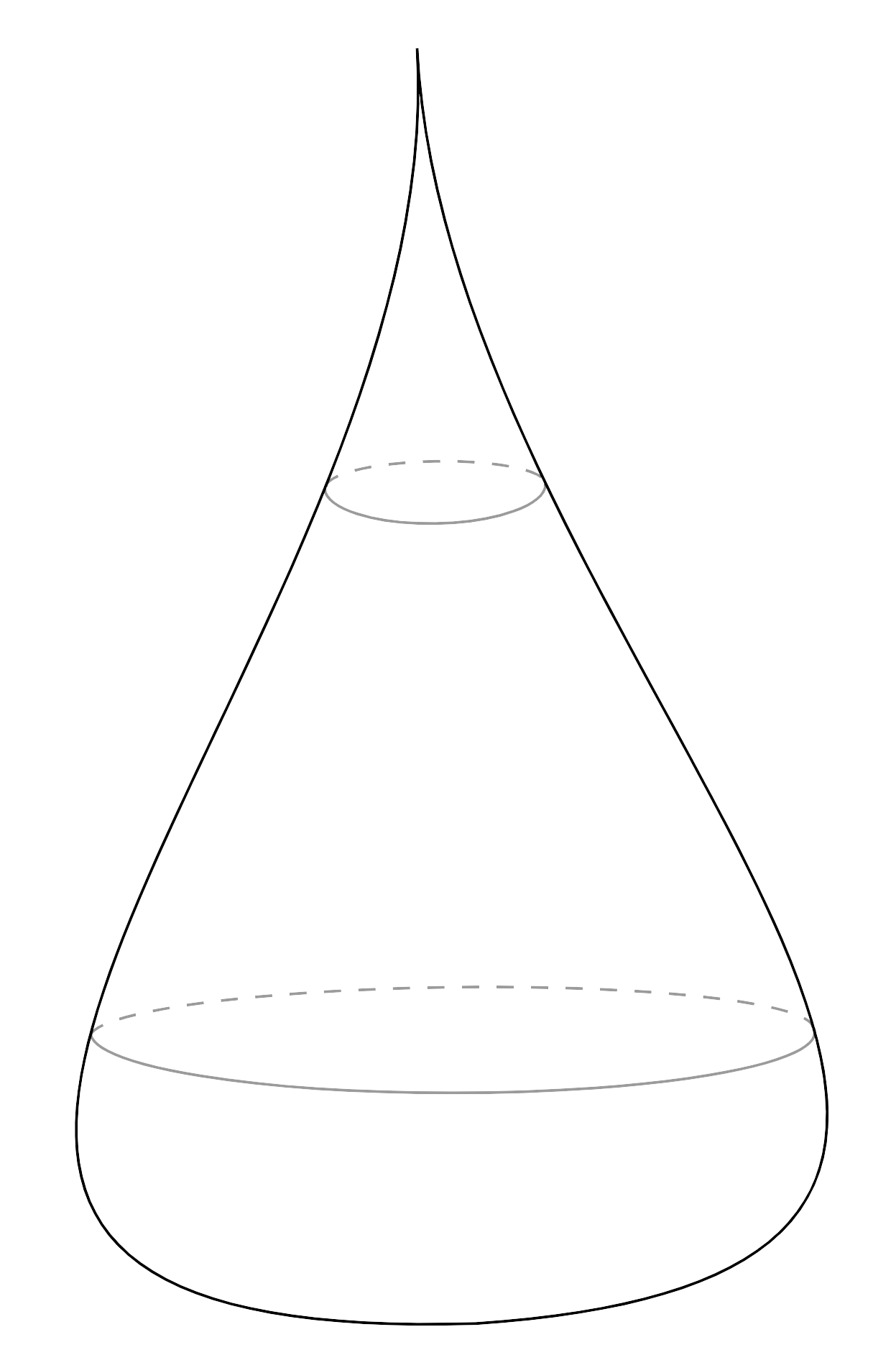}

\caption{On the left, a two-sided admissible domain of $\R^2$. The top part of its boundary consists of a Von Koch curve of Hausdorff dimension $\frac{\ln 4}{\ln 3}$, while the bottom part is of Hausdorff dimension $1$. On the right, a domain of $\R^3$ which is not an extension domain for it has an outward cusp.}
\label{Fig:AdmissibleDomains}
\end{figure}

If $u\in \dot H^1(\R^n)$, then there exists a quasi-continuous representative $\tilde u$ of a representative of $u$ modulo constants, see~\cite[Theorems 2.2.12, 2.2.13 and 2.3.4]{chen_symmetric_2012} and~\cite[Appendix A]{claret_layer_2025}.
If $\Omega$ is an admissible domain, then $[\tilde u|_{\bord}]_{\Hpo}$, defined as the equivalence class modulo constant of the q.e. equivalence class of the restriction $\tilde u|_{\bord}$, does not depend on the choice of $\tilde u$, only on $u$. We denote by $\Hpo$ the vectors space of all $[\tilde u|_{\bord}]_{\Hpo}$ for $u\in\dot H^1(\R^n)$. In this setting, we define the trace operators as in~\cite[Proposition 2.2]{claret_layer_2025}, that is as a variant of~\cite[Theorem 6.1, Remark 6.2, Corollary 6.3]{biegert_traces_2009}.

\begin{definition}[Trace operators]
Let $\Omega$ be a domain of $\R^n$. If $\Omega$ is admissible, then we define the interior trace operator on $\bord$ as
\begin{align*}
\Trd^{\bord}_{\mathrm{i}}:\dot H^1(\Omega)&\longrightarrow\Hpo\\
u&\longmapsto [(\mathrm{Ext}_\Omega u)^\sim|_{\bord}]_{\Hpo},
\end{align*}
where $(\mathrm{Ext}_\Omega u)^\sim$ is a quasi-continuous representative of a representative modulo constant of $\mathrm{Ext}_\Omega u\in \dot H^1(\R^n)$.\\
If $\overline{\Omega}^c$ is admissible, then we define the exterior trace operator on $\bord$ as
\begin{align*}
\Trd^{\bord}_{\mathrm{e}}:\dot H^1(\overline{\Omega}^c)&\longrightarrow\Hpo\\
u&\longmapsto [(\mathrm{Ext}_{\overline{\Omega}^c} u)^\sim|_{\bord}]_{\Hpo}.
\end{align*}
If $\Omega$ is two-sided admissible, then we define the jump in trace across $\bord$ as 
\begin{equation*}
\llbracket \Trd^{\bord} u\rrbracket:=\Trd^{\bord}_{\mathrm{i}}u-\Trd^{\bord}_{\mathrm{e}}u,\quad u\in \dot H^1(\nbord).
\end{equation*}
\end{definition}

If $\Omega$ is an admissible domain, consider the space of harmonic functions on $\Omega$:
\begin{equation}\label{Eq:Int-Harmo-Space}
\dot{V}_0(\Omega):=\big\{v\in \dot H^1(\Omega)\;\mid\;\Delta v=0\mbox{ weakly on }\Omega\big\},
\end{equation}
and denote the restriction of the operator $\Trd^{\bord}_{\mathrm{i}}$ to $\dot V_0(\Omega)$ as $\trd^{\bord}_{\mathrm{i}}:\dot V_0(\Omega)\to\Hpo$. Similarly, if $\overline{\Omega}^c$ is admissible, denote the restriction of the operator $\Trd^{\bord}_{\mathrm{e}}$ to the space of harmonic functions on $\overline{\Omega}^c$, $\dot V_0(\overline{\Omega}^c)$, as $\trd^{\bord}_{\mathrm{e}}:\dot V_0(\overline{\Omega}^c)\to\Hpo$.
The following theorem, stated in this form in~\cite[Theorem 2.10]{claret_layer_2025}, synthesizes results from~\cite[Example 2.3.2 and Theorem 2.3.3]{fukushima_dirichlet_2010} and~\cite[Theorems 2.2.12 and 2.2.13]{chen_symmetric_2012}.

\begin{theorem}[Trace theorem]\label{Th:Trace}
Let $\Omega$ be an admissible domain. Then, the following assertions hold:
\begin{enumerate}
\item[(i)] the kernel $\operatorname{Ker} \Trd^{\bord}_{\mathrm{i}}$ is the orthogonal complement $(\dot{V}_0(\Omega))^\bot$ of $\dot{V}_0(\Omega)$ in $\dot{H}^1(\Omega)$, denoted by $\dot H^1_0(\Omega)$;
\item[(ii)] the space $\Hpo$ is a Hilbert space when endowed with the norm
\begin{equation*}
\lVert f\rVert_{\Trd^{\bord}_{\mathrm{i}}}:=\min\big\{\lVert v \rVert_{\dot{H}^1(\Omega)}\mid\ v\in \dot{H}^1(\Omega) \mbox{ and } \Trd^{\bord}_{\mathrm{i}}v=f\big\},
\end{equation*}
where the minimum is uniquely achieved, for $v\in \dot V_0(\Omega)$ with $\trd^{\bord}_{\mathrm{i}} v=f$;
\item[(iii)] with respect to $\lVert \cdot \rVert_{\Trd^{\bord}_{\mathrm{i}}}$, $\Trd^{\bord}_{\mathrm{i}}$ is bounded with operator norm one. 
Its restriction $\trd^{\bord}_{\mathrm{i}}:\dot{V}_0(\Omega)\to \Hpo$ to $\dot{V}_0(\Omega)$ is an isometry and onto.
\end{enumerate}
If $\overline{\Omega}^c$ is admissible, then counterparts of $(i)$, $(ii)$ and $(iii)$ hold for $\Trd^{\bord}_{\mathrm{e}}$ instead of $\Trd^{\bord}_{\mathrm{i}}$, replacing $\Omega$ and $\trd^{\bord}_{\mathrm{i}}$ with $\overline{\Omega}^c$  and $\trd^{\bord}_{\mathrm{e}}$ respectively. In particular, the norm
\begin{equation*}
\lVert f\rVert_{\Trd^{\bord}_{\mathrm{e}}}:=\min\big\{\lVert v \rVert_{\dot{H}^1(\overline{\Omega}^c)}\mid\ v\in \dot{H}^1(\overline{\Omega}^c) \mbox{ and } \Trd^{\bord}_{\mathrm{e}}\:v=f\big\}
\end{equation*}
is a Hilbert space norm on $\Hpo$.
\end{theorem}

In what follows, we will endow $\Hpo$ with the norm $\lVert\cdot\rVert_{\Trd^{\bord}_{\mathrm{i}}}$, which will be denoted by $\lVert\cdot\rVert_{\Hpo}:=\lVert\cdot\rVert_{\Trd^{\bord}_{\mathrm{i}}}$, so that
\begin{equation}\label{Eq:B-norm}
\lVert f\rVert_{\Hpo}=\min\big\{\lVert v \rVert_{\dot{H}^1(\Omega)}\mid\ v\in \dot{H}^1(\Omega) \mbox{ and } \Trd^{\bord}_{\mathrm{i}}v=f\big\}.
\end{equation}
If $\Omega$ is an admissible domain, we define the Dirichlet harmonic extension from $\bord$ to $\Omega$ as $\DOd_\Omega=(\trd^{\bord}_{\mathrm{i}})^{-1}:\Hpo\to \dot V_0(\Omega)$, as well as the Dirichlet harmonic extension operator from $\Omega$ to $\R^n$ (harmonic on $\overline{\Omega}^c$) as
\begin{equation}\label{Eq:Dir-Int-Ext}
\begin{aligned}
\dot{\mathrm{E}}_\Omega: \dot H^1(\Omega)&\longrightarrow \dot H^1(\R^n)\\
u&\longmapsto u\oplus v,
\end{aligned}
\end{equation}
decomposed along $\dot H^1(\Omega)\oplus \dot H^1(\overline{\Omega}^c)$, where $v$ denotes the orthogonal projection of $(\mathrm{Ext}_\Omega u)|_{\overline{\Omega}^c}$ onto the space of harmonic functions on $\overline{\Omega}^c$, $\dot V_0(\overline{\Omega}^c)$, defined as~\eqref{Eq:Int-Harmo-Space}. Unless stated otherwise, the norm $\lVert\dot{\mathrm{E}}_{\Omega}\rVert$ of this operator will always be understood as its operator norm in the space $\mathcal{L}(\dot H^1(\Omega),\dot H^1(\R^n))$.

Similarly, if $\overline{\Omega}^c$ is admissible, we define the Dirichlet harmonic extension from $\bord$ to $\overline{\Omega}^c$ as $\DOd_{\overline{\Omega}^c}=(\trd^{\bord}_{\mathrm{e}})^{-1}:\Hpo\to \dot V_0(\overline{\Omega}^c)$, as well as the Dirichlet harmonic extension operator from $\overline{\Omega}^c$ to $\R^n$ (harmonic on $\Omega$) as
\begin{equation}\label{Eq:Dir-Ext-Ext}
\begin{aligned}
\dot{\mathrm{E}}_{\overline{\Omega}^c}: \dot H^1(\overline{\Omega}^c)&\longrightarrow \dot H^1(\R^n)\\
u&\longmapsto v\oplus u,
\end{aligned}
\end{equation}
where $v$ denotes the orthogonal projection of $(\mathrm{Ext}_{\overline{\Omega}^c} u)|_\Omega$ onto $\dot V_0(\Omega)$. Unless stated otherwise, the norm $\lVert\dot{\mathrm{E}}_{\overline{\Omega}^c}\rVert$ of this operator will always be understood as its the operator norm in the space $\mathcal{L}(\dot H^1(\overline{\Omega}^c),\dot H^1(\R^n))$.

\begin{proposition}\label{Prop:Norm-B-Equiv}
If $\Omega$ is a two-sided admissible domain, then both norms on $\Hpo$ defined in Theorem~\ref{Th:Trace} are equivalent and it holds
\begin{equation}\label{Eq:TrNormEquiv}
\lVert\cdot\rVert_{\Trd^{\bord}_{\mathrm{i}}}\le \sqrt{\lVert\dot{\mathrm{E}}_{\overline{\Omega}^c}\rVert^2-1}\;\lVert\cdot\rVert_{\Trd^{\bord}_{\mathrm{e}}}\qquad \mbox{and}\qquad \lVert\cdot\rVert_{\Trd^{\bord}_{\mathrm{e}}}\le  \sqrt{\lVert\dot{\mathrm{E}}_{\Omega}\rVert^2-1}\;\lVert\cdot\rVert_{\Trd^{\bord}_{\mathrm{i}}},
\end{equation}
with optimal constants. 
\end{proposition}

\begin{proof}
Let $f\in \Hpo$ be non-null.
Let $u\in \dot V_0(\Omega)\oplus \dot V_0(\overline{\Omega}^c)$ be such that $\trd^{\bord}_{\mathrm{i}}(u|_{\Omega})=\trd^{\bord}_{\mathrm{e}}(u|_{\overline{\Omega}^c})=f$. Then, $u=\dot{\mathrm{E}}_{\overline{\Omega}^c}(u|_{\overline{\Omega}^c})$, so that
\begin{equation*}
\frac{\lVert f\rVert_{\Trd^{\bord}_{\mathrm{i}}}^2}{\lVert f\rVert_{\Trd^{\bord}_{\mathrm{e}}}^2}=\frac{\lVert u|_\Omega\rVert_{\dot H^1(\Omega)}^2}{\lVert u|_{\overline{\Omega}^c}\rVert_{\dot H^1(\overline{\Omega}^c)}^2}=\frac{\lVert u\rVert_{\dot H^1(\R^n)}^2}{\lVert u|_{\overline{\Omega}^c}\rVert_{\dot H^1(\overline{\Omega}^c)}^2}-1\le \lVert\dot{\mathrm{E}}_{\overline{\Omega}^c}\rVert^2-1.
\end{equation*}
Regarding the optimality of the constant, notice that for all $v\in \dot H^1(\overline{\Omega}^c)$ with $\Trd^{\bord}_{\mathrm{e}}v=f$, it holds $\dot{\mathrm{E}}_{\overline{\Omega}^c}v=u|_\Omega\oplus v$. Since, $\lVert u|_{\overline{\Omega}^c}\rVert_{\dot H^1(\overline{\Omega}^c)}\le\lVert v \rVert_{\dot H^1(\overline{\Omega}^c)}$, we can deduce that restricting $\dot{\mathrm{E}}_{\overline{\Omega}^c}$ to $\dot V_0(\overline{\Omega}^c)$ has no impact on its norm, and the result follows.
The second inequality is similar.
\end{proof}

Finally, if either $\Omega$ or $\overline{\Omega}^c$ is admissible, we define the Dirichlet harmonic extension from $\bord$ to $\R^n$ (harmonic on $\nbord$) $\dot{\mathrm{H}}_{\bord}:\Hpo\to \dot H^1(\R^n)$ as 
\begin{equation}\label{Eq:D-Ext}
\dot{\mathrm H}_{\bord}:=
\begin{cases}
\dot{\mathrm{E}}_\Omega\circ\DOd_\Omega &\mbox{if }\Omega\mbox{ is admissible,}\\
\dot{\mathrm{E}}_{\overline{\Omega}^c}\circ\DOd_{\overline{\Omega}^c} &\mbox{if }\overline{\Omega}^c\mbox{ is admissible}.
\end{cases}
\end{equation}
If $\Omega$ is a two-sided admissible domain, then both definitions coincide.

Since the class of admissible domains contains non-smooth domains, the normal derivatives at the boundary cannot always be defined in the strong sense, using the normal vector. For that matter, we use a weak notion of normal derivative as in~\cite{lancia_transmission_2002} and~\cite[Subsection 2.5]{claret_layer_2025}, as an element of the dual space $\Hmo:=\mathcal{L}(\Hpo,\R)$.

\begin{definition}[Weak normal derivatives]
Let $\Omega$ be a domain of $\R^n$. If $\Omega$ is admissible, then we define the weak interior normal derivative on $\bord$ of $u\in \dot H^1(\Omega)$ with $\Delta u\in L^2(\Omega)$ as the unique $\phi\in\Hmo$ such that
\begin{equation*}
\forall v\in \dot H^1(\Omega),\quad \langle \phi,\Trd^{\bord}_{\mathrm{i}}v\rangle_{\Hmo,\Hpo}=\int_\Omega(\Delta u)v\,\dx+\int_\Omega \nabla u\cdot\nabla v\,\dx.
\end{equation*}
We denote $\ddno u{\mathrm{i}}|_{\bord}:=\phi$.\\
If $\overline{\Omega}^c$ is admissible, then we define the weak exterior normal derivative on $\bord$ of $u\in \dot H^1(\overline{\Omega}^c)$ with $\Delta u\in L^2(\overline{\Omega}^c)$ as the unique $\psi\in\Hmo$ such that
\begin{equation*}
\forall v\in \dot H^1(\overline{\Omega}^c),\quad \langle \psi,\Trd^{\bord}_{\mathrm{e}}v\rangle_{\Hmo,\Hpo}=-\int_{\overline{\Omega}^c}(\Delta u)v\,\dx-\int_{\overline{\Omega}^c} \nabla u\cdot\nabla v\,\dx.
\end{equation*}
We denote $\ddno u{\mathrm{e}}|_{\bord}:=\psi$.\\
If $\Omega$ is two-sided admissible, then we define the jump in normal derivative across $\bord$ of $u\in\dot H^1(\nbord)$ with $\Delta u\in L^2(\nbord)$ as
\begin{equation}\label{Eq:ddn-jump}
\bigg\llbracket\ddno u{}\Big|_{\bord}\bigg\rrbracket:=\ddno u{\mathrm{i}}\Big|_{\bord} -\ddno u{\mathrm{e}}\Big|_{\bord}.
\end{equation}
\end{definition}

We refer to~\cite[Subsection 2.5]{claret_layer_2025} for properties of the normal derivation operators. We simply recall the following result on isometric properties of those operators~\cite[Corollary 2.17]{claret_layer_2025}.

\begin{proposition}[\cite{claret_layer_2025}]\label{Prop:ddn-isom}
Let $\Omega$ be a domain of $\R^n$. If $\Omega$ is admissible, then $\ddno{}{\mathrm{i}}|_{\bord}:\dot V_0(\Omega)\to\Hmo$ is an isometry and onto when $\Hmo$ is endowed with the subordinate norm to $\lVert\cdot\rVert_{\Hpo}$ from~\eqref{Eq:B-norm}, denoted by $\lVert\cdot\rVert_{\Hmo}$. If $\overline{\Omega}^c$ is admissible, then $\ddno{}{\mathrm{e}}|_{\bord}:\dot V_0(\overline{\Omega}^c)\to\Hmo$ is an isometry and onto when $\Hmo$ is endowed with the subordinate norm to $\lVert\cdot\rVert_{\Trd^{\bord}_{\mathrm{e}}}$.
\end{proposition}

Note that the norms on $\Hmo$ mentioned in Proposition~\ref{Prop:ddn-isom} are equivalent, by Proposition~\ref{Prop:Norm-B-Equiv}. In what follows and unless stated otherwise, $\Hmo$ will be endowed with the norm $\lVert\cdot\rVert_{\Hmo}$.

\begin{definition}[Layer potential operators~\cite{claret_layer_2025}]\label{Def:LP}
Let $\Omega$ be a two-sided admissible domain. The \emph{single} and \emph{double layer potential operators} associated to the harmonic transmission problem
\begin{equation}\label{Eq:Tr-Prob}
\begin{cases}
-\Delta u=0 &\mbox{on }\nbord,\\
\llbracket\Trd^{\bord}u\rrbracket=f,\\
\llbracket\ddno u{}|_{\bord}\rrbracket=g,
\end{cases}
\end{equation}
are defined as the unique pair of operators
\begin{equation*}
(\Sd_{\bord},\Dd_{\bord}):\Hmo\times\Hpo\longrightarrow \dot H^1(\nbord)\times\dot H^1(\nbord)
\end{equation*} 
such that, for all $f\in\Hpo$ and $g\in\Hmo$, the unique weak solution $u\in\dot H^1(\nbord)$ to~\eqref{Eq:Tr-Prob} be given by $u=\Sd_{\bord}g-\Dd_{\bord}f$.
\end{definition}

Once again, we refer to~\cite{claret_layer_2025} regarding properties of the layer potential operators associated to~\eqref{Eq:Tr-Prob} on a two-sided admissible domain. As a complement, we highlight the fact that those operators are bounded independently from the geometry of the two-sided admissible domain.

\begin{proposition}\label{Prop:LP-Bounds}
Let $\Omega$ be a two-sided admissible domain. Then, the layer potential operators associated to~\eqref{Eq:Tr-Prob} have operator norms non-superior to $1$ with respect to the spaces specified in Definition~\ref{Def:LP}.
\end{proposition}

\begin{proof}
By Green's formula, it holds for $g\in\Hmo$,
\begin{equation*}
\forall v\in\dot H^1(\R^n),\quad\langle\Sd_{\bord}g,v\rangle_{\dot H^1(\R^n)}=\langle g,\Trd^{\bord}_{\mathrm{i}} v\rangle_{\Hmo,\Hpo}\le \lVert g \rVert_{\Hmo}\lVert v \rVert_{\dot H^1(\R^n)},
\end{equation*}
by Theorem~\ref{Th:Trace}, Point (iii). Similarly, for all $f\in\Hpo$,
\begin{equation*}
\lVert\Dd_{\bord}f\rVert_{\dot H^1(\nbord)}^2=\bigg|\bigg\langle\ddno{}{\mathrm{i}}\Big|_{\bord}\Dd_{\bord}f,f\bigg\rangle_{\Hmo,\Hpo}\bigg|\le \lVert\Dd_{\bord}f\rVert_{\dot H^1(\nbord)}\lVert f\rVert_{\Hpo},
\end{equation*}
by~\cite[Corollary 2.17]{claret_layer_2025} which induces a variant of Theorem~\ref{Th:Trace}, Point (iii) for normal derivatives. The result follows.
\end{proof}

\subsection{Convergence along sequences of Hilbert spaces}\label{Subsec:CV-KS}

In this part, we recall the definition of convergence of, and along a sequence of Hilbert spaces introduced in~\cite[Subsections 2.2 and 2.3]{kuwae_convergence_2003}. We also prove some results to gain a better understanding of that notion, and which will be used throughout this paper.

\begin{definition}[Convergence of Hilbert spaces~\cite{kuwae_convergence_2003}]\label{Def:CV-Hilb}
A sequence $(H_k)_{k\in\N}$ of Hilbert spaces is said to converge towards a Hilbert space $H$ through $(\mathcal{C},(\mathcal{T}_k)_{k\in\N})$, where $\mathcal C$ is a dense subspace of $H$ and $(\mathcal{T}_k:\mathcal{C}\to H_k)_{k\in\N}$ is a sequence of linear maps, if it holds:
\begin{equation}\label{CV}
\forall f\in \mathcal{C},\quad\lim_{k\to+\infty}\lVert\mathcal{T}_k f\rVert_{H_k}=\lVert f\rVert_H.
\end{equation}
In that case, we will denote: $H_k\xrightarrow[k\to\infty]{(\mathcal{T}_k)}H$.
\end{definition}

It is essential to point out that the main piece of information given by the convergence of spaces from Definition~\ref{Def:CV-Hilb} is not that the sequence $(H_k)_{k\in\N}$ converges to $H$, but the way that convergence occurs, i.e., the pair $(\mathcal{C},(\mathcal{T}_k)_{k\in\N})$. Indeed, if $(H_k)$ is any sequence of separable Hilbert spaces and $H$ is any separable Hilbert space, then $(H_k)$ converges to $H$ through $(H,(\mathcal{T}_k)_{k\in\N})$ choosing $\mathcal{T}_k:H\to H_k$ isometric, even if the spaces in question are `unrelated'. However, making such a choice for $(\mathcal{T}_k)$ will lead (in general) to distasteful results when it comes to the convergence of vectors, operators and so on.

\begin{definition}[Strong convergence of vectors~\cite{kuwae_convergence_2003}]\label{Def:CV-Vect}
Let $(H_k)_{k\in\N}$ be a sequence of Hilbert spaces and $H$ be a Hilbert space such that $H_k\to H$ through $(\mathcal{C},(\mathcal{T}_k)_{k\in\N})$. A sequence $(u_k)_{k\in\N}$ with $u_k\in H_k$ -- denoted by $(u_k\in H_k)_{k\in\N}$ -- is said to converge (strongly) towards an element $u\in H$ if there exists a sequence $(v_m)_{m\in\N}\in \mathcal{C}^\N$ with $v_m\to u$ in $H$, such that:
\begin{equation*}
\lim_{m\to+\infty}\bigg(\limsup_{k\to+\infty}\lVert\mathcal{T}_kv_m-u_k\rVert_{H_k}\bigg)=0.
\end{equation*}
In that case, we will denote: $(u_k\in H_k)_{k\in\N}\xrightarrow[]{(\mathcal{T}_k)}u\in H$.
\end{definition}

The general definition of convergence of Hilbert spaces from Definition~\ref{Def:CV-Hilb} involves a dense subspace of the limit space. In this paper, the dense subspace will always be the limit space itself. In that case, the notion of convergence of vectors from Definition~\ref{Def:CV-Vect} can be simplified.

\begin{lemma}\label{Lem:CV-KS}
Let $(H_k)_{k\in\N}$ be a sequence of Hilbert spaces and $H$ be a Hilbert space such that $(H_k)_{k\in\N}\xrightarrow[]{}H$ through $(H,(\mathcal{T}_k)_{k\in\N})$.
Then, for all $(u_k\in H_k)_{k\in\N}$ and $u\in H$, it holds:
\begin{equation*}
u_k\xrightarrow[k\to\infty]{}u\quad\mbox{through }(H,(\mathcal{T}_k)_{k\in\N}) \qquad\Longleftrightarrow\qquad \lVert\mathcal{T}_ku-u_k\rVert_{H_k}\xrightarrow[k\to\infty]{}0.
\end{equation*}
\end{lemma}

\begin{proof}
Assume that $(u_k\in H_k)_{k\in\N}\xrightarrow[]{(\mathcal{T}_k)} u\in H$. There exists $(v_m)_{m\in\N}\in H^\N$ such that:
\begin{equation*}
v_m\xrightarrow[m\to\infty]{H}u\quad\mbox{and}\quad \limsup_{k\to\infty} \lVert\mathcal{T}_kv_m-u_k\rVert_{H_k}\xrightarrow[m\to\infty]{}0.
\end{equation*}
From
\begin{equation*}
\lVert\mathcal{T}_ku-u_k\rVert_{H_k}\le \underbrace{\lVert\mathcal{T}_k(v_m-u)\rVert_{H_k}}_{\xrightarrow[k\to\infty]{}\lVert v_m-u\rVert_H}+\lVert\mathcal{T}_kv_m-u_k\rVert_{H_k},
\end{equation*}
the implication holds. The converse follows from the definition.
\end{proof}

As explained above, the choice of $(\mathcal{T}_k)_{k\in\N}$ accounts for an essential step when considering such convergences and a change in those operators can completely alter the way objects will converge along the sequence of Hilbert spaces, even though the spaces themselves remain untouched, see Figure~\ref{Fig:CV-Hilb}.

\begin{figure}[t]
\centering
\begin{tikzpicture}
\draw [densely dotted] (0,0) --  ++ (2,1) -- ++ (0,6) -- ++ (-2,-1) -- cycle ;
\draw [densely dotted] (2.3,0) --  ++ (2,1) -- ++ (0,6) -- ++ (-2,-1) -- cycle ;
\draw [densely dotted] (4.6,0) --  ++ (2,1) -- ++ (0,6) -- ++ (-2,-1) -- cycle ;
\draw [densely dotted] (7.2,0) --  ++ (2,1) -- ++ (0,6) -- ++ (-2,-1) -- cycle ;
\draw [densely dotted] (10,0) --  ++ (2,1) -- ++ (0,6) -- ++ (-2,-1) -- cycle ;
\draw (1.2,-0.4) node{$H_1$} ;
\draw (3.5,-0.4) node{$H_2$} ;
\draw (5.8,-0.4) node{$H_3$} ;
\draw (8.4,-0.4) node{$H_k$} ;
\draw (11.2,-0.4) node{$H$} ;
\draw (7.1,-0.4) node{$\dots$};
\draw (9.8,-0.5) node{$\xrightarrow[k\to\infty]{}$};
\draw (1,3.5) node{\tiny$\bullet$} ;
\draw (3.3,3.5) node{\tiny$\bullet$} ;
\draw (5.6,3.5) node{\tiny$\bullet$} ;
\draw (8.2,3.5) node{\tiny$\bullet$} ;
\draw (11,3.5) node{\tiny$\bullet$} node[below right]{$0$} ;
\draw [dashed] (1,3.5) ellipse (0.75 and 1.5) ;
\draw [dashed] (3.3,3.5) ellipse (0.75 and 1.5) ;
\draw [dashed] (5.6,3.5) ellipse (0.75 and 1.5) ;
\draw [dashed] (8.2,3.5) ellipse (0.75 and 1.5) ;
\draw [dashed] (11,3.5) ellipse (0.75 and 1.5) ;
\draw (0.5,{3.9+sqrt(5)/2}) node[red]{\tiny$\bullet$} node[below right, red]{$\mathcal{T}_1u$} ;
\draw (2.8,{3.8+sqrt(5)/2}) node[red]{\tiny$\bullet$} node[below right, red]{$\mathcal{T}_2u$} ;
\draw (5.1,{3.7+sqrt(5)/2}) node[red]{\tiny$\bullet$} node[below right, red]{$\mathcal{T}_3u$} ;
\draw (7.7,{3.7+sqrt(5)/2-0.1/2.3*2.6}) node[red]{\tiny$\bullet$} node[below right, red]{$\mathcal{T}_ku$} ;
\draw [solid, red] (0.5,{3.9+sqrt(5)/2}) -- ++(1.8,-0.078) edge[dashed] ++(0.5,-0.022) ;
\draw [solid, red] (2.8,{3.8+sqrt(5)/2}) -- ++(1.8,-0.078) edge[dashed] ++(0.5,-0.022) ;
\draw [solid, red] (5.1,{3.7+sqrt(5)/2}) -- ++(2.1,{-0.1/2.3*2.1}) edge[dashed] ++(0.5,{-0.1/2.3*0.5}) ;
\draw [solid, red] (7.7,{3.6+sqrt(5)/2}) -- ++(2.3,-0.1) edge[dashed] ++(0.5,{-0.1/2.3*0.5}) ;
\draw (0.9,{4.3+sqrt(5)/2}) node{\tiny$\bullet$} node[above]{$u_1$} ;
\draw (3.1,{4.1+sqrt(5)/2}) node{\tiny$\bullet$} node[above]{$u_2$} ;
\draw (5.3,{3.9+sqrt(5)/2}) node{\tiny$\bullet$} node[above]{$u_3$} ;
\draw ({5.3+2.2/2.3*2.6},{3.9+sqrt(5)/2-0.2/2.3*2.6}) node{\tiny$\bullet$} node[above]{$u_k$} ;
\draw (0.5,{3.5-sqrt(5)/2}) node[blue]{\tiny$\bullet$} node[above right, blue]{$\mathcal{R}_1u$} ;
\draw (2.8,{3.5-sqrt(5)/2}) node[blue]{\tiny$\bullet$} node[above right, blue]{$\mathcal{R}_2u$} ;
\draw (5.1,{3.5-sqrt(5)/2}) node[blue]{\tiny$\bullet$} node[above right, blue]{$\mathcal{R}_3u$} ;
\draw (7.7,{3.5-sqrt(5)/2}) node[blue]{\tiny$\bullet$} node[above right, blue]{$\mathcal{R}_ku$} ;
\draw [solid, blue] (0.5,{3.5-sqrt(5)/2}) -- ++(1.8,0) edge[dashed] ++(0.5,0);
\draw [solid, blue] (2.8,{3.5-sqrt(5)/2}) -- ++(1.8,0) edge[dashed] ++(0.5,0);
\draw [solid, blue] (5.1,{3.5-sqrt(5)/2}) -- ++(2.1,0) edge[dashed] ++(0.5,0);
\draw [blue] (7.7,{3.5-sqrt(5)/2}) .. controls +(1,0) and +(-0.5,{-sqrt(5)/2+0.3}) .. (10,3.8);
\draw [dashed, blue] (10,3.8)  .. controls +(0.5,{sqrt(5)/2-0.3}) and +(0,0) .. (10.5,{3.5+sqrt(5)/2});
\draw (10.5,{3.5+sqrt(5)/2}) node[red]{\tiny$\bullet$} node[below right, red]{$u$} ;
\end{tikzpicture}
\caption{Illustration of the impact the sequence of linear maps has on the convergences from Definitions~\ref{Def:CV-Hilb} and~\ref{Def:CV-Vect}. The sequence $(H_k)_{k\in\N}$ converges to $H$ through $(H,(\mathcal{T}_k)_{k\in\N})$ and through $(H,(\mathcal{R}_k)_{k\in\N})$. The ellipses represent the spheres of center $0$ and radius $\lVert u\rVert_H$ with respect to the norm on each space. Since the sequences of operators $(\mathcal{T}_k)_{k\in\N}$ and $(\mathcal{R}_k)_{k\in\N}$ are not asymptotically close (in the sense of Lemma~\ref{Lem:EquivCVFramework}), they induce different notions of convergence: the sequence $(u_k\in H_k)_{k\in\N}$ goes to $u\in H$ through $(\mathcal{T}_k)_{k\in\N}$, but not through $(\mathcal{R}_k)_{k\in\N}$.}
\label{Fig:CV-Hilb}
\end{figure}
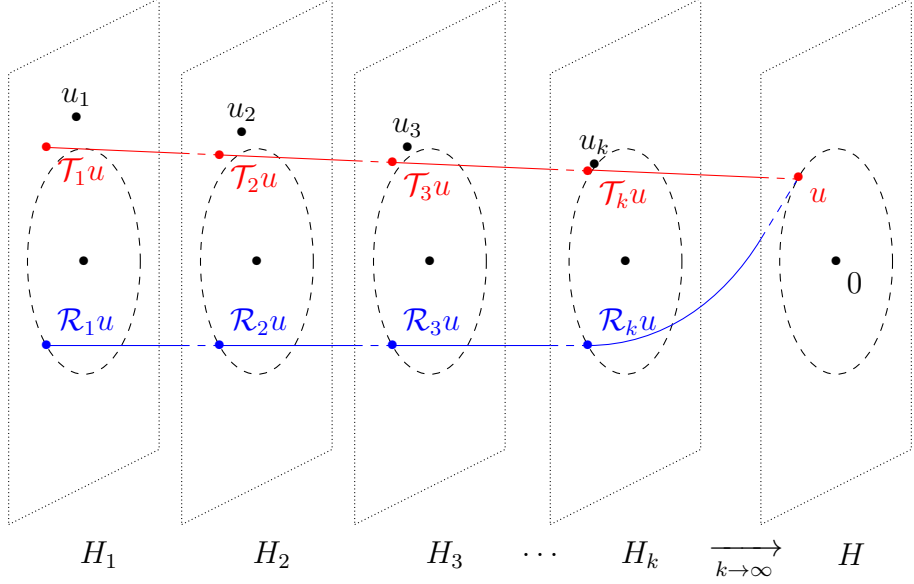

The following lemma indicates what is needed for two sequences of operators in Definition~\ref{Def:CV-Hilb} to induce the same topology for the convergence of vectors in the sense of Definition~\ref{Def:CV-Vect}.

\begin{lemma}\label{Lem:EquivCVFramework}
Let $(H_k)_{k\in\N}$ be a sequence of Hilbert spaces and $H$ be a Hilbert space such that $(H_k)_{k\in\N}\xrightarrow[]{}H$ through $(H,(\mathcal{T}_k)_{k\in\N})$ and through $(H,(\mathcal{R}_k)_{k\in\N})$. Then, if $(u_k\in H_k)\xrightarrow[]{(\mathcal{T}_k)}u\in H$, the following equivalence holds:
\begin{equation*}
u_k\xrightarrow[k\to\infty]{(\mathcal{R}_k)}u\quad\iff\quad \lVert(\mathcal{T}_k-\mathcal{R}_k)u\rVert_{H_k}\xrightarrow[k\to\infty]{} 0.
\end{equation*}
\end{lemma}

\begin{proof}
It holds
\begin{multline*}
\lVert(\mathcal{T}_k-\mathcal{R}_k)u\rVert_{H_k}-\lVert u_k-\mathcal{T}_ku\rVert_{H_k}\le\lVert u_k-\mathcal{R}_ku\rVert_{H_k}\\
\le \lVert(\mathcal{T}_k-\mathcal{R}_k)u\rVert_{H_k}+\lVert u_k-\mathcal{T}_ku\rVert_{H_k},
\end{multline*}
and the equivalence follows by Lemma~\ref{Lem:CV-KS}.
\end{proof}

From the notion of convergence of vectors along a sequence of Hilbert spaces, that of linear bounded operators is derived intuitively.

\begin{definition}[Convergence of bounded operators~\cite{kuwae_convergence_2003}]
Let $(H_k)_{k\in\N}$ and $(V_k)_{k\in\N}$ be sequences of Hilbert spaces, and $H$ and $V$ be Hilbert spaces such that $H_k\to H$ through $(\mathcal{C}_H,(\mathcal{T}_k)_{k\in\N})$ and $V_k\to V$ through $(\mathcal{C}_V,(\mathcal{R}_k)_{k\in\N})$. A sequence $(\mathrm{L}_k:H_k\to V_k)_{k\in\N}$ of linear bounded operators is said to converge to a linear bounded operator $\mathrm{L}:H\to V$ if, for all $(u_k\in H_k)_{k\in\N}$ and $u\in H$, it holds
\begin{equation*}
u_k\xrightarrow[k\to\infty]{(\mathcal{T}_k)}u\quad\Longrightarrow \quad \mathrm{L}_ku_k\xrightarrow[k\to\infty]{(\mathcal{R}_k)}\mathrm{L}u.
\end{equation*}
In that case, we will denote: $\mathrm{L}_k\xrightarrow[k\to\infty]{(\mathcal{T}_k),(\mathcal{R}_k)}\mathrm{L}$.
\end{definition}

This notion of convergence of operators preserves self-adjointness.

\begin{lemma}\label{Lem:CV-SA}
Let $(H_k)_{k\in\N}$ be a sequence of Hilbert spaces and $H$ be a Hilbert space such that $(H_k)_{k\in\N}\xrightarrow[]{}H$ through $(\mathcal{C}_H,(\mathcal{T}_k)_{k\in\N})$. Let $(\Lrm_k\in \mathcal{L}(H_k))$ be a sequence of self-adjoint operators such that $\Lrm_k\xrightarrow[k\to\infty]{(\mathcal{T}_k),(\mathcal{T}_k)} \Lrm\in\mathcal{L}(H)$. Then $\Lrm$ is self-adjoint.
\end{lemma}

\begin{proof}
Let $(u_k\in H_k)_k\xrightarrow[k\to\infty]{(\mathcal{T}_k)} u\in H$ and $(v_k\in H_k)_k\xrightarrow[k\to\infty]{(\mathcal{T}_k)} v\in H$. Then, it holds
\begin{align*}
\langle \Lrm_k^* u,v\rangle_{H_k} &= \langle u, \Lrm_k v\rangle_{H_k} \longrightarrow \langle u, \Lrm v\rangle_{H} = \langle \Lrm^*u,v\rangle_H,\\
&= \langle \Lrm_k u,v\rangle_{H_k}\longrightarrow \langle \Lrm u ,v\rangle_{H},
\end{align*}
hence $\Lrm^*=\Lrm$.
\end{proof}

To conclude this section, we prove a result which will be used extensively in this paper on the convergence of inverse operators.

\begin{lemma}\label{Lem:InverseCV}
Let $(H_k)_{k\in\N}$ and $(V_k)_{k\in\N}$ be sequences of Hilbert spaces, and $H$ and $V$ be Hilbert spaces such that $H_k\to H$ through $(H,(\mathcal{T}_k)_{k\in\N})$ and $V_k\to V$ through $(V,(\mathcal{R}_k)_{k\in\N})$. Let $(\Lrm_k:H_k\to V_k)_{k\in\N}$ and $\Lrm:H\to V$ be linear bounded invertible operators such that $(\Lrm_k^{-1})_{k\in\N}$ is uniformly bounded in $k$. Assume $\Lrm_k\xrightarrow[]{(\mathcal{T}_k),(\mathcal{R}_k)}\Lrm$ as $k\to\infty$. Then $\Lrm_k^{-1}\xrightarrow[]{(\mathcal{R}_k),(\mathcal{T}_k)}\Lrm^{-1}$ as $k\to\infty$.
\end{lemma}

\begin{proof}
Let $(v_k\in V_k)_{k\in\N}\xrightarrow[k\to\infty]{(\mathcal{T}_k)} v\in V$. Then, it holds
\begin{align*}
\lVert\Lrm_k^{-1}v_k-\mathcal{T}_k\Lrm^{-1}v\rVert_{H_k} &\le\lVert\Lrm_k^{-1}(v_k-\mathcal{R}_kv)\rVert_{H_k}+\lVert\Lrm_k^{-1}(\mathcal{R}_k\Lrm-\Lrm_k\mathcal{T}_k)\Lrm^{-1}v\rVert_{H_k}\\
&\le \sup_{\ell\in\N}\lVert\Lrm_\ell^{-1}\rVert\big(\lVert v_k-\mathcal{R}_kv\rVert_{V_k}+\lVert(\mathcal{R}_k\Lrm-\Lrm_k\mathcal{T}_k)\Lrm^{-1}v\rVert_{V_k}\big).
\end{align*}
Since $(\mathcal{T}_k\Lrm^{-1}v\in H_k)\xrightarrow[]{(\mathcal{T}_k)} \Lrm^{-1}v\in H$ as $k\to\infty$, it holds $(\Lrm_k\mathcal{T}_k\Lrm^{-1}v\in V_k)\xrightarrow[]{(\mathcal{R}_k)}\Lrm\Lrm^{-1}v\in V$ and the right-hand side goes to $0$.
\end{proof}

\section{Convergence of layer potential operators for transmission problems}\label{Sec:CV-LP}

For any two-sided admissible domain $\Omega$, the space $\dot H^1(\nbord)$ can be described as any of the following three direct orthogonal sums: 
\begin{align}
\dot H^1(\nbord)&=\dot H^1(\Omega)\oplus \dot H^1(\overline{\Omega}^c), \label{Eq:Decomp-Geo}\\
&= \dot V_0(\nbord)\oplus \dot H^1_0(\nbord), \label{Eq:Decomp-Trans}\\
&= \mathrm{ker}\llbracket\Trd^{\bord}\rrbracket\oplus (\mathrm{ker}\llbracket\Trd^{\bord}\rrbracket)^\perp. \label{Eq:Decomp-H1Rn}
\end{align}
Decomposition~\eqref{Eq:Decomp-Geo} is a `geographic' decomposition of the space,~\eqref{Eq:Decomp-Trans} is in terms of the harmonic transmission problem (using notations similar to Theorem~\ref{Th:Trace}, Point (i)), and~\eqref{Eq:Decomp-H1Rn} is in terms of trace continuity.
Naturally, it holds
\begin{equation*}
\dot H^1_0(\nbord)=\dot H^1_0(\Omega)\oplus \dot H^1_0(\overline{\Omega}^c),
\end{equation*}
as the space of elements with null interior and exterior traces, and the space of solutions to the harmonic transmission problem $\dot V_0(\nbord)$ (i.e., the space of harmonic functions on $\nbord$) can be decomposed into a single layer potential part and a double layer potential part:
\begin{equation*}
\dot V_0(\nbord)=\dot{V}_\Scal(\nbord)\oplus \dot{V}_\D(\nbord),
\end{equation*}
where the single layer potential part is characterised by a homogeneous jump in trace, that is
\begin{equation}\label{Eq:SLP-Space}
\dot{V}_\Scal(\nbord):=\mathrm{ker}\llbracket\Trd^{\bord}\rrbracket\cap \dot V_0(\nbord),
\end{equation}
and the double layer potential part by a homogeneous jump in normal derivative:
\begin{equation}\label{Eq:DLP-Space}
\dot{V}_\D(\nbord):=\mathrm{ker}\left\llbracket\ddno{}{}\Big|_{\bord}\right\rrbracket\cap \dot V_0(\nbord).
\end{equation}
The following lemma gives a characterization of the space of double layer potentials as the orthogonal complement of {$\mathrm{ker}\llbracket\Trd^{\bord}\rrbracket$}.

\begin{lemma}\label{Lem:DLP-ortho}
Let $\Omega$ be a two-sided admissible domain. Then, the space of double layer potentials is characterised by $\dot{V}_\D(\nbord)=(\mathrm{ker}\llbracket\Trd^{\bord}\rrbracket)^\perp$.
\end{lemma}

\begin{proof}
The inclusion $\dot{V}_\D(\nbord)\subset (\mathrm{ker}\llbracket\Trd^{\bord}\rrbracket)^\perp$ follows from the fact that $u\in \dot{V}_\D(\nbord)$ implies $\llbracket\ddno u{}|_{\bord}\rrbracket=0$ and Green's formula.
For the other inclusion, let $u\in (\mathrm{ker}\llbracket\Trd^{\bord}\rrbracket)^\perp$. Then, it holds
\begin{equation*}
\forall v\in \mathrm{ker}\llbracket\Trd^{\bord}\rrbracket,\quad \int_{\nbord}{\nabla u\cdot\nabla v\,\dx}+\int_{\nbord}{uv\,\dx}=0.
\end{equation*}
Since smooth functions compactly supported (up to locally constant functions) on $\nbord$ are in $\mathrm{ker}\llbracket\Trd^{\bord}\rrbracket$, it holds $-\Delta u=0$ weakly on $\nbord$. Therefore, by Green's formula,
\begin{equation*}
\forall v\in \mathrm{ker}\llbracket\Trd^{\bord}\rrbracket,\quad\bigg\langle\bigg\llbracket\ddno u{}\Big|_{\bord}\bigg\rrbracket,\Trd^{\bord} v\bigg\rangle_{\B'\!,\,\B}=\langle u,v\rangle_{\dot H^1(\R^n)}=0,
\end{equation*}
where we denote $\Trd^{\bord} v:=\Trd_{\mathrm i}^{\bord} v=\Trd_{\mathrm e}^{\bord} v$ since the interior and exterior traces of $v$ coincide.
Since $\Trd^{\bord}(\mathrm{ker}\llbracket\Trd^{\bord}\rrbracket)=\Hpo$ since $\Omega$ is an extension domain, we can deduce $\llbracket\ddno u{}|_{\bord}\rrbracket=0$, hence $u\in \dot{V}_\D(\nbord)$.
\end{proof}

\subsection{Convergence on interior domains}\label{Subsec:CV-Int}

In this part, we investigate the link between the convergence of $\dot H^1$ functions defined on interior domains and the convergence of the associated interior boundary values. If $(\Omega_k)_{k\in\N}$ is a non-decreasing sequence of domains of union $\Omega$ (not necessarily admissible domains at this point), one can construct elements of $\dot H^1(\Omega_k)$ from $\dot H^1(\Omega)$ simply by considering restrictions.

\begin{lemma}\label{Lem:CV-H1-i}
Let $\Omega$ be a domain of $\R^n$ and $(\Omega_k)_{k\in\N}$ be a non-decreasing sequence of domains such that $\Omega_k\nearrow\Omega$. Then, it holds
\begin{equation*}
\dot H^1(\Omega_k)\xrightarrow[k\to\infty]{}\dot H^1(\Omega)\quad\mbox{through}\quad\big(\dot H^1(\Omega),(\cdot|_{\Omega_k})_{k\in\N}\big).
\end{equation*}
\end{lemma}

\begin{proof}
From the pointwise convergence of the characteristic functions and dominated convergence, for all $u\in \dot H^1(\Omega)$, it holds $\lVert u|_{\Omega_k}\rVert_{\dot H^1(\Omega_k)}\to \lVert u\rVert_{\dot H^1(\Omega)}$.
\end{proof}

\begin{remark}\label{Rem:CV-V1-i}
The operators $\cdot|_{\Omega_k}$ from Lemma~\ref{Lem:CV-H1-i} preserve harmonicity in the sense that $\dot V_0(\Omega)|_{\Omega_k}\subset \dot V_0(\Omega_k)$ for all $k\in\N$. For that matter, the following convergence holds:
\begin{equation*}
\dot V_0(\Omega_k)\xrightarrow[k\to\infty]{}\dot V_0(\Omega)\quad\mbox{through}\quad\big(\dot V_0(\Omega),(\cdot|_{\Omega_k})_{k\in\N}\big).
\end{equation*}
\end{remark}

As explained in Subsection~\ref{Subsec:CV-KS}, the notion of convergence of Hilbert spaces, and consequently, of vectors along that convergence, it very weak. In the case of the convergence described in Lemma~\ref{Lem:CV-H1-i}, in addition to the convergence of the norms, the following result proves the convergence in $\dot H^1(\R^n)$ of the Dirichlet extensions, which hints at the fact that the notion of convergence from Lemma~\ref{Lem:CV-H1-i} can be strengthened (see Proposition~\ref{Prop:Link-CV-H1Rn} below).

\begin{proposition}\label{Prop:ExtCV-i}
Let $\Omega$ be an admissible domain and $(\Omega_k)_{k\in\N}$ be a non-decreasing sequence of admissible domains such that $\Omega_k\nearrow\Omega$. Then, for all $u\in \dot H^1(\Omega)$,
\begin{equation*}
\dot{\mathrm{E}}_{\Omega_k}(u|_{\Omega_k})\xrightarrow[k\to\infty]{}\dot{\mathrm{E}}_{\Omega}u\quad\mbox{in }\dot H^1(\R^n).
\end{equation*}
\end{proposition}

\begin{proof}
Let $u\in \dot H^1(\Omega)$. For $k\in\N$, denote $u_k:=u|_{\Omega_k}\in \dot H^1(\Omega_k)$. Then $u$ and $u_k$ have the same trace on $\bord_k$ and $\dot{\mathrm{E}}_{\Omega_k}u_k$ is harmonic on $\overline{\Omega}_k^c$. Therefore, by definition of $u_k$ on $\Omega_k$ and by energy minimization on $\overline{\Omega}_k^c$, it holds:
\begin{equation}\label{Eq:Extensions-UpperBound}
\forall k\in\N,\quad \rVert\dot{\mathrm{E}}_{\Omega_k}u_k\rVert_{\dot H^1(\R^n)}\le \lVert\dot{\mathrm{E}}_\Omega u\rVert_{\dot H^1(\R^n)}.
\end{equation}
Hence, there exists $u_\infty\in \dot H^1(\R^n)$ and an increasing sequence $(k_m)_m\in\N^\N$ such that:
\begin{equation*}
\dot{\mathrm{E}}_{\Omega_{k_m}}u_{k_m}\xrightharpoonup[m\to\infty]{}u_\infty\quad\mbox{in }\dot H^1(\R^n).
\end{equation*}
If $k\in\N$, then for all $\ell\ge k$, $(\dot{\mathrm{E}}_{\Omega_\ell}u_\ell)|_{\Omega_k}=u|_{\Omega_k}$. Since $\bigcup_k \Omega_k=\Omega$, we can deduce:
\begin{equation*}
u_\infty|_\Omega=u.
\end{equation*}
Moreover, since all $\dot{\mathrm{E}}_{\Omega_k}u_k$ are harmonic on $\overline{\Omega}^c$, the weak limit $u_\infty$ is harmonic on $\overline{\Omega}^c$ as well (by definition of the weak Laplacian). Therefore, $u_\infty=\dot{\mathrm{E}}_\Omega u$ is the only weak subsequential limit of $(\dot{\mathrm{E}}_{\Omega_k}u_k)_{k\in\N}$, hence 
\begin{equation}\label{Eq:Extensions-WCV}
\dot{\mathrm{E}}_{\Omega_k}u_k\xrightharpoonup[k\to\infty]{}\dot{\mathrm{E}}_\Omega u\quad\mbox{in }\dot H^1(\R^n).
\end{equation}
By~\eqref{Eq:Extensions-UpperBound} and the lower semi-continuity of the weak limit, $ \rVert\dot{\mathrm{E}}_{\Omega_k}u_k\rVert_{\dot H^1(\R^n)}\to\lVert\dot{\mathrm{E}}_\Omega u\rVert_{\dot H^1(\R^n)}$, hence
\begin{equation*}
\dot{\mathrm{E}}_{\Omega_k}u_k\xrightarrow[k\to\infty]{}\dot{\mathrm{E}}_\Omega u\quad\mbox{in }\dot H^1(\R^n).\qedhere
\end{equation*}
\end{proof}

Having defined the framework for the convergence of the interior $\dot H^1$ spaces, we turn to the convergence of the trace spaces.

\begin{lemma}\label{Lem:CV-B}
Let $\Omega$ be an admissible domain. Let $(\Omega_k)_{k\in\N}$ be a non-decreasing sequence of admissible domains such that $\Omega_k\nearrow\Omega$. Then, it holds:
\begin{equation*}
\dot \B(\bord_k)\xrightarrow[k\to\infty]{}\Hpo \quad \mbox{through }\big(\Hpo,(\trd^{\bord_k}\circ\,\dot{\mathrm{H}}_{\bord})_{k\in\N}\big),
\end{equation*}
where $\dot{\mathrm{H}}_{\bord}$ is defined by~\eqref{Eq:D-Ext}.
\end{lemma}

\begin{proof}
Let $f\in\Hpo$. As for Lemma~\ref{Lem:CV-H1-i}, it holds:
\begin{equation*}
\lVert(\dot{\mathrm{H}}_{\bord}f)|_{\Omega_k}\rVert_{\dot H^1(\Omega_k)}^2\xrightarrow[k\to\infty]{}\lVert(\dot{\mathrm{H}}_{\bord}f)|_\Omega\rVert_{\dot H^1(\Omega)}^2.
\end{equation*}
Since the trace operators $\trd_{\mathrm{i}}^{\bord_k}:\dot V_0(\Omega_k)\to\dot \B(\bord_k)$ are isometric, we can deduce:
\begin{equation*}
\lVert\trd^{\bord_k}\dot{\mathrm{H}}_{\bord}f\rVert_{\dot \B(\bord_k)}\xrightarrow[k\to\infty]{}\lVert\trd^{\bord}\dot{\mathrm{H}}_{\bord}f\rVert_{\Hpo}=\lVert f\rVert_{\Hpo}.\qedhere
\end{equation*}
\end{proof}

\begin{remark}
In Lemma~\ref{Lem:CV-B}, we have made use of the fact that $\dot{\mathrm{H}}_{\bord}(\Hpo)\subset \dot H^1(\R^n)$, hence the traces involved can be regarded as either interior or exterior.
\end{remark}

The isometry defined by the trace operators on the spaces of harmonic functions (Theorem~\ref{Th:Trace}) allows to prove that, in the framework of Lemmas~\ref{Lem:CV-H1-i} and~\ref{Lem:CV-B}, convergence of harmonic functions and of their traces are equivalent.

\begin{proposition}\label{Prop:TrCV-i}
Let $\Omega$ be an admissible domain and $(\Omega_k)_{k\in\N}$ be a non-decreasing sequence of admissible domains such that $\Omega_k\nearrow\Omega$. Let $u\in \dot V_0(\Omega)$ and $(u_k\in \dot V_0(\Omega_k))_{k\in\N}$. Then, the following equivalence holds:
\begin{equation*}
u_k\xrightarrow[k\to\infty]{(\cdot|_{\Omega_k})}u\quad\iff\quad (\trd_{\mathrm{i}}^{\bord_k}u_k\in\dot \B(\bord_k))_{k\in\N}\xrightarrow[]{(\trd^{\bord_k}\circ\,\dot{\mathrm{H}}_{\bord})}\trd^{\bord}_{\mathrm{i}}u\in\Hpo.
\end{equation*}
\end{proposition}

\begin{proof}
Using the isometric properties of the interior trace, it holds
\begin{equation}\label{Eq:TrCVDom}
\big\lVert\trd^{\bord_k}_{\mathrm{i}} u_k - \trd^{\bord_k}\underbrace{({\dot{\mathrm{H}}}_{\bord}\trd_{\mathrm{i}}^{\bord}u)}_{\dot{\mathrm{E}}_\Omega u}\big\rVert_{\dot \B(\bord_k)}= \big\lVert u_k- u|_{\Omega_k}\big\rVert_{\dot H^1(\Omega_k)}.
\end{equation}
The equivalence follows (see Lemma~\ref{Lem:CV-KS}).
\end{proof}

Regarding the link between convergence of $\dot H^1$ functions on interior domains and of interior boundary values, Proposition~\ref{Prop:TrCV-i} states the equivalence in the case of harmonic functions.
Given the decomposition $\dot H^1(\Omega)=\dot V_0(\Omega)\oplus \dot H^1_0(\Omega)$, all that is left is to study the case of functions in the kernel of the trace operator.
Unlike what was pointed out in Remark~\ref{Rem:CV-V1-i}, the operators $\cdot|_{\Omega_k}$ from Lemma~\ref{Lem:CV-H1-i} do not preserve trace nullity, in the sense that $\dot H^1_0(\Omega)|_{\Omega_k}\not\subset \dot H^1_0(\Omega_k)$ in general.
However, since the sequence $(\Omega_k)_{k\in\N}$ also converges to $\Omega$ in the sense of compact sets -- by monotone convergence, adapting the proof of Theorem~\ref{Th:DyadApproxCV}, Point (iii) for example --, then $(\dot H^1_0(\Omega_k))_{k\in\N}$ can be seen as a non-decreasing sequence of subsets of $\dot H^1_0(\Omega)$ with dense union.

\begin{lemma}\label{Lem:H10-Density}
Let $\Omega$ be an admissible domain and $(\Omega_k)_{k\in\N}$ be a non-decreasing sequence of admissible domains such that $\Omega_k\nearrow\Omega$. Then $\bigcup_k \dot H^1_0(\Omega_k)$ is a dense subspace of $\dot H^1_0(\Omega)$.
\end{lemma}

\begin{proof}
Let $u\in \dot H^1_0(\Omega)$. There exists $(\varphi_k)_k\in C^\infty_0(\Omega)^\N$ such that $\varphi_k\to u$ in $\dot H^1(\Omega)$. For all $k\in\N$, by the convergence in the sense of compact sets, there exists $m\in\N$ such that $\mathrm{supp}\,\varphi_k\subset\Omega_m$, hence $\varphi_k\in \dot H^1_0(\Omega_m)$. The result follows.
\end{proof}

In addition, the following lemma holds.

\begin{lemma}\label{Lem:CV-Proj-Density}
Assume $H_k$ is a non-decreasing sequence of Hilbert subspaces of $H$ such that $\bigcup_k H_k$ is dense in $H$. For $k\in\N$, denote by $p_k$ the orthogonal projection from $H$ to $H_k$. Then, it holds
\begin{equation*}
\forall u\in H,\quad p_ku\xrightarrow[k\to\infty]{}u\quad\mbox{in }H.
\end{equation*}
\end{lemma}

\begin{proof}
Let $u\in H$. Let $(\varphi_m)_{m\in\N}$ be such that $\varphi_m\to u$ and $\varphi_m\in H_{k_m}$ for all $m\in\N$, where $(k_m)\in\N^\N$ is increasing. Then, it holds
\begin{equation*}
\lVert\varphi_m-p_{k_m}u\rVert_H=\lVert p_{k_m}(\varphi_m-u)\rVert_H\le\lVert\varphi_m-u\rVert_H\xrightarrow[m\to\infty]{}0.
\end{equation*}
Therefore, $\lVert u-p_{k_m}u\rVert_H\to0$ as $m\to\infty$. Since $(H_k)_k$ forms a non-decreasing sequence, $(\lVert u-p_ku\rVert_H)_k$ is non-increasing so that $p_ku\to u$ in $H$ as $k\to\infty$.
\end{proof}

Under the hypotheses of Lemma~\ref{Lem:H10-Density}, the last two lemmas allow to prove that the orthogonal projectors from $\dot H^1_0(\Omega)$ to $\dot H^1_0(\Omega_k)$ converge to the identity on $\dot H^1_0(\Omega)$. We extend this result to the case in which the projectors are defined on all of $\dot H^1(\Omega)$ instead.

\begin{proposition}\label{Lem:CV-Proj-H10}
Let $\Omega$ be an domain and $(\Omega_k)_{k\in\N}$ be a non-decreasing sequence of admissible domains such that $\Omega_k\nearrow\Omega$. Denote by $p^0$ the orthogonal projection from $\dot H^1(\Omega)$ to $\dot H^1_0(\Omega)$ and for $k\in\N$, denote by $p^0_k$ the orthogonal projection from $\dot H^1(\Omega)$ to $\dot H^1_0(\Omega_k)$ (thought of as a subspace of $\dot H^1_0(\Omega)$). Then, it holds
\begin{equation*}
\forall u\in \dot H^1(\Omega),\quad p^0_ku\xrightarrow[k\to\infty]{\dot H^1(\Omega)}p^0u.
\end{equation*}\end{proposition}

\begin{remark}\label{Rem:ProjH10}
In the definition of $p_k^0$, $\dot H^1_0(\Omega_k)$ is thought of as its extension by $0$ to $\Omega$. In that setting, for all $u\in \dot H^1(\Omega)$, it holds $p^0_ku=p^0_k(u\mathds{1}_{\Omega_k})$. It follows that $p^0_k$ can be seen as the orthogonal projection from $\dot H^1(\Omega_k)$ to $\dot H^1_0(\Omega_k)$.
\end{remark}

\begin{proof}
Lemmas~\ref{Lem:H10-Density} and~\ref{Lem:CV-Proj-Density} yield the result in the case $u\in \dot H^1_0(\Omega)$ (in which case $p^0u=u$). To extend the convergence to all $u\in \dot H^1(\Omega)$, it is enough to prove $p^0_k\circ p^0=p^0_k$. The space $\dot H^1(\Omega)$ can be decomposed into the following direct orthogonal sums:
\begin{equation*}
\dot H^1(\Omega)=\dot V_0(\Omega)\oplus \dot H^1_0(\Omega)=\dot H^1_0(\Omega_k)\oplus \dot H^1_0(\Omega_k)^\perp,
\end{equation*}
where the orthogonal to $\dot H^1_0(\Omega_k)$ in $\dot H^1(\Omega)$ is described by
\begin{equation*}
\dot H^1_0(\Omega_k)^\perp=\big\{v\in \dot H^1(\Omega)\;\mid\;v|_{\Omega_k}\in \dot V_0(\Omega_k)\big\}.
\end{equation*}
Noticing that $\dot V_0(\Omega)\subset \dot H^1_0(\Omega_k)^\perp$, $\dot H^1(\Omega)$ can be further decomposed into
\begin{equation}\label{Eq:DoubleDecompH1}
\dot H^1(\Omega)=\dot V_0(\Omega)\oplus (\dot H^1_0(\Omega)\cap \dot H^1_0(\Omega_k)) \oplus (\dot H^1_0(\Omega)\cap \dot H^1_0(\Omega_k)^\perp).
\end{equation}
where the direct sums are orthogonal. Let $u\in \dot H^1(\Omega)$, decomposed along~\eqref{Eq:DoubleDecompH1} into $u=v+\varphi+w$. By Remark~\ref{Rem:ProjH10} and the fact that $v|_{\Omega_k},w|_{\Omega_k}\in \dot V_0(\Omega_k)$, we can deduce
\begin{equation*}
p^0_k((v+\varphi+w)|_{\Omega_k})=p^0_k((\varphi+w)|_{\Omega_k})=\varphi|_{\Omega_k},
\end{equation*}
hence $p_k^0u=p_k^0(p^0u)$.
\end{proof}

The convergence of the projectors on the kernel of the trace operator from Lemma~\ref{Lem:CV-Proj-H10} allows to complement Proposition~\ref{Prop:TrCV-i} and prove that the interior convergence implies the trace convergence without assuming harmonicity. Of course, the converse cannot hold without that assumption, since the $\dot H^1_0$ part of an $\dot H^1$ function is lost upon applying the trace operator.

\begin{proposition}\label{Prop:Cv-Tr-H1-i}
Let $\Omega$ be an admissible domain and $(\Omega_k)_{k\in\N}$ be a non-decreasing sequence of admissible domains such that $\Omega_k\nearrow\Omega$. Let $(u_k\in \dot H^1(\Omega_k))_{k\in\N}$ and $u\in \dot H^1(\Omega)$. Then, if $u_k\xrightarrow[k\to\infty]{(\cdot|_{\Omega_k})} u$, it holds
\begin{equation*}
\Trd_{\mathrm{i}}^{\bord_k}u_k\xrightarrow[k\to\infty]{(\trd^{\bord_k}\circ\,\dot{\mathrm{H}}_{\bord})}\Trd_{\mathrm{i}}^{\bord}u,
\end{equation*}
using the convergence frameworks from Lemmas~\ref{Lem:CV-H1-i} and~\ref{Lem:CV-B}.
\end{proposition}

\begin{proof}
Denote $u_k:=v_k+\varphi_k$ and $u=v+\varphi$, where $v_k\in \dot V_0(\Omega_k)$, $\varphi_k\in \dot H^1_0(\Omega_k)$, $v\in \dot V_0(\Omega)$ and $\varphi\in \dot H^1_0(\Omega)$.
Since $p_k^0$ is an orthogonal projection, it holds
\begin{align*}
\lVert\varphi|_{\Omega_k}-\varphi_k\rVert_{\dot H^1(\Omega_k)}-\lVert p^0u-p^0_ku\rVert_{\dot H^1(\Omega_k)}&\le\lVert p_k^0u-\varphi_k\rVert_{\dot H^1(\Omega_k)}\\
&\le \lVert u|_{\Omega_k}-u_k\rVert_{\dot H^1(\Omega_k)}\xrightarrow[k\to\infty]{}0.
\end{align*}
Since the second term on the left-hand side goes to $0$ by Lemma~\ref{Lem:CV-Proj-Density}, we can deduce $\varphi_k\xrightarrow[k\to\infty]{(\cdot|_{\Omega_k})}\varphi$, hence $v_k\xrightarrow[k\to\infty]{(\cdot|_{\Omega_k})}v$ and the result follows by Proposition~\ref{Prop:TrCV-i}.
\end{proof}

An immediate corollary of Proposition~\ref{Prop:Cv-Tr-H1-i} is that, under the same hypotheses, it holds $\Trd_{\mathrm{i}}^{\bord_k}(u|_{\Omega_k})\to \Trd_{\mathrm{i}}^{\bord}u$ through $(\trd^{\bord_k}\circ\dot{\mathrm{H}}_{\bord_k})_{k\in\N}$ for all $u\in \dot H^1(\Omega)$.
This result seems natural and goes to show that the convergence framework from Lemma~\ref{Lem:CV-B} is `consistent'. Given the definition of convergence of Hilbert spaces (Definition~\ref{Def:CV-Hilb}) and the Riesz representation theorem, the convergence of the trace spaces from Lemma~\ref{Lem:CV-B} induces a framework for the convergence of their dual spaces. This framework involves the Riesz isometry between the trace space and its dual, which is non other than the Poincaré-Steklov operator~\cite[Lemma 3.1]{claret_layer_2025}, defined for any admissible domain $\Omega$ by
\begin{equation}\label{Eq:P-S-Op}
\begin{aligned}
\dot\Lambda_{\bord}:\Hpo &\longrightarrow\Hmo\\
\trd^{\bord}_{\mathrm i}u&\longmapsto\ddno u{\mathrm i},
\end{aligned}
\end{equation}
for all $u\in \dot V_0(\Omega)$.

\begin{lemma}\label{Lem:CV-B'}
Let $\Omega$ be an admissible domain. Let $(\Omega_k)_{k\in\N}$ be a non-decreasing sequence of admissible domains such that $\Omega_k\nearrow\Omega$. Then, it holds:
\begin{equation*}
\dot\B'(\bord_k)\xrightarrow[k\to\infty]{}\Hmo \quad \mbox{through }\big(\Hmo,(\Xi_k)_{k\in\N}\big),
\end{equation*}
where $\Xi_k:=\dot \Lambda_{\bord_k}\circ\trd^{\bord_k}\circ\,\dot{\mathrm{H}}_{\bord}\circ\dot \Lambda_{\bord}^{-1}:\Hmo\to\dot\B'(\bord_k)$.
\end{lemma}

\begin{proof}
This convergence follows by Lemma~\ref{Lem:CV-B} and the fact that the Poincaré-Steklov operators $\dot \Lambda_{\bord_k}:\dot \B(\bord_k)\to\dot\B'(\bord_k)$ and $\dot \Lambda_{\bord}:\Hpo\to\Hmo$ are Riesz isometries~\cite[Lemma 3.1]{claret_layer_2025}, for instance
\begin{equation*}
\forall f_1,f_2\in\Hpo,\quad\langle f_1,f_2\rangle_{\Hpo}=\langle \dot \Lambda_{\bord} f_1,f_2\rangle_{\Hmo,\Hpo}.\qedhere
\end{equation*}
\end{proof}

Given the connection between the convergences from Lemmas~\ref{Lem:CV-B} and~\ref{Lem:CV-B'}, one can formulate a statement similar to that of Proposition~\ref{Prop:TrCV-i}, only regarding normal derivatives instead of traces. 

\begin{proposition}\label{Prop:CV-ddni}
Let $\Omega$ be an admissible domain and $(\Omega_k)_{k\in\N}$ be a non-decreasing sequence of admissible domains such that $\Omega_k\nearrow\Omega$. Let $u\in \dot V_0(\Omega)$ and $(u_k\in \dot V_0(\Omega_k))_{k\in\N}$. Then, the following equivalence holds:
\begin{equation*}
u_k\xrightarrow[k\to\infty]{(\cdot|_{\Omega_k})}u\quad\iff\quad \bigg(\ddno{u_k}{\mathrm{i}}\Big|_{\bord_k}\in\dot\B'(\bord_k)\bigg)_{k\in\N}\xrightarrow[]{(\Xi_k)}\ddno u{\mathrm{i}}\Big|_{\bord}\in\Hmo.
\end{equation*}
\end{proposition}

\begin{proof}
Using the isometric properties of the normal derivative (Proposition~\ref{Prop:ddn-isom}), it holds
\begin{align*}
\bigg\lVert\ddno{u_k}{\mathrm{i}}\Big|_{\bord_k}-\Xi_k\ddno u{\mathrm{i}}\Big|_{\bord}\bigg\rVert_{\dot\B'(\bord_k)} &=\bigg\lVert\ddno{u_k}{\mathrm{i}}\Big|_{\bord_k}-\ddno{(\dot{\mathrm{E}}_{\Omega}u)}{\mathrm{i}}\Big|_{\bord_k}\bigg\rVert_{\dot\B'(\bord_k)}\\
&= \big\lVert u_k- u|_{\Omega_k}\big\rVert_{\dot H^1(\Omega_k)}.
\end{align*}
The equivalence follows.
\end{proof}

A direct corollary of the equivalence between convergence of harmonic functions and of their interior boundary values from Propositions~\ref{Prop:TrCV-i} and~\ref{Prop:CV-ddni} is the convergence of the Poincaré-Steklov operators and of their inverses.

\begin{corollary}
Let $\Omega$ be an admissible domain and $(\Omega_k)_{k\in\N}$ be a non-decreasing sequence of admissible domains such that $\Omega_k\nearrow\Omega$. Then, the Poincaré-Steklov operators (defined for $\Omega_k$ as in~\eqref{Eq:P-S-Op}) and their inverses converge:
\begin{equation*}
\dot\Lambda_{\bord_k}\xrightarrow[k\to\infty]{(\trd^{\bord_k}\circ\,\dot{\mathrm{H}}_{\bord}),(\Xi_k)}\dot\Lambda_{\bord}\qquad\mbox{and}\qquad \dot\Lambda_{\bord_k}^{-1}\xrightarrow[k\to\infty]{(\Xi_k),(\trd^{\bord_k}\circ\,\dot{\mathrm{H}}_{\bord})}\dot\Lambda_{\bord}^{-1}.
\end{equation*}
\end{corollary}

\begin{remark}\label{Rem:B'-Change}
The convergence framework for the duals of the trace spaces introduced in Lemma~\ref{Lem:CV-B'} is consistent with that of the trace spaces from Lemma~\ref{Lem:CV-B}.
However, it could be deemed more natural to consider a framework based on the Neumann extension from $\bord$ to $\nbord$
\begin{equation}\label{Eq:N-Harmo-Ext}
\dot{\mathrm N}_{\bord}:g\in\Hmo\longmapsto u\in \dot V_0(\nbord)\quad \mbox{with }\ddno u{}\Big|_{\bord}:=\ddno u{\mathrm i}\Big|_{\bord}=\ddno u{\mathrm e}\Big|_{\bord}=g,
\end{equation}
and proceed similarly to what was done for the trace spaces.
Proceeding as in the proof of Lemma~\ref{Lem:CV-B}, it holds $\dot\B'(\bord_k)\to\Hmo$ through $(\Hmo,(\ddno{}{}\big|_{\bord_k}\circ\dot{\mathrm N}_{\bord})_{k\in\N})$, where $\dot{\mathrm{N}}_{\bord}:\Hmo\to \dot H^1(\nbord)$ is the Neumann harmonic extension (harmonic on $\nbord$). Under the monotone convergence hypothesis, both $\Xi_k$ and $\ddno{}{}\big|_{\bord_k}\circ\dot{\mathrm N}_{\bord}$ are equal.
Indeed, for $g\in\Hmo$, let $u\in \dot V_0(\Omega)$ with $\ddno u{\mathrm i}\big|_{\bord}=g$.
If $v,w\in\dot V_0(\overline{\Omega}^c)$ are such that $u\oplus v\in\dot V_{\Scal}(\nbord)$ and $u\oplus w\in\dot V_{\D}(\nbord)$ (defined by~\eqref{Eq:SLP-Space} and~\eqref{Eq:DLP-Space}), then it holds
\begin{align*}
&\Xi_kg=\ddno{(u\oplus v)}{}\Big|_{\bord_k},&&\mbox{while}&&\ddno{(\dot{\mathrm N}_{\bord}g)}{}\Big|_{\bord_k}=\ddno{(u\oplus w)}{}\Big|_{\bord_k}.
\end{align*}
Since $\Omega_k\subset\Omega$, both quantities coincide. When that hypothesis is foregone (see Section~\ref{Sec:Monotonicity}), the operators differ but still induce the same notion of convergence, see Lemma~\ref{Lem:Equiv-B'-Frameworks}.
\end{remark}

\subsection{Convergence on exterior domains}\label{Subsec:CV-Ext}

In this part, we investigate the link between the convergence of $\dot H^1$ functions defined on exterior domains and the convergence of the associated exterior boundary values.
If $(\Omega_k)_{k\in\N}$ is a non-decreasing sequence of two-sided admissible domains of union $\Omega$, one can construct elements of $\dot H^1(\overline{\Omega}_k^c)$ from $\dot H^1(\overline{\Omega})$ by considering extensions.
Since we are interested in the convergence of the operators associated to the transmission problem for $-\Delta$, the most intuitive approach is to consider the convergence of those spaces through $((\dot{\mathrm{E}}_{\overline{\Omega}^c}\cdot)|_{\overline{\Omega}_k^c})_{k\in\N}$.
However, such a transformation does not preserve harmonicity in the sense that $(\dot{\mathrm{E}}_{\overline{\Omega}^c}\cdot)|_{\overline{\Omega}_k^c}(\dot V_0(\overline{\Omega}^c))\not\subset \dot V_0(\overline{\Omega}_k^c)$ in general.
For that matter, we use a different convergence framework, designed to preserve the generic decomposition $\dot H^1=\dot V_0\oplus \dot H^1_0$ and which relies on the following result on convergence of harmonic extensions.

\begin{proposition}\label{Prop:ExtCV-e}
Let $\Omega$ be a two-sided admissible domain and $(\Omega_k)_{k\in\N}$ be a non-decreasing sequence of two-sided admissible domains such that $\Omega_k\nearrow\Omega$. Then, for all $v\in \dot V_0(\overline{\Omega}^c)$,
\begin{equation*}
\dot{\mathrm{E}}_{\Omega_k}(\dot{\mathrm{E}}_{\overline{\Omega}^c}v|_{\Omega_k})\xrightarrow[k\to\infty]{}\dot{\mathrm{E}}_{\overline{\Omega}^c}v\quad\mbox{in }\dot H^1(\R^n).
\end{equation*}
\end{proposition}

\begin{proof}
Let $v\in \dot V_0(\overline{\Omega}^c)$. Denote $u:=(\dot{\mathrm{E}}_{\overline{\Omega}^c}v)|_\Omega\in \dot H^1(\Omega)$. Then,
$\dot{\mathrm{E}}_{\overline{\Omega}^c}v=\dot{\mathrm{E}}_\Omega u$ and $\dot{\mathrm{E}}_{\Omega_k}\big((\dot{\mathrm{E}}_{\overline{\Omega}^c}v)|_{\Omega_k}\big)=\dot{\mathrm{E}}_{\Omega_k}(u|_{\Omega_k})$. The convergence follows by Proposition~\ref{Prop:ExtCV-i}.
\end{proof}

Note that the result from Proposition~\ref{Prop:ExtCV-e} only applies for the extension of harmonic functions. 
Therefore, when introducing a convergence framework for the exterior $\dot H^1$ spaces, we rely on the orthogonal decomposition of $\dot H^1$ as $\dot V_0\oplus \dot H^1_0$.

\begin{corollary}\label{Cor:CV-H1-e}
Let $\Omega$ be a two-sided admissible domain. Let $(\Omega_k)_{k\in\N}$ be a sequence of two-sided admissible domains such that $\Omega_k\nearrow\Omega$. Then, it holds
\begin{equation*}
\dot H^1(\overline{\Omega}_k^c)\xrightarrow[k\to\infty]{}\dot H^1(\overline{\Omega}^c)\quad\mbox{through}\quad \big(\dot H^1(\overline{\Omega}^c),(\mathcal{E}_k\oplus i_k)_{k\in\N}\big),
\end{equation*}
with $\mathcal{E}_k\oplus i_k$ decomposed along the direct orthogonal sum $\dot V_0(\overline{\Omega}^c)\oplus \dot H^1_0(\overline{\Omega}^c)$, where
\begin{equation*}
\mathcal{E}_k:=\big(\dot{\mathrm{E}}_{\Omega_k}\big[(\dot{\mathrm{E}}_{\overline{\Omega}^c}\cdot)|_{\Omega_k}\big]\big)\big|_{\overline{\Omega}_k^c}:\dot V_0(\overline{\Omega}^c) \longrightarrow \dot V_0(\overline{\Omega}_k^c),
\end{equation*}
and $i_k:\dot H^1_0(\overline{\Omega}^c)\to \dot H^1_0(\overline{\Omega}_k^c)$ is the extension by $0$.
\end{corollary}

\begin{proof}
Let $u\in \dot H^1(\overline{\Omega}^c)$. There exists a unique pair $(v,w)\in \dot V_0(\overline{\Omega}^c)\times \dot H^1_0(\overline{\Omega}^c)$ such that $u=v+w$, and the decomposition is orthogonal. Since, for all $k\in\N$, $i_k$ is norm-preserving, the result follows by Proposition~\ref{Prop:ExtCV-e}.
\end{proof}

\begin{remark}
By design, the operators from Corollary~\ref{Cor:CV-H1-e} preserve the decomposition $\dot H^1=\dot V_0\oplus \dot H^1_0$. For that matter, the following convergences hold:
\begin{align*}
\dot V_0(\overline{\Omega}_k^c)\xrightarrow[k\to\infty]{}\dot V_0(\overline{\Omega}^c)\quad&\mbox{through}\quad \big(\dot V_0(\overline{\Omega}^c),(\mathcal{E}_k)_{k\in\N}\big),\\
\dot H^1_0(\overline{\Omega}_k^c)\xrightarrow[k\to\infty]{}\dot H^1_0(\overline{\Omega}^c)\quad&\mbox{through}\quad \big(\dot H^1_0(\overline{\Omega}^c),(i_k)_{k\in\N}\big).
\end{align*}
\end{remark}

Using this convergence framework, we formulate a result similar to Proposition~\ref{Prop:TrCV-i} for exterior values instead of interior ones, relying on the isometries defined by the exterior trace operators.
However, the convergence framework for the trace spaces from Lemma~\ref{Lem:CV-B} is centered around the trace norms associated to the interior trace operators. Therefore, we must further assume that the trace norm equivalence~\eqref{Eq:TrNormEquiv} is uniform along the sequence of domains.

\begin{proposition}\label{Prop:TrCV-e}
Let $\Omega$ be a two-sided admissible domain and $(\Omega_k)_{k\in\N}$ be a non-decreasing sequence of two-sided admissible domains such that $\Omega_k\nearrow\Omega$.
If the Dirichlet harmonic extensions $(\dot{\mathrm{E}}_{\Omega_k})_{k\in\N}$ are uniformly bounded, then the following implication holds:
\begin{equation*}
(v_k\in \dot V_0(\overline{\Omega}_k^c))_{k\in\N}\xrightarrow[]{(\mathcal E_k)}v\in \dot V_0(\overline{\Omega}^c),
\end{equation*}
implies
\begin{equation*}
(\trd_{\mathrm{e}}^{\bord_k}v_k\in\dot \B(\bord_k))_{k\in\N}\xrightarrow[]{(\trd^{\bord_k}\circ\,\dot{\mathrm{H}}_{\bord})}\trd_{\mathrm{e}}^{\bord}v\in\Hpo,
\end{equation*}
using the convergence frameworks from Lemma~\ref{Lem:CV-B} and Corollary~\ref{Cor:CV-H1-e}.
If the Dirichlet harmonic extensions $(\dot{\mathrm{E}}_{\overline{\Omega}_k^c})_{k\in\N}$ are uniformly bounded, then the converse implication holds instead.
\end{proposition}

\begin{proof}
Using the isometric properties of the exterior trace (Theorem~\eqref{Th:Trace}, Point (iv)) and the uniform boundedness of $(\dot{\mathrm{E}}_{\Omega_k})_{k\in\N}$, there exists $c>0$ uniform in $k$ such that:
\begin{equation*}
\big\lVert\trd^{\bord_k}_{\mathrm{e}} u_k - \trd^{\bord_k}\underbrace{({\dot{\mathrm{H}}}_{\bord}\trd_{\mathrm{e}}^{\bord}u)}_{\dot{\mathrm{E}}_{\overline{\Omega}^c} u}\big\rVert_{\dot \B(\bord_k)}\le c \big\lVert u_k-\dot{\mathrm{E}}_{\Omega_k}(\dot{\mathrm{E}}_{\overline{\Omega}^c}u|_{\Omega_k})|_{\overline{\Omega}_k^c}\big\rVert_{\dot H^1(\overline{\Omega}_k^c)},
\end{equation*}
and by uniform boundedness of $(\dot{\mathrm{E}}_{\overline{\Omega}_k^c})_{k\in\N}$, there exists $c'>0$ uniform in $k$ such that
\begin{equation*}
\big\lVert u_k-\dot{\mathrm{E}}_{\Omega_k}(\dot{\mathrm{E}}_{\overline{\Omega}^c}u|_{\Omega_k})|_{\overline{\Omega}_k^c}\big\rVert_{\dot H^1(\overline{\Omega}_k^c)}\le c'\big\lVert\trd^{\bord_k}_{\mathrm{e}} u_k - \trd^{\bord_k}({\dot{\mathrm{H}}}_{\bord}\trd_{\mathrm{e}}^{\bord}u)\big\rVert_{\dot \B(\bord_k)}.
\end{equation*}
The result follows (see Lemma~\ref{Lem:CV-KS}). 
\end{proof}

\begin{corollary}
Let $\Omega$ be a two-sided admissible domain and $(\Omega_k)_{k\in\N}$ be a non-decreasing sequence of two-sided admissible domains such that $\Omega_k\nearrow\Omega$. Assume that the Dirichlet harmonic extensions $(\dot{\mathrm{E}}_{\Omega_k})_{k\in\N}$ are uniformly bounded. Then, if $(u_k\in \dot H^1(\overline{\Omega}_k^c))_{k\in\N}\xrightarrow[k\to\infty]{(\mathcal{E}_k\oplus i_k)} u\in \dot H^1(\overline{\Omega}^c)$, it holds
\begin{equation*}
\Trd_{\mathrm{e}}^{\bord_k}u_k\xrightarrow[k\to\infty]{(\trd^{\bord_k}\circ\,\dot{\mathrm{H}}_{\bord})}\Trd_{\mathrm{e}}^{\bord}u.
\end{equation*}
In particular, it holds $\Trd_{\mathrm{e}}^{\bord_k}u\xrightarrow[k\to\infty]{(\trd^{\bord_k}\circ\,\dot{\mathrm{H}}_{\bord})} \Trd_{\mathrm{e}}^{\bord}u$.
\end{corollary}

\begin{proof}
Denote $u_k=v_k+\varphi_k$ and $u=v+\varphi$, where $v_k\in \dot V_0(\overline{\Omega}_k^c)$, $\varphi_k\in \dot H^1_0(\overline{\Omega}_k^c)$, $v\in \dot V_0(\overline{\Omega}^c)$ and $\varphi\in \dot H^1_0(\overline{\Omega}^c)$. Then, using the orthogonal decomposition $\dot H^1(\overline{\Omega}_k^c)=\dot V_0(\overline{\Omega}_k^c)\oplus \dot H^1_0(\overline{\Omega}_k^c)$, it holds
\begin{equation*}
\lVert v_k-(\dot{\mathrm{E}}_{\Omega_k}(\dot{\mathrm{E}}_{\overline{\Omega}^c}v|_{\Omega_k})|_{\overline{\Omega}_k^c}\rVert_{\dot H^1(\overline{\Omega}_k^c)}\le \lVert u_k-(\mathcal{E}_k\oplus i_k) u\rVert_{\dot H^1(\overline{\Omega}_k^c)}\xrightarrow[k\to\infty]{}0.
\end{equation*}
The convergence of the traces follows by Proposition~\ref{Prop:TrCV-i}.
\end{proof}

As mentioned at the beginning of this subsection, the convergence framework from Corollary~\ref{Cor:CV-H1-e} was chosen for it preserves harmonicity, yet it is not the simplest framework of which one could think. Using Proposition~\ref{Prop:TrCV-e} allows to notice both induce similar notions of convergence when the limit is harmonic.

\begin{corollary}\label{Cor:Equiv-CV-H1-e}
Let $\Omega$ be a two-sided admissible domain. If the extension operators $\dot{\mathrm{E}}_{\Omega_k}$ and $\dot{\mathrm{E}}_{\overline{\Omega}_k^c}$, $k\in\N$, are uniformly bounded, then for all $u\in \dot V_0(\overline{\Omega}^c)$, it holds
\begin{equation*}
\big\lVert\mathcal{E}_ku-(\dot{\mathrm{E}}_{\overline{\Omega}^c}u)|_{\overline{\Omega}_k^c}\big\rVert_{\dot H^1(\overline{\Omega}_k^c)}\xrightarrow[k\to\infty]{}0.
\end{equation*}
\end{corollary}

\begin{proof}
By~\eqref{Eq:TrNormEquiv}, there exist $c_1,c_2>0$ such that for all $(v_k\in \dot V_0(\overline{\Omega}_k^c))_{k\in\N}$ and all $v\in \dot V_0(\overline{\Omega}^c)$, it holds
\begin{multline*}
c_1\lVert v_k-(\dot{\mathrm{E}}_{\overline{\Omega}^c}v)|_{\overline{\Omega}_k^c}\rVert_{\dot H^1(\overline{\Omega}_k^c)}\le \lVert\trd^{\bord_k}_{\mathrm{e}}v_k-\trd^{\bord_k}\dot{\mathrm{H}}_{\bord}\trd^{\bord}_{\mathrm{e}}v\rVert_{\dot \B(\bord_k)}\\
\le c_2\lVert v_k-(\dot{\mathrm{E}}_{\overline{\Omega}^c}v)|_{\overline{\Omega}_k^c}\rVert_{\dot H^1(\overline{\Omega}_k^c)}.
\end{multline*}
Therefore, one can state a result similar to Proposition~\ref{Prop:ExtCV-e}, only with convergence through $((\dot{\mathrm{E}}_{\overline{\Omega}^c}\cdot)|_{\overline{\Omega}_k^c})_{k\in\N}$ instead of $(\mathcal{E}_k)_{k\in\N}$. The result follows by Lemma~\ref{Lem:EquivCVFramework}.
\end{proof}

From there, using the isometric properties of the normal derivation allows to state the equivalence between the convergence of harmonic functions on exterior domains and of their exterior normal derivatives.

\begin{proposition}\label{Prop:CV-ddne}
Let $\Omega$ be a two-sided admissible domain and $(\Omega_k)_{k\in\N}$ be a non-decreasing sequence of two-sided admissible domains such that $\Omega_k\nearrow\Omega$. Then, if the Dirichlet harmonic extensions $(\dot{\mathrm{E}}_{\Omega_k})_{k\in\N}$ and $(\dot{\mathrm{E}}_{\overline{\Omega}_k^c})_{k\in\N}$ are uniformly bounded, then the statements
\begin{equation*}
(v_k\in \dot V_0(\overline{\Omega}_k^c))_{k\in\N}\xrightarrow[]{(\mathcal E_k)}v\in \dot V_0(\overline{\Omega}^c),
\end{equation*}
and
\begin{equation*}
\bigg(\ddno{v_k}{\mathrm{e}}\Big|_{\bord_k}\in\dot\B'(\bord_k)\bigg)_{k\in\N}\xrightarrow[]{(\Xi_k)}\ddno v{\mathrm{e}}\Big|_{\bord}\in\Hmo,
\end{equation*}
are equivalent, using the convergence frameworks from Lemma~\ref{Lem:CV-B'} and Corollary~\ref{Cor:CV-H1-e}.
\end{proposition}

\begin{proof}
By Corollary~\ref{Cor:Equiv-CV-H1-e}, since $v\in \dot V_0(\overline{\Omega}^c)$, it holds
\begin{equation*}\label{Eq:Implication-H1CV}
(v_k\in \dot V_0(\overline{\Omega}_k^c))_{k\in\N}\xrightarrow[]{(\mathcal{E}_k)}v\in \dot V_0(\overline{\Omega}^c) \iff \lVert u_k-\dot{\mathrm E}_{\overline{\Omega}^c}u|_{\overline{\Omega}_k^c}\rVert_{\dot H^1(\overline{\Omega}_k^c)}\xrightarrow[k\to\infty]{}0.
\end{equation*}
while by Remark~\ref{Rem:B'-Change}, it holds
\begin{equation*}
\ddno{v_k}{\mathrm{e}}\Big|_{\bord_k}\xrightarrow[k\to\infty]{(\Xi_k)}\ddno v{\mathrm{e}}\Big|_{\bord}
\iff\bigg\lVert\ddno{v_k}{\mathrm{e}}\Big|_{\bord_k}-\ddno{}{\mathrm{e}}\Big|_{\bord_k}\dot{\mathrm{N}}_{\bord}\ddno v{\mathrm{e}}\Big|_{\bord}\bigg\rVert_{\dot\B'(\bord_k)}\xrightarrow[k\to\infty]{}0.
\end{equation*}
Using the isometric properties of the exterior normal derivative and the uniform boundedness of the extension operators, there exist $c, c'>0$ uniform in $k$ such that
\begin{multline*}
c\lVert v_k-(\dot{\mathrm{E}}^{\mathrm N}_{\overline{\Omega}^c}v)|_{\overline{\Omega}_k^c}\rVert_{\dot H^1(\overline{\Omega}_k^c\backslash\bord)}\le \bigg\lVert\ddno{v_k}{\mathrm{e}}\Big|_{\bord_k}-\ddno{}{\mathrm{e}}\Big|_{\bord_k}\dot{\mathrm{N}}_{\bord}\ddno v{\mathrm{e}}\Big|_{\bord}\bigg\rVert_{\dot\B'(\bord_k)}\\
\le c'\lVert v_k-(\dot{\mathrm{E}}^{\mathrm N}_{\overline{\Omega}^c}v)|_{\overline{\Omega}_k^c}\rVert_{\dot H^1(\overline{\Omega}_k^c\backslash\bord)},
\end{multline*}
where $\dot{\mathrm{E}}^{\mathrm N}_{\overline{\Omega}^c}:\dot H^1(\overline{\Omega}^c)\to \dot H^1(\nbord)$ is the Neumann harmonic extension:
\begin{equation*}
\forall u\in\dot H^1(\Omega),\quad \big(\dot{\mathrm{E}}^{\mathrm N}_{\overline{\Omega}^c}u\big)\big|_\Omega=u,\quad \big(\dot{\mathrm{E}}^{\mathrm N}_{\overline{\Omega}^c}u\big)\big|_{\overline{\Omega}^c}\in \dot V_0(\overline{\Omega}^c)\quad \mbox{and}\quad \left\llbracket\ddno{}{}\Big|_{\bord}(\dot{\mathrm{E}}^{\mathrm N}_{\overline{\Omega}^c}u)\right\rrbracket=0,
\end{equation*}
where the jump in normal derivative was defined by~\eqref{Eq:ddn-jump}.
Still by uniform boundedness of the extensions, and convergence in the sense of characteristic functions, 
\begin{equation*}
\lVert(\dot{\mathrm{E}}^{\mathrm N}_{\overline{\Omega}^c}-\dot{\mathrm{E}}_{\overline{\Omega}^c})v\rVert_{\dot H^1(\overline{\Omega}_k^c\backslash\bord)}\xrightarrow[k\to\infty]{}0,
\end{equation*}
hence the equivalence.
\end{proof}

\subsection{Convergence on the whole space}\label{Subsec:CV-Rn}

The frameworks for the convergence of $\dot H^1$ functions on interior and exterior domains from Lemma~\ref{Lem:CV-H1-i} and Corollary~\ref{Cor:CV-H1-e} induce a notion of convergence for $\dot H^1$ functions on the whole space.

\begin{lemma}\label{Lem:CV-H1-Rn}
Let $\Omega$ be a two-sided admissible domain. Let $(\Omega_k)_{k\in\N}$ be a sequence of two-sided admissible domains such that $\Omega_k\nearrow\Omega$. Then, it holds
\begin{equation*}
\dot H^1(\nbord_k)\xrightarrow[k\to\infty]{}\dot H^1(\nbord)\quad\mbox{through}\quad\big(\dot H^1(\nbord),(\cdot|_{\Omega_k}\oplus \mathcal{E}_k \oplus i_k)_{k\in\N}\big),
\end{equation*}
with $\cdot|_{\Omega_k}\oplus \mathcal{E}_k \oplus i_k$ decomposed along the direct orthogonal sum $\dot H^1(\Omega)\oplus \dot V_0(\overline{\Omega}^c)\oplus \dot H^1_0(\overline{\Omega}^c)$.
\end{lemma}

\begin{proof}
The result follows from Lemma~\ref{Lem:CV-H1-i}, Corollary~\ref{Cor:CV-H1-e}, and the fact that the decomposition is orthogonal.
\end{proof}

\begin{remark}\label{Rem:Spaces-CV}
By design, the convergence framework from Lemma~\ref{Lem:CV-H1-Rn} is compatible with those from Lemma~\ref{Lem:CV-H1-i} and Corollary~\ref{Cor:CV-H1-e}, in the sense that if
\begin{equation*}
(u_k\in \dot H^1(\nbord_k))_{k\in\N}\xrightarrow[k\to\infty]{(\cdot|_{\Omega_k}\oplus \mathcal{E}_k \oplus i_k)}u\in \dot H^1(\nbord),
\end{equation*}
then
\begin{equation*}
(u_k|_{\Omega_k}\in \dot H^1(\Omega_k))_{k\in\N}\xrightarrow[k\to\infty]{(\cdot|_{\Omega_k})}u|_\Omega\in \dot H^1(\Omega),
\end{equation*}
and
\begin{equation*}
(u_k|_{\overline{\Omega}_k^c}\in \dot H^1(\overline{\Omega}_k^c))_{k\in\N}\xrightarrow[k\to\infty]{(\mathcal{E}_k \oplus i_k)}u|_{\overline{\Omega}^c}\in \dot H^1({\overline{\Omega}^c}).
\end{equation*}
In addition, it also holds:
\begin{equation*}
\dot V_0(\nbord_k)\xrightarrow[k\to\infty]{}\dot V_0(\nbord)\quad\mbox{through}\quad\big(\dot V_0(\nbord), (\cdot|_{\Omega_k}\oplus\mathcal{E}_k)_{k\in\N}\big),
\end{equation*}
decomposed along $\dot V_0(\Omega)\oplus\dot V_0(\overline{\Omega}^c)$.
\end{remark}

In Subsections~\ref{Subsec:CV-Int} and~\ref{Subsec:CV-Ext}, we have established the link between convergence of harmonic functions inside and outside, and convergence of their interior and exterior boundary values. For the transmission problem, beyond harmonicity, it is necessary to connect the interior and exterior part in the sense that we need to consider the jumps in trace and normal derivative.
The following proposition proves the convergence of the orthogonal projectors giving the single layer potential parts of the transmission solutions (in the sense of~\eqref{Eq:Decomp-Trans} and~\eqref{Eq:SLP-Space}) along a converging sequence of domains.

\begin{proposition}\label{Prop:ProjTrCV}
Let $\Omega$ be a two-sided admissible domain. Assume there exists a non-decreasing sequence $(\Omega_k)_{k\in\N}$ of two-sided admissible domains such that $\Omega_k\nearrow\Omega$ and that the extensions $(\dot{\mathrm{E}}_{\Omega_k})_{k\in\N}$ and $(\dot{\mathrm{E}}_{\overline{\Omega}_k^c})_{k\in\N}$ are uniformly bounded. For $k\in\N$, define the orthogonal projectors $p_k^{\Scal}:\dot V_0(\R^n\backslash\bord_k)\to \dot{V}_\Scal(\R^n\backslash\bord_k)$ and $p^{\Scal}:\dot V_0(\R^n\backslash\bord)\to \dot{V}_\Scal(\R^n\backslash\bord)$. Then, it holds:
\begin{equation}\label{Eq:WCV-Proj}
p_k^{\Scal}\xrightarrow[k\to\infty]{(\cdot|_{\Omega_k}\oplus\mathcal{E}_k),(\cdot|_{\Omega_k}\oplus\mathcal{E}_k)}p^{\Scal}.
\end{equation}
\end{proposition}

\begin{proof}
For all $k\in\N$, $\lVert p_k^{\Scal}\rVert_{\mathcal{L}(\dot V_0(\nbord_k))}=1$, hence there exists a linear bounded operator $p_\infty^{\Scal}$ on $\dot V_0(\nbord)$ which is a weak subsequential limit of $(p_k^{\Scal})_{k\in\N}$.
Since all $p_k^{\Scal}$ are self-adjoint, the limit is also self-adjoint (Lemma~\ref{Lem:CV-SA}): it is a strong subsequential limit by~\cite[Lemma 2.6]{kuwae_convergence_2003}. Since all $p_k^{\Scal}$ are projectors, we can deduce $(p_\infty^{\Scal})^2=p_\infty^{\Scal}$, i.e., it is also a projector.
By Propositions~\ref{Prop:TrCV-i} and~\ref{Prop:TrCV-e}, the null jump in trace is preserved by the limit so that $p_\infty^{\Scal}(\dot V_0(\nbord))\subset \dot{V}_\Scal(\nbord)$.\\
Conversely, let $u\in \dot{V}_\Scal(\nbord)$ and denote $f:=\trd^{\bord}u$. Consider a sequence $(f_k\in\dot \B(\bord_k))$ converging to $f\in\Hpo$.
Then, by Propositions~\ref{Prop:TrCV-i} and~\ref{Prop:TrCV-e}, it holds:
\begin{equation*}
\DOd_{\Omega_k}f_k\xrightarrow[k\to\infty]{(\cdot|_{\Omega_k})}\DOd_{\Omega}f\quad\mbox{and}\quad \DOd_{\overline{\Omega}_k^c}f_k\xrightarrow[k\to\infty]{(\mathcal{E}_k)}\DOd_{\overline{\Omega}^c}f.
\end{equation*}
Consequently, it holds
\begin{equation*}
\begin{array}{ccc}
\DOd_{\Omega_k}f_k\oplus\DOd_{\overline{\Omega}_k^c}f_k &\xrightarrow[k\to\infty]{(\cdot|_{\Omega_k}\oplus\mathcal{E}_k)} & u\\
\rotatebox{90}{$\,=$} &&\\
p_k^{\Scal}(\DOd_{\Omega_k}f_k\oplus\DOd_{\overline{\Omega}_k^c}f_k) &\xrightarrow[k\to\infty]{(\cdot|_{\Omega_k}\oplus\mathcal{E}_k)} & p_\infty^{\Scal}u,
\end{array}
\end{equation*}
hence $p_\infty^{\Scal}(\dot V_0(\nbord))=\dot{V}_\Scal(\nbord)$. Therefore, $p^{\Scal}$ is the only possible subsequential limit, and the result follows. 
\end{proof}

\begin{remark}\label{Rem:LP-CV}
Proposition~\ref{Prop:ProjTrCV} states that if $(u_k\in \dot V_0(\nbord_k))_{k\in\N}\to u\in \dot V_0(\nbord)$ through $(\cdot|_{\Omega_k}\oplus\mathcal{E}_k)_{k\in\N}$, then the single (resp. double) potential part of $u_k$ converges to the single (resp. double) potential part of $u$: the decomposition $\dot V_0(\nbord)=\dot{V}_\Scal(\nbord)\oplus \dot{V}_\D(\nbord)$ is preserved.
\end{remark}

Using the convergence of the jumps in trace and normal derivative under the hypothesis of uniform boundedness of the harmonic extensions, and the uniform boundedness of the layer potential operators, we can state the convergence of the layer potential operators along a sequence of two-sided admissible domains.

\begin{theorem}[Convergence of the harmonic layer potential operators]\label{Th:CV-LP}
Let $\Omega$ be a two-sided admissible domain and $(\Omega_k)_{k\in\N}$ be a non-decreasing sequence of two-sided admissible domains such that $\Omega_k\nearrow\Omega$. Assume that the Dirichlet harmonic extensions $(\dot{\mathrm{E}}_{\Omega_k})_{k\in\N}$ and $(\dot{\mathrm{E}}_{\overline{\Omega}_k^c})_{k\in\N}$ from~\eqref{Eq:Dir-Int-Ext} and~\eqref{Eq:Dir-Ext-Ext} are uniformly bounded. Then, it holds
\begin{equation*}
\Sd_{\bord_k}\xrightarrow[k\to\infty]{(\Xi_k),(\cdot|_{\Omega_k}\oplus\mathcal{E}_k)}\Sd_{\bord}\qquad\mbox{and}\qquad\Dd_{\bord_k}\xrightarrow[k\to\infty]{(\trd^{\bord_k}\circ\,\dot{\mathrm{H}}_{\bord}),(\cdot|_{\Omega_k}\oplus\mathcal{E}_k)}\Dd_{\bord},
\end{equation*}
using the convergence frameworks from Lemmas~\ref{Lem:CV-B},~\ref{Lem:CV-B'} and~\ref{Lem:CV-H1-Rn}.
\end{theorem}

\begin{proof}
$(\Sd_{\bord_k})_{k\in\N}$ and $(\Dd_{\bord_k})_{k\in\N}$ are uniformly bounded sequences of operators by Proposition~\ref{Prop:LP-Bounds}. By the convergence of their respective inverses (i.e., the jumps in normal derivative and in trace) proved in Subsections~\ref{Subsec:CV-Int} and~\ref{Subsec:CV-Ext}, and Lemma~\ref{Lem:InverseCV}, the convergence of the layer potential operators follows.
\end{proof}

So far, the results have been stated in the framework of convergence along a sequence of Hilbert spaces introduced in~\cite{kuwae_convergence_2003}. Nonetheless, it was proved in Lemma~\ref{Lem:CV-KS} that notion of convergence can be handled with more ease when the Hilbert spaces converge in a `nice enough' manner. In the case of the convergence of the $\dot V_0$ spaces in the sense of Remark~\ref{Rem:Spaces-CV}, the convergence of vectors -- therefore, of bounded operators -- can be understood in terms of (strong) convergence in $\dot H^1(\R^n)$ of Dirichlet harmonic extensions.

\begin{proposition}\label{Prop:Link-CV-H1Rn}
Let $\Omega$ be a two-sided admissible domain. Let $(\Omega_k)_{k\in\N}$ be a sequence of two-sided admissible domains such that $\Omega_k\nearrow\Omega$. Then,
\begin{enumerate}
\item[(i)] if $(\dot{\mathrm{E}}_{\Omega_k})$ is uniformly bounded, then the following equivalence holds:
\begin{multline*}
(u_k\in \dot H^1(\Omega_k))_{k\in\N}\xrightarrow[]{(\cdot|_{\Omega_k})}u\in \dot H^1(\Omega) \\
\iff\quad \dot{\mathrm{E}}_{\Omega_k}u_k\xrightarrow[k\to\infty]{} \dot{\mathrm{E}}_\Omega u\quad\mbox{in }\dot H^1(\R^n);
\end{multline*}

\item[(ii)] if $(\dot{\mathrm{E}}_{\overline{\Omega}_k^c})$ is uniformly bounded, then the following equivalence holds:
\begin{multline*}
(v_k\in \dot V_0(\overline{\Omega}_k^c))_{k\in\N}\xrightarrow[]{(\mathcal{E}_k)}v\in \dot V_0(\overline{\Omega}^c)\\
\iff\quad \dot{\mathrm{E}}_{\overline{\Omega}_k^c}v_k\xrightarrow[k\to\infty]{} \dot{\mathrm{E}}_{\overline{\Omega}^c} v\quad\mbox{in }\dot H^1(\R^n).
\end{multline*}
\end{enumerate}
\end{proposition}

\begin{proof}
Let us prove Point (i). From the uniform boundedness of the extension operators $\dot{\mathrm{E}}_{\Omega_k}$, there exists $c>0$ such that for all $k\in\N$,
\begin{equation*}
\lVert u_k-u|_{\Omega_k}\rVert_{\dot H^1(\Omega_k)}\le \lVert\dot{\mathrm{E}}_{\Omega_k}u_k-\dot{\mathrm{E}}_{\Omega_k}(u|_{\Omega_k})\rVert_{\dot H^1(\R^n)} \le c\,\lVert u_k-u|_{\Omega_k}\rVert_{\dot H^1(\Omega_k)}.
\end{equation*}
The equivalence follows from Lemma~\ref{Lem:CV-KS} and Proposition~\ref{Prop:ExtCV-i}, Point (i).

Let us prove Point (ii). It holds
\begin{equation*}
\lVert v_k-\mathcal{E}_kv\rVert_{\dot H^1(\overline{\Omega}_k^c)} \le \lVert\dot{\mathrm{E}}_{\overline{\Omega}_k^c}v_k-\dot{\mathrm{E}}_{\Omega_k}(\dot{\mathrm{E}}_{\overline{\Omega}^c}v|_{\Omega_k})\rVert_{\dot H^1(\R^n)},
\end{equation*}
and from the uniform boundedness of the extension operators $\dot{\mathrm{E}}_{\overline{\Omega}_k^c}$, there exists $c'>0$ such that for all $k\in\N$,
\begin{equation*}
\lVert\dot{\mathrm{E}}_{\overline{\Omega}_k^c}v_k-\dot{\mathrm{E}}_{\Omega_k}(\dot{\mathrm{E}}_{\overline{\Omega}^c}v|_{\Omega_k})\rVert_{\dot H^1(\R^n)} \le c'\,\lVert v_k-\mathcal{E}_kv\rVert_{\dot H^1(\overline{\Omega}_k^c)}.
\end{equation*}
The equivalence follows from Lemma~\ref{Lem:CV-KS} and Proposition~\ref{Prop:ExtCV-i}, Point (ii).
\end{proof}

As explained above, the strengthened convergences of Proposition~\ref{Prop:Link-CV-H1Rn} allow to legitimize the choices made when introducing the convergence of the $\dot H^1$ spaces (Lemma~\ref{Lem:CV-H1-i} and Corollary~\ref{Cor:CV-H1-e}).
Together with Propositions~\ref{Prop:TrCV-i} and~\ref{Prop:CV-ddni} on the link between convergence of harmonic functions and their boundary values, it further allows to legitimize the choices made for the convergences of the trace spaces and their dual spaces: under the hypotheses of Proposition~\ref{Prop:Link-CV-H1Rn}, saying a sequence of harmonic functions $u_k$ on $\Omega_k$ converges to $u$, harmonic on $\Omega$ -- in the sense that their harmonic extensions converge in $\dot H^1(\R^n)$ -- is equivalent to saying their boundary values converge in the sense of Lemma~\ref{Lem:CV-B} and Lemma~\ref{Lem:CV-B'}.

As a consequence of Proposition~\ref{Prop:Link-CV-H1Rn}, Remark~\ref{Rem:LP-CV} can be restated to prove the layer potentials across $\bord$ can be approximated by sequences of layer potentials across $(\bord_k)_{k\in\N}$ in terms of convergence in $\dot H^1(\R^n)$.

\begin{theorem}[Convergence of the harmonic layer potential operators in $\dot H^1(\R^n)$]\label{Th:CV-LP-H1-Rn}
Let $\Omega$ be a two-sided admissible domain. Assume there exists a non-decreasing sequence $(\Omega_k)_{k\in\N}$ of two-sided admissible domains such that $\Omega_k\nearrow\Omega$ and that the extension operators $(\dot{\mathrm{E}}_{\Omega_k})_{k\in\N}$ and $(\dot{\mathrm{E}}_{\overline{\Omega}_k^c})_{k\in\N}$ from~\eqref{Eq:Dir-Int-Ext} and~\eqref{Eq:Dir-Ext-Ext} are uniformly bounded. Then,
\begin{enumerate}
\item[(i)] for every $(f_k\in\dot \B(\bord_k))_{k\in\N}\xrightarrow[k\to\infty]{(\trd^{\bord_k}\circ\,\dot{\mathrm{H}}_{\bord})} f\in\Hpo$, it holds
\begin{align*}
&& \dot{\mathrm{E}}_{\Omega_k}(\Dd_{\bord_k}f_k|_{\Omega_k})&\xrightarrow[k\to\infty]{}\dot{\mathrm{E}}_\Omega(\Dd_{\bord}f|_\Omega) &&\mbox{in }\dot H^1(\R^n),\\
&\text{and} &
\dot{\mathrm{E}}_{\overline{\Omega}_k^c}(\Dd_{\bord_k}f_k|_{\overline{\Omega}_k^c})&\xrightarrow[k\to\infty]{}\dot{\mathrm{E}}_{\overline{\Omega}^c}(\Dd_{\bord}f|_{\overline{\Omega}^c}) &&\mbox{in }\dot H^1(\R^n);
\end{align*}

\item[(ii)] for every $(g_k\in\dot\B'(\bord_k))_{k\in\N}\xrightarrow[k\to\infty]{(\Xi_k)} g\in\Hmo$, it holds
\begin{equation*}
\Sd_{\bord_k}g_k\xrightarrow[k\to\infty]{}\Sd_{\bord}g\qquad\mbox{in }\dot H^1(\R^n).
\end{equation*}
\end{enumerate}
\end{theorem}

A case in which the uniform boundedness of the extension operators in Proposition~\ref{Prop:Link-CV-H1Rn} and Theorem~\ref{Th:CV-LP-H1-Rn} is satisfied is that of $(\varepsilon,\infty)$-domains (or $(\varepsilon,\delta)$ in general), see~\cite{jones_quasiconformal_1981,rogers_degree-independent_2006}. Since the inverses of the layer potential operators, given by the jumps in trace and normal derivative, converge (this is actually how the convergence of the layer potentials was established in Theorem~\ref{Th:CV-LP}), the converses of the statements from~\ref{Th:CV-LP-H1-Rn} hold as well.

\subsection{Convergence of Neumann-Poincar\'e operators and Calder\'on projectors}\label{Subsec:CV-K-Calderon}

Having proved the convergence of the layer potential operators in Theorem~\ref{Th:CV-LP}, we turn to that of their boundary values. We recall the definition of the Neumann-Poincaré operator from~\cite{claret_layer_2025} in the case of two-sided admissible domains.

If $\Omega$ is a two-sided admissible domain, define the Neumann-Poincaré operator associated to~\eqref{Eq:Tr-Prob} as
\begin{equation*}
\Kd_{\bord}:=\frac12(\trd_{\mathrm i}^{\bord}+\trd_{\mathrm e}^{\bord})\circ\Dd_{\bord}:\Hpo\to\Hpo.
\end{equation*}
It is linear and bounded, and the traces of the double layer potential operator can be expressed as follows:
\begin{equation*}
\trd_{\mathrm i}^{\bord}\circ\Dd_{\bord}=-\frac12I+\Kd_{\bord}\qquad\mbox{and}\qquad\trd_{\mathrm e}^{\bord}\circ\Dd_{\bord}=\frac12I+\Kd_{\bord}.
\end{equation*}
In the case of a Lipschitz domain, it can be rewritten in terms of an integral over the boundary with respect to the surface measure $\sigma$~\cite{fabes_potential_1978, verchota_layer_1984}, namely, for $f\in\Hpo$,
\begin{equation*}
\Kd_{\bord}f(x)=\frac1{\omega_n}\int_{\bord}\frac{\langle y-x,\nu_y\rangle}{\lvert y-x\rvert^d}f(y)\sigma(\dy),\quad x\in\bord,
\end{equation*}
where $\omega_n$ is the surface area of the unit sphere in $\R^n$ and $\nu_y$ is the outward normal vector at $y\in\bord$.

The adjoint of the Neumann-Poincaré operator can be expressed in terms of the single layer potential operator as
\begin{equation*}
\Kd_{\bord}^*=\frac12\Big(\ddno{}{\mathrm i}\Big|_{\bord}+\ddno{}{\mathrm e}\Big|_{\bord}\Big)\circ\Sd_{\bord}:\Hmo\to\Hmo,
\end{equation*}
so that
\begin{equation*}
\ddno{}{\mathrm i}\Big|_{\bord}\circ\Sd_{\bord}=\frac12I+\Kd_{\bord}\qquad\mbox{and}\qquad\ddno{}{\mathrm i}\Big|_{\bord}\circ\Sd_{\bord}=-\frac12I+\Kd_{\bord}^*.
\end{equation*}

Considering the remaining boundary values of the layer potentials, namely the trace of the single layer potential and the normal derivative of the double layer potential
\begin{align*}
\Vd_{\bord}&:=\trd^{\bord}\circ\Sd_{\bord}:\Hmo\to\Hpo,\\
\Wd_{\bord}&:=-\ddno{}{}\Big|_{\bord}\circ\Dd_{\bord}:\Hpo\to\Hmo,
\end{align*}
allows to construct the Calder\'on projectors associated to~\eqref{Eq:Tr-Prob}:
\begin{equation*}
\dot{\mathcal C}^{\bord}_{\mathrm i}:=\frac12I+\dot{\mathcal{M}}_{\bord}\qquad\mbox{and}\qquad\dot{\mathcal C}^{\bord}_{\mathrm e}:=-\frac12I+\dot{\mathcal{M}}_{\bord},
\end{equation*}
where
\begin{equation*}
\dot{\mathcal{M}}_{\bord}:=
\begin{pmatrix}
-\Kd_{\bord} & \Vd_{\bord} \\
\Wd_{\bord} & \Kd_{\bord}^*
\end{pmatrix}:\Hpo\times\Hmo\to\Hpo\times\Hmo.
\end{equation*}
Those operators allow to recover the boundary values of a harmonic transmission solution from its boundary jumps: if $-\Delta u=0$ weakly on $\nbord$, then
\begin{equation*}
\dot{\mathcal C}^{\bord}_{\mathrm{i,e}}
\begin{pmatrix}
-\llbracket\trd^{\bord}u\rrbracket\\
\Big\llbracket\ddno u{}\Big|_{\bord}\Big\rrbracket
\end{pmatrix}
=
\begin{pmatrix}
\trd^{\bord}_{\mathrm{i,e}}\\\ddno u{\mathrm{i,e}}\Big|_{\bord}
\end{pmatrix}.
\end{equation*}
We refer to~\cite{claret_layer_2025} regarding properties of the operators described above on a two-sided admissible domain. It follows from the convergence of the layer potential operators from Theorem~\ref{Th:CV-LP} that the Calder\'on projectors converge under the same hypotheses.

\begin{corollary}\label{Cor:K-CV}
Let $\Omega$ be a two-sided admissible domain and $(\Omega_k)_{k\in\N}$ be a non-decreasing sequence of two-sided admissible domains such that $\Omega_k\nearrow\Omega$. Assume that the Dirichlet harmonic extensions $(\dot{\mathrm{E}}_{\Omega_k})_{k\in\N}$ and $(\dot{\mathrm{E}}_{\overline{\Omega}_k^c})_{k\in\N}$ are uniformly bounded. Then, it holds
\begin{equation*}
\Kd_{\bord_k}\xrightarrow[k\to\infty]{(\trd^{\bord_k}\circ\,\dot{\mathrm{H}}_{\bord}),(\trd^{\bord_k}\circ\,\dot{\mathrm{H}}_{\bord})}\Kd_{\bord}\qquad\mbox{and}\qquad\Kd^*_{\bord_k}\xrightarrow[k\to\infty]{(\Xi_k),(\Xi_k)}\Kd^*_{\bord}.
\end{equation*}
In addition, it holds
\begin{equation*}
\dot{\mathcal{V}}_{\bord_k}\xrightarrow[k\to\infty]{(\Xi_k),(\trd^{\bord_k}\circ\,\dot{\mathrm{H}}_{\bord})}\dot{\mathcal{V}}_{\bord}\qquad\mbox{and}\qquad\dot{\mathcal{W}}_{\bord_k}\xrightarrow[k\to\infty]{(\trd^{\bord_k}\circ\,\dot{\mathrm{H}}_{\bord}),(\Xi_k)}\dot{\mathcal{W}}_{\bord},
\end{equation*}
so that the Calder\'on projectors converge
\begin{equation*}
\dot{\mathcal C}^{\bord_k}_{\mathrm i}\xrightarrow[k\to\infty]{}\dot{\mathcal C}^{\bord}_{\mathrm i}\qquad\mbox{and}\qquad\dot{\mathcal C}^{\bord_k}_{\mathrm e}\xrightarrow[k\to\infty]{}\dot{\mathcal C}^{\bord}_{\mathrm e},
\end{equation*}
in the sense that their components converge in the frameworks specified above.
\end{corollary}

\begin{proof}
This follows from the convergence of the layer potential operators (Theorem~\ref{Th:CV-LP}), the trace operators (Propositions~\ref{Prop:TrCV-i} and~\ref{Prop:TrCV-e}) and the normal derivation operators (Propositions~\ref{Prop:CV-ddni} and~\ref{Prop:CV-ddne}).
\end{proof}

\subsection{Convergence of Neumann series for the Neumann-Poincaré operators}\label{Subsec:CV-Neumann-Series}

The convergence of the Neumann-Poincaré operators can be used to prove that of the associated Neumann series. Those provide a formula to recover the jumps in trace (resp. normal derivative) from boundary values of a double (resp. single) layer potential.

If $\Omega$ is a two-sided admissible domain, it was proved in~\cite[Theorem 6.11]{claret_layer_2025} that the operators $(\pm\frac12I+\Kd_{\bord})$ are strictly contractive on $\Hpo$ for the norm induced by the operator $\Vd_{\bord}^{-1}=\llbracket\ddno{}{}|_{\bord}\rrbracket\circ\dot{\mathrm{H}}_{\bord}:\Hpo\to\Hmo$, namely
\begin{equation}\label{Eq:Vnorm}
\lVert\cdot\rVert_{\Vd_{\bord}^{-1}}^2:=\langle \Vd_{\bord}^{-1}\cdot,\cdot\rangle_{\Hmo,\Hpo},
\end{equation}
which is equivalent to our usual norm $\lVert\cdot\rVert_{\Hpo}$ by Remark~\ref{rem-equivalent} below.

\begin{theorem}\label{Th:NP-contraction}
Let $\Omega$ be a two-sided admissible domain. There exists a constant $c\in]0,1[$ such that for all $f\in\Hpo$, it holds
\begin{equation*}
(1-c)\lVert f\rVert_{\Vd_{\bord}^{-1}}\le\Big\lVert \pm\Big(\frac12I+\Kd_{\bord}\Big)f\Big\rVert_{\Vd_{\bord}^{-1}}\le c\lVert f\rVert_{\Vd_{\bord}^{-1}}.
\end{equation*}
Moreover, the constant $c$ depends only on the norms of the interior and exterior Dirichlet harmonic extensions $ \dot{\mathrm{E}}_{\Omega_k} $ and $\dot{\mathrm{E}}_{\overline{\Omega}_k^c} $, and otherwise is independent of $\Omega$.
\end{theorem}
\begin{proof}
The proof follows from Proposition~\ref{prop-equivalent} and~\cite[Theorems 6.10 (iii) and 6.11 (iii)]{claret_layer_2025}.  
\end{proof}

The convergence of the associated Neumann series follows.

\begin{corollary}\label{Cor:NS-UCV}
If $ \Omega $ is a two-sided admissible domains, then the   associated Neumann series  
\begin{equation*}
\sum_{\ell=0}^{+\infty}\Big(\pm\frac12I+\Kd_{\bord} \Big)^\ell 
\end{equation*}
converge uniformly on bounded sets with respect to the norm $\lVert \cdot\rVert_{\Vd_{\bord}^{-1}}$.

 If $(\Omega_k)_{k\in\N}$ are two-sided admissible domains such that the Dirichlet harmonic extensions $(\dot{\mathrm{E}}_{\Omega_k})_{k\in\N}$ and $(\dot{\mathrm{E}}_{\overline{\Omega}_k^c})_{k\in\N}$ are uniformly bounded, then the   associated Neumann series  
\begin{equation*}
\sum_{\ell=0}^{+\infty}\Big(\pm\frac12I+\Kd_{\bord_k}\Big)^\ell\ 
\end{equation*}
  have uniform remainder estimates, 
that is, 
they satisfy the Weierstrass M-test uniformly in 
$k$ with respect to the norms $\lVert \cdot\rVert_{\Vd_{\bord_k}^{-1}}$.
 
\end{corollary}

By Corollary~\ref{Cor:NS-UCV}, the Neumann series associated to $(\pm\frac12I+\Kd_{\bord})$ converge in $\Hpo$. For instance, it holds
\begin{equation*}
\Big(\frac{1}{2}I-\Kd_{\bord}\Big)f=f_{\mathrm i}\quad\iff\quad  f=\sum_{\ell=0}^{+\infty} \Big(\frac{1}{2}I+\Kd_{\bord}\Big)^\ell f_{\mathrm i},
\end{equation*}
which allows to recover the jump in trace $-f\in\Hpo$ of a harmonic transmission solution on $\nbord$ with no jump in normal derivative from its interior trace $f_{\mathrm i}\in\Hpo$.

We wish to use the convergence of the Neumann-Poincaré operators from Corollary~\ref{Cor:K-CV} to deduce the convergence of Neumann series along a converging sequence of domains.
However, that convergence takes place in the framework of Lemma~\ref{Lem:CV-B}, where the trace spaces were endowed with the interior trace norm from Theorem~\ref{Th:Trace}.
For that matter, it is necessary to ensure convergences similar to Corollary~\ref{Cor:K-CV} hold when the trace spaces are endowed with the norms defined by~\eqref{Eq:Vnorm}.

Consider the Hilbert space $\dot{\B}_{\Vd^{-1}}(\bord):=(\Hpo,\lVert\cdot\rVert_{\Vd_{\bord}^{-1}})$.
For $f\in\dot{\B}_{\Vd^{-1}}(\bord)$, it holds
\begin{align}\label{Eq:Vnorm-Isom}
\lVert f\rVert_{\Vd_{\bord}^{-1}}^2 &=\langle \Vd_{\bord}^{-1}f,f\rangle_{\Hmo,\Hpo} \nonumber\\
&= \Big\langle \Big\llbracket \ddno{}{}\Big|_{\bord}\dot{\mathrm{H}}_{\bord}f\Big\rrbracket,f\Big\rangle_{\Hmo,\Hpo} \nonumber\\
&= \lVert\dot{\mathrm{H}}_{\bord}f\rVert_{\dot H^1(\R^n)}^2
\end{align}
by Green's formula, so that $\dot{\mathrm{H}}_{\bord}:\dot{\B}_{\Vd^{-1}}(\bord)\to \dot H^1(\R^n)$ defines an isometry.

\begin{proposition}\label{Prop:CV-B-Vnorm}
Let $\Omega$ be a two-sided admissible domain. Let $(\Omega_k)_{k\in\N}$ be a non-decreasing sequence of two-sided admissible domains such that $\Omega_k\nearrow\Omega$. Then, it holds:
\begin{equation*}
\dot \B_{\Vd^{-1}}(\bord_k)\xrightarrow[k\to\infty]{}\dot\B_{\Vd^{-1}}(\bord)\quad \mbox{through }\big(\dot\B_{\Vd^{-1}}(\bord),(\trd^{\bord_k}\circ\,\dot{\mathrm{H}}_{\bord})_{k\in\N}\big).
\end{equation*}
\end{proposition}

\begin{proof}
Let $f\in\dot\B_{\Vd^{-1}}(\bord)$. By Proposition~\ref{Prop:ExtCV-i} for $u=(\dot{\mathrm H}_{\bord}f)|_\Omega$, it holds
\begin{equation*}
\lVert \dot{\mathrm{H}}_{\bord_k}\trd^{\bord_k}\dot{\mathrm{H}}_{\bord}f\rVert_{\dot H^1(\R^n)}\xrightarrow[k\to\infty]{}\lVert\dot{\mathrm{H}}_{\bord}f\rVert_{\dot H^1(\R^n)}.
\end{equation*}
The result follows by~\eqref{Eq:Vnorm-Isom}.
\end{proof}

Given the definition of the convergence of Hilbert spaces (Definition~\ref{Def:CV-Hilb}), it appears that replacing the norms on the spaces with equivalent norms can cause the convergence to fail if the associated sequence of operators is not adjusted accordingly.
For example, if $(H_k,\lVert\cdot\rVert_{H_k})_{k\in\N}\to (H,\lVert\cdot\rVert_H)$ through $(H,(\mathcal{T}_k)_{k\in\N})$, then $(H_k,\lVert\cdot\rVert_{H_k})_{k\in\N}\to (H,2\lVert\cdot\rVert_H)$ through $(H,(2\mathcal{T}_k)_{k\in\N})$ instead.
In the case of the convergence of the trace spaces from Lemma~\ref{Lem:CV-B} and Proposition~\ref{Prop:CV-B-Vnorm}, both convergences hold through the same sequence of operators: $(\trd^{\bord_k}\circ\dot{\mathrm{H}}_{\bord})_{k\in\N}$.
In addition, we will prove in Section~\ref{Sec:EquivNorm} that, in the case of a two-sided admissible domain $\Omega$, the norms $\lVert\cdot\rVert_{\Hpo}$ and $\lVert\cdot\rVert_{\Vd(\bord)^{-1}}$ are equivalent with constants only depending on the norms of the interior and exterior harmonic extension operators (see~\eqref{Eq:TraceNorm-Rn}).
It follows that if $\Omega_k\nearrow\Omega$ and the harmonic extension operators $(\dot {\mathrm{E}}_{\Omega_k})_{k\in\N}$ and $(\dot{\mathrm{E}}_{\overline{\Omega}_k^c})_{k\in\N}$ are uniformly bounded, then the convergences from Lemma~\ref{Lem:CV-B} and Proposition~\ref{Prop:CV-B-Vnorm} are equivalent.

\begin{proposition}\label{Prop:CV-B-Frameworks-Equiv}
Let $\Omega$ be a two-sided admissible domain and $(\Omega_k)_{k\in\N}$ be a non-decreasing sequence of two-sided admissible domains such that $\Omega_k\nearrow\Omega$. Assume that the Dirichlet harmonic extensions $(\dot{\mathrm{E}}_{\Omega_k})_{k\in\N}$ and $(\dot{\mathrm{E}}_{\overline{\Omega}_k^c})_{k\in\N}$ are uniformly bounded. Then, for all $(f_k\in\dot\B(\bord_k))_{k\in\N}$ and $f\in\Hpo$, it holds
\begin{multline*}
(f_k\in\dot\B(\bord_k))_{k\in\N}\xrightarrow[]{(\trd^{\bord_k}\circ\dot{\mathrm{H}}_{\bord})}f\in\Hpo\\
\iff\quad (f_k\in\dot\B_{\Vd^{-1}}(\bord_k))_{k\in\N}\xrightarrow[]{(\trd^{\bord_k}\circ\dot{\mathrm{H}}_{\bord})}f\in\dot\B_{\Vd^{-1}}(\bord).
\end{multline*}
\end{proposition}

\begin{proof}
By~\eqref{Eq:TraceNorm-Rn} below and uniform boundedness of the harmonic extension operators, there exist constants $c_1,c_2>0$ such that for all $k\in\N$, it holds
\begin{equation*}
c_1\lVert\cdot\rVert_{\Vd_{\bord_k}^{-1}}\le\lVert\cdot\rVert_{\dot\B(\bord_k)}\le c_2\lVert\cdot\rVert_{\Vd_{\bord_k}^{-1}}.
\end{equation*}
The result follows by Lemma~\ref{Lem:CV-KS}.
\end{proof}

It follows from Proposition~\ref{Prop:CV-B-Frameworks-Equiv} that all the convergence results for elements in the trace spaces still hold when the trace spaces are endowed with the norm~\eqref{Eq:Vnorm} instead. In particular, the convergence
\begin{equation}\label{Eq:CV-1/2I+K}
\frac12I+\Kd_{\bord_k}\xrightarrow[k\to\infty]{(\trd^{\bord_k}\circ\,\dot{\mathrm{H}}_{\bord}),(\trd^{\bord_k}\circ\,\dot{\mathrm{H}}_{\bord})}\frac12I+\Kd_{\bord},
\end{equation}
can be understood as a convergence of operators acting on $\dot{\B}_{\Vd^{-1}}$ spaces, which are then contractive. As mentioned above, this means the associated Neumann series converge, and~\eqref{Eq:CV-1/2I+K} implies a term by term convergence of the series:
\begin{equation*}
\forall\ell\in\N,\quad \Big(\frac12I+\Kd_{\bord_k}\Big)^\ell\xrightarrow[k\to\infty]{(\trd^{\bord_k}\circ\,\dot{\mathrm{H}}_{\bord}),(\trd^{\bord_k}\circ\,\dot{\mathrm{H}}_{\bord})}\Big(\frac12I+\Kd_{\bord}\Big)^\ell.
\end{equation*}
We prove the convergence of the series as a whole.

\begin{theorem}[Convergence of the Neumann series]\label{Th:CV-Neumann-Series}
Let $\Omega$ be a two-sided admissible domain and $(\Omega_k)_{k\in\N}$ be a non-decreasing sequence of two-sided admissible domains such that $\Omega_k\nearrow\Omega$. Assume that the Dirichlet harmonic extensions $(\dot{\mathrm{E}}_{\Omega_k})_{k\in\N}$ and $(\dot{\mathrm{E}}_{\overline{\Omega}_k^c})_{k\in\N}$ from~\eqref{Eq:Dir-Int-Ext} and~\eqref{Eq:Dir-Ext-Ext} are uniformly bounded. Then,
\begin{equation*}
\sum_{\ell=0}^{+\infty}\Big(\frac12I+\Kd_{\bord_k}\Big)^\ell\xrightarrow[k\to\infty]{(\trd^{\bord_k}\circ\,\dot{\mathrm{H}}_{\bord}),(\trd^{\bord_k}\circ\,\dot{\mathrm{H}}_{\bord})}\sum_{\ell=0}^{+\infty}\Big(\frac12I+\Kd_{\bord}\Big)^\ell,
\end{equation*}
using the convergence framework from Lemma~\ref{Lem:CV-B}.
Moreover, those series have uniform remainder estimates, that is, they satisfy the Weierstrass M-test uniformly in $k$.
\end{theorem}

\begin{proof}
For $k\in\N$, since $\frac12I+\Kd_{\bord_k}$ is contractive on $\dot{\B}_{\Vd^{-1}}(\bord_k)$, it holds
\begin{align*}
\sum_{\ell=0}^{+\infty}\Big(\frac12I+\Kd_{\bord_k}\Big)^\ell &=\Big(\frac12I-\Kd_{\bord_k}\Big)^{-1}\\
&=-(\trd^{\bord_k}_{\mathrm i}\circ\Dd_{\bord_k})^{-1}\\
&=\llbracket\trd^{\bord_k}\cdot\rrbracket\circ\dot{\mathrm N}_{\bord_k},
\end{align*}
where $\dot{\mathrm N}_{\bord_k}:\dot{\B}'(\bord_k)\to\dot V_{\D}(\nbord_k)$ is the Neumann harmonic extension (as in Remark~\ref{Rem:B'-Change}). By Theorem~\ref{Th:Trace} and Proposition~\ref{Prop:Norm-B-Equiv}, it holds
\begin{equation*}
\big\lVert\llbracket\trd^{\bord_k}\cdot\rrbracket\big\rVert_{\mathcal{L}(\dot H^1(\nbord_k),\dot{\B}(\bord_k)}\le 1+\sqrt{\lVert\dot{\mathrm{E}}_{\overline{\Omega}_k^c}\rVert^2-1},
\end{equation*}
while noticing that the Neumann harmonic extension $\dot{\mathrm N}_{\bord_k}$ can be expressed in terms of Dirichlet harmonic extensions using the Poincaré-Steklov operators for $-\Delta$ on $\Omega_k$ and $\overline{\Omega}_k^c$ yields
\begin{equation*}
\lVert\dot{\mathrm N}_{\bord_k}\rVert_{\mathcal{L}(\dot{\B}'(\bord_k),\dot V_{\D}(\nbord_k))}\le c
\end{equation*}
where $c>0$ is a constant depending only on the norms of the extensions $\dot{\mathrm{E}}_{\Omega_k}$ and $\dot{\mathrm{E}}_{\overline{\Omega}_k^c}$. In particular, we may choose $c$ to be uniform in $k\in\N$, so that
\begin{equation*}
\frac12I-\Kd_{\bord_k}\xrightarrow[k\to\infty]{(\trd^{\bord_k}\circ\,\dot{\mathrm{H}}_{\bord}),(\trd^{\bord_k}\circ\,\dot{\mathrm{H}}_{\bord})}\frac12I-\Kd_{\bord},
\end{equation*}
together with Lemma~\ref{Lem:InverseCV} yields the result.
\end{proof}

\section{Application to the Riemann-Hilbert problem on admissible domains}\label{Sec:HT-CI}

As an application of the convergence of the layer potential operators from Theorem~\ref{Th:CV-LP-H1-Rn}, we define the notion of Cauchy integral in $\C$ beyond the Lipschitz case.

\subsection{Cauchy integral, Hilbert transform and layer potentials for Lipschitz domains}\label{Subsec:HT-LP-Lip}

In this expository section we recall the classical framework related to the Cauchy integral in the case when $\Omega$ is a Lipschitz domain in $\C$, which can be identified with $\R^2$ as usual.

We recall the definitions of the following complex Hilbert spaces:
\begin{equation*}
\dot H^1_\C(\Omega) :=\dot H^1(\Omega)+i\dot H^1(\Omega),
\end{equation*}
and the spaces $\dot H^1_\C(\cbord)$, $\dot H_\C^{\frac{1}{2}}(\bord)$ and $\dot H_\C^{-\frac{1}{2}}(\bord)$ (the fractional Sobolev space modulo constants, see~\cite[Chapter IV, Appendix]{DAUTRAY-LIONS1988}, and its dual) in the same way.
In a complex framework, we can extend the definitions of the double layer potential operator for the transmission problem~\eqref{Eq:Tr-Prob} across a Lipschitz boundary, given in~\cite{ammari_reconstruction_2004} for instance:
\begin{align*}
\Dd_{\bord} : \dot H_\C^{\frac{1}{2}}(\bord) &\longrightarrow \dot H^1_\C(\cbord)\\
f_{\mathrm{R}}+if_{\mathrm{I}} &\longmapsto \Dd_{\bord}f_{\mathrm{R}} + i\Dd_{\bord}f_{\mathrm{I}}
\end{align*}
for the double layer potential, and similarly for the single layer potential operator $\Sd_{\bord}:\dot H_\C^{\frac{1}{2}}(\bord)\to \dot H^1_\C(\R^n)$ and Neumann-Poincaré operator $\Kd_{\bord}:\dot H_\C^{\frac{1}{2}}(\bord)\to \dot H_\C^{\frac{1}{2}}(\bord)$. Note that the definitions of those complex-valued operators displays their algebraic decompositions. 

Let $f\in \dot H_\C^{\frac{1}{2}}(\bord)$. The Cauchy integral of $f$ over $\bord$ counter-clockwise is defined by:
\begin{equation*}
\Phi_{\bord}f(z):=\frac{1}{2i\pi}\int_{\bord}{\frac{f(\zeta)}{\zeta-z}\,\mathrm{d}\zeta},\quad z\in\C.
\end{equation*}
It is holomorphic on $\C\backslash\bord$~\cite[Section 15]{muskhelishvili_singular_1977}. Moreover, it is connected to the Neumann-Poincaré operator as in~\cite[Section 6]{khavinson_poincares_2007}:
\begin{equation}\label{LinkNP-CI}
\Kd_{\bord}f(z)=\int_{\bord}f(\zeta)\mathrm{Re}\bigg(\frac{\mathrm{d}\zeta}{2i\pi(\zeta-z)}\bigg),\quad z\in\bord.
\end{equation}

If $f_{\mathrm{R}}\in \dot H^{\frac{1}{2}}(\bord)$ is $\R$-valued, then~\eqref{LinkNP-CI} yields:
\begin{equation}\label{ReNP-CI}
\Kd_{\bord}f_{\mathrm{R}}(z)=\mathrm{Re}(\Phi_{\bord}f_{\mathrm{R}}(z)),\quad z\in\bord.
\end{equation}
Since $\Phi_{\bord}f_{\mathrm{R}}$ is holomorphic on $\C\backslash\bord$, its real part is harmonic and by~\eqref{ReNP-CI}, $\Phi_{\bord}f_{\mathrm{R}}$ can be written as:
\begin{equation}\label{CI-HarmoConj}
\Phi_{\bord}f_{\mathrm{R}}(z)=-\Dd_{\bord}f_{\mathrm{R}}(z)+ih_{\mathrm{R}}(z),\quad z\in\C,
\end{equation}
where $h_{\mathrm{R}}$ is the harmonic conjugate~\cite[p. 294]{rudin_real_1987} of $-\Dd_{\bord}f_{\mathrm{R}}$. By the Plemelj formulae~\cite[Section 17]{muskhelishvili_singular_1977}, it holds $\llbracket\Trd^{\bord}\Phi_{\bord}f_{\mathrm{R}}\rrbracket=f_{\mathrm{R}}$, which implies $h_{\mathrm{R}}$, which is harmonic, can be expressed as a single layer potential.

If $if_{\mathrm{I}}\in i\dot H^{\frac{1}{2}}(\bord)$ is $i\R$-valued, then~\eqref{LinkNP-CI} yields:
\begin{equation}\label{ImNP-CI}
\Kd_{\bord}(if_{\mathrm{I}})(z)=\mathrm{Im}(\Phi_{\bord}(if_{\mathrm{I}})(z)),\quad z\in\bord.
\end{equation}
As before, $\Phi_{\bord}(if_{\mathrm{I}})$ can be written as:
\begin{equation*}
\Phi_{\bord}(if_{\mathrm{I}})(z)=\Sd_{\bord}g_{\mathrm{I}}(z)-i\Dd_{\bord}(if_{\mathrm{I}})(z),\quad z\in\C,
\end{equation*}
where $\Sd_{\bord}g_{\mathrm{I}}$ is the anti-harmonic conjugate of $-\Dd_{\bord}(if_{\mathrm{I}})$.

Those considerations yield the following decomposition.

\begin{proposition}\label{Prop:Link-CI-LP}
Let $\Omega$ be a Lipschitz domain in $\C$. Let $f\in \dot H_\C^{\frac{1}{2}}(\bord)$.
\begin{enumerate}
\item[(i)] If $f$ is real-valued, then $\mathrm{Re}(\Phi_{\bord}f)$ is a double layer potential and $\mathrm{Im}(\Phi_{\bord}f)$ is a single layer potential;

\item[(ii)] If $f$ is purely imaginary-valued, then $\mathrm{Re}(\Phi_{\bord}f)$ is a single layer potential and $\mathrm{Im}(\Phi_{\bord}f)$ is a double layer potential.
\end{enumerate}
\end{proposition}

\subsection{Cauchy integral and Hilbert transform for two-sided admissible domain}\label{Subsec:CI-Ext}

We wish to use the connection between the Cauchy integral and the Neumann-Poincaré operator given by equation~\eqref{LinkNP-CI}, and the consequent Proposition~\ref{Prop:Link-CI-LP} to give a definition on irregular domains.
However, the Plemelj formulae~\cite[Section 17]{muskhelishvili_singular_1977}, which allow to identify the harmonic conjugate in equation~\eqref{CI-HarmoConj} as a single layer potential in the Lipschitz case, are proved on Lipschitz boundaries, relying on geometrical considerations.
At this point, we give the following definition of the Cauchy integral on a two-sided admissible domain, inspired by Proposition~\ref{Prop:Link-CI-LP}.

\begin{definition}[Cauchy integral on two-sided admissible domains of $\C$]\label{Def:CI}
Let $\Omega$ be a two-sided admissible domain of $\C$. For $f=f_{\mathrm{R}}+if_{\mathrm{I}}\in\dot{\B}_\C(\bord)$, we define the Cauchy integral on $\bord$ as:
\begin{equation*}
\begin{cases}
\Phi_{\bord}f_{\mathrm{R}}:=-\Dd_{\bord}f_{\mathrm{R}}+i\Sd_{\bord}g_{\mathrm{R}},\\
\Phi_{\bord}(if_{\mathrm{I}})=\Sd_{\bord}g_{\mathrm{I}}-i\Dd_{\bord}(if_{\mathrm{I}}),
\end{cases}
\end{equation*}
where $\Sd_{\bord}g_{\mathrm{R}}$ is its single layer potential part of the harmonic conjugate of $-\Dd_{\bord}f_{\mathrm{R}}$ and $\Sd_{\bord}g_{\mathrm{I}}$ is the single layer potential part of the anti-harmonic conjugate of $-\Dd_{\bord}(if_{\mathrm{I}})$.
\end{definition}

Note that, although the Cauchy integral defined in this way can be decomposed using the layer potential operators as in the Lipschitz case, it is not defined as a holomorphic function.
In this part, we prove that when the two-sided admissible domain we consider can be `suitably' approximated by Lipschitz domain, the generalized Cauchy integral from Definition~\ref{Def:CI} is holomorphic away from the boundary.
More specifically, we prove that, in that case, the generalized Cauchy integral can be expressed as the limit of a sequence of Cauchy integrals for Lipschitz domains, in a sense which preserves holomorphism.
To do so, we make use of the convergence framework from Section~\ref{Sec:CV-LP}, which we adapt to the complex-valued case, in the sense of the following lemma.

\begin{lemma}
Let $(H_k)_{k\in\N}$ be a sequence of Hilbert spaces composed of $\R$-valued functions or distributions, and $H$ be a such Hilbert space such that $H_k\to H$ through $(\mathcal{T}_k)_{k\in\N}$. Then it holds
\begin{equation*}
H_k+i H_k\xrightarrow[k\to\infty]{} H+iH\quad\mbox{through }(H+iH,(\mathcal{T}_k)_{k\in\N}),
\end{equation*}
where $\mathcal{T}_k$ is understood as $\mathcal{T}_k(u+iv)=\mathcal{T}_ku+i\mathcal{T}_kv$ for all $u,v\in H$.
\end{lemma}

Given the definition of generalized Cauchy integral from Definition~\ref{Def:CI}, it appears that the harmonic conjugation operators play a key role; the following proposition proves the convergence of the interior harmonic conjugations along a non-decreasing sequence of two-sided admissible domains.

\begin{proposition}\label{Prop:HarmoConjCV-i}
Let $\Omega$ be a two-sided admissible domain of $\C$ and let $\Omega_k\nearrow\Omega$. For $k\in\N$, let $\mathrm{HC}_{\Omega_k}:\dot V_0(\Omega_k)\to \dot V_0(\Omega_k)$ be the harmonic conjugation operator on $\Omega_k$, and let $\mathrm{HC}_\Omega:\dot V_0(\Omega)\to \dot V_0(\Omega)$ be that on $\Omega$. Then, it holds
\begin{equation*}
\mathrm{HC}_{\Omega_k}\xrightarrow[k\to\infty]{(\cdot|_{\Omega_k}),(\cdot|_{\Omega_k})} \mathrm{HC}_{\Omega}.
\end{equation*}
\end{proposition}

\begin{proof}
By the Cauchy-Riemann equations, the operators $\mathrm{HC}_{\Omega_k}$ are uniformly bounded (actually, they are isometric). Therefore, it is enough to prove
\begin{equation*}
\forall u\in \dot V_0(\Omega),\quad \lVert\mathrm{HC}_{\Omega_k}(u|_{\Omega_k})-(\mathrm{HC}_{\Omega}u)|_{\Omega_k}\rVert_{\dot V_0(\Omega_k)}\xrightarrow[k\to\infty]{}0.
\end{equation*}
Since both $u|_{\Omega_k}+i\mathrm{HC}_{\Omega_k}(u|_{\Omega_k})$ and $(u+i\mathrm{HC}_{\Omega}u)|_{\Omega_k}$ are holomorphic on $\Omega_k$ it follows that $i(\mathrm{HC}_{\Omega_k}(u|_{\Omega_k})-(\mathrm{HC}_{\Omega}u)|_{\Omega_k})$ is a purely imaginary-valued holomorphic function. Therefore, it is constant on $\Omega_k$, that is null as an element of $\dot V_0(\Omega_k)$. The convergence follows.
\end{proof}

Since the interior harmonic conjugations converge, and knowing that in the Lipschitz case, the harmonic conjugate of a double layer potential is a single layer potential, we can prove that property is preserved upon passing to the limit with respect to the domain. As a consequence, we can prove the Cauchy integral from Definition~\ref{Def:CI} is not only a sum of layer potentials, but also a holomorphic function away from the boundary.

\begin{theorem}[Holomorphism of the Cauchy integral]\label{Th:CI-Holom}
Let $\Omega$ be a two-sided admissible domain of $\C$. Assume there exists a non-decreasing sequence $(\Omega_k)_{k\in\N}$ of Lipschitz domains such that $\Omega_k\nearrow\Omega$, and that the extension operators $(\dot{\mathrm{E}}_{\Omega_k})_{k\in\N}$ and $(\dot{\mathrm{E}}_{\overline{\Omega}_k^c})_{k\in\N}$ from~\eqref{Eq:Dir-Int-Ext} and~\eqref{Eq:Dir-Ext-Ext} are uniformly bounded. Then, for all $f\in\dot\B_\C(\bord)$, the Cauchy integral $\Phi_{\bord}f$ is holomorphic on $\cbord$.
\end{theorem}

\begin{proof}
Assume $f$ is real-valued. Let $(f_k\in\dot \B(\bord_k))_{k\in\N}$ be such that $f_k\to f$ through $(\trd^{\bord_k}\circ\,\dot{\mathrm{H}}_{\bord})_{k\in\N}$. Then, by Theorem~\ref{Th:CV-LP-H1-Rn} and Proposition~\ref{Prop:HarmoConjCV-i}, it holds
\begin{equation*}
\Dd_{\bord_k}f_k\xrightarrow[k\to\infty]{(\cdot|_{\Omega_k}\oplus\mathcal{E}_k)}\Dd_{\bord}f\quad\mbox{hence}\quad \mathrm{HC}_{\Omega_k}((\dot{\D}_{\bord_k}f_k)|_{\Omega_k})\xrightarrow[k\to\infty]{(\cdot|_{\Omega_k})}\mathrm{HC}_\Omega((\dot{\D}_{\bord}f)|_\Omega).
\end{equation*}
By Proposition~\ref{Prop:Link-CI-LP}, since for all $k\in\N$, $\Omega_k$ is a Lipschitz domain, the harmonic conjugate of $\dot{\D}_{\bord_k}f_k$ is a single layer potential: there exists $g_k\in\dot{\B}'(\bord_k)$ such that the harmonic conjugate of $-\Dd_{\bord_k}f_k$ is $\Sd_{\bord_k}g_k$. Moreover, since $\mathrm{HC}_\Omega((\dot{\D}_{\bord}f)|_\Omega)$ is harmonic, it can be written as $-(\Sd_{\bord}g)|_\Omega$ for some $g\in\Hmo$. Then, $\Sd_{\bord}g=\dot{\mathrm{E}}_{\Omega}((\Sd_{\bord}g)|_\Omega)$ so that by Proposition~\ref{Prop:Link-CV-H1Rn}, it holds
\begin{equation*}
\Sd_{\bord_k}g_k\xrightarrow[k\to\infty]{}\Sd_{\bord}g\quad\mbox{in }\dot H^1_\C(\R^n).
\end{equation*}
For all $k\in\N$, $\Phi_{\bord_k}f_k=-\Dd_{\bord_k}f_k+i\Sd_{\bord_k}g_k$ is holomorphic on $\cbord_k$. Therefore, by Theorem~\ref{Th:CV-LP-H1-Rn}, $-\Dd_{\bord}f+i\Sd_{\bord}g$ is holomorphic on $\cbord$, hence $\Phi_{\bord}f=-\Dd_{\bord}f+i\Sd_{\bord}g$ and the result holds.
The proof in the purely imaginary case is similar.
\end{proof}

\begin{remark}
If $\Omega$ is a Koch snowflake, then the hypotheses of Theorem~\ref{Th:CI-Holom} are fulfilled and the Cauchy integral for $\Omega$ from Definition~\ref{Def:CI} is holomorphic on $\cbord$. It is also the case when the boundary is described by a Koch mixture (see~\cite{capitanelli_robin_2010}) as it is discussed in~\cite[Appendix B]{dekkers_mixed_2022}, and for a variety of self-similar fractals.
\end{remark}

\section{Equivalence of trace norms}\label{Sec:EquivNorm}

If $\Omega$ is a two-sided admissible domain, then it was proved in Proposition~\ref{Prop:Norm-B-Equiv} that, on the space $\Hpo$, the norms $\lVert\cdot\rVert_{\Hpo}=\lVert\cdot\rVert_{\Trd^{\bord}_{\mathrm{i}}}$ and $\lVert\cdot\rVert_{\Trd^{\bord}_{\mathrm{e}}}$ are equivalent, with optimal constants expressed in terms of the harmonic extension operators.
Given their definition (see also Theorem~\ref{Th:Trace}), the norm $\lVert\cdot\rVert_{\Trd^{\bord}_{\mathrm{i}}}$ is connected to the space $\dot H^1(\Omega)$, while the norm $\lVert\cdot\rVert_{\Trd^{\bord}_{\mathrm{e}}}$ is connected to $\dot H^1(\overline{\Omega}^c)$.
With that in mind, it can be deemed natural to consider another norm on the space $\Hpo$ in a way which is more `symmetric' in the sense that it is not centered on the interior domain, nor on the exterior domain:
\begin{equation}\label{Eq:TraceNorm-Rn}
\lVert f\rVert_{\Trd^{\bord}}:=\min\big\{\lVert v \rVert_{\dot{H}^1(\R^n)}\mid\ v\in \dot{H}^1(\R^n) \mbox{ and } \Trd^{\bord}v=f\big\}.
\end{equation}
Note that the trace in the previous formula can be regarded as either interior or exterior, since they coincide for elements of $\dot H^1(\R^n)$. It follows by energy minimization that, for all $f\in\Hpo$, $v\in \dot H^1(\R^n)$ with $\Trd^{\bord}v=f$ is such that $\lVert f\rVert_{\Trd^{\bord}}=\lVert v \rVert_{\dot H^1(\R^n)}$ if and only if $v\in \dot V_\Scal(\nbord)$.

\begin{remark}\label{rem-equivalent}
Given the isometry $\lVert\cdot\rVert_{\Trd^{\bord}}$ induces between the trace space and $\dot H^1(\R^n)$,   this norm is non other than the norm 
defined by~\eqref{Eq:Vnorm}
\begin{equation}\label{Eq:TraceNorm-Rn-}  
\lVert f\rVert_{\Vd_{\bord}^{-1}}=\lVert f\rVert_{\Trd^{\bord}}.
\end{equation}
 In this section, to remain consistent with the notations $\lVert\cdot\rVert_{\Trd^{\bord}_{\mathrm i}}$ and $\lVert\cdot\rVert_{\Trd^{\bord}_{\mathrm e}}$, we will keep denoting it by $\lVert\cdot\rVert_{\Trd^{\bord}}$.
\end{remark}

\begin{proposition}\label{prop-equivalent}
If $\Omega$ is a two-sided admissible domain, then the norms on $\Hpo$ defined in Theorem~\ref{Th:Trace} are equivalent to the norm from~\eqref{Eq:TraceNorm-Rn}:
\begin{equation}\label{NormTrEq}
\lVert\cdot\rVert_{\Trd^{\bord}}\le \lVert\dot {\mathrm{E}}_\Omega\rVert\,\lVert\cdot\rVert_{\Trd^{\bord}_{\mathrm{i}}}\qquad \mbox{and}\qquad \lVert\cdot\rVert_{\Trd^{\bord}_{\mathrm{i}}}\le  \sqrt{1-\frac1{\lVert\dot{\mathrm{E}}_{\overline{\Omega}^c}\rVert^2}}\,\lVert\cdot\rVert_{\Trd^{\bord}},
\end{equation}
and
\begin{equation*}
\lVert\cdot\rVert_{\Trd^{\bord}}\le \lVert\dot {\mathrm{E}}_{\overline{\Omega}^c}\rVert\,\lVert\cdot\rVert_{\Trd^{\bord}_{\mathrm{e}}}\qquad \mbox{and}\qquad \lVert\cdot\rVert_{\Trd^{\bord}_{\mathrm{e}}}\le  \sqrt{1-\frac1{\lVert\dot{\mathrm{E}}_{\Omega}\rVert^2}}\,\lVert\cdot\rVert_{\Trd^{\bord}},
\end{equation*}
with optimal constants.
\end{proposition}

\begin{proof}
By Theorem~\ref{Th:Trace}, Equation~\eqref{NormTrEq} is equivalent to:
\begin{equation*}
\lVert u\rVert_{\dot H^1(\R^n)}\le \lVert\dot {\mathrm{E}}_\Omega\rVert\,\lVert u\rVert_{\dot H^1(\Omega)}\qquad \mbox{and}\qquad \lVert\cdot\rVert_{\dot H^1(\Omega)}\le  \sqrt{1-\frac1{\lVert\dot{\mathrm{E}}_{\overline{\Omega}^c}\rVert^2}}\,\lVert u\rVert_{\dot H^1(\R^n)},
\end{equation*}
for $u\in \dot V_\Scal(\nbord)$.
The first inequality and optimality of the constant follow from $u=\dot{\mathrm{E}}_\Omega(u|_\Omega)$. Moreover, $u=\dot{\mathrm{E}}_{\overline{\Omega}^c}(u|_{\overline{\Omega}^c})$, hence
\begin{equation*}
\frac{\lVert u|_{\Omega}\rVert_{\dot H^1(\Omega)}^2}{\lVert u\rVert_{\dot H^1(\R^n)}^2}=1-\frac{\lVert u|_{\overline{\Omega}^c}\rVert_{\dot H^1(\overline{\Omega}^c)}^2}{\lVert u\rVert_{\dot H^1(\R^n)}^2}\le 1-\frac1{\lVert\dot {\mathrm{E}}_{\overline{\Omega}^c}\rVert^2}.
\end{equation*}
The second inequality follows, and by the same reasoning as for Proposition~\ref{Prop:Norm-B-Equiv}, so does the optimality of the constant. The proof for the second equivalence is similar.
\end{proof}

\section{Forgoing the monotonicity assumption}\label{Sec:Monotonicity}

In the spirit of the dyadic approximation introduced in Section~\ref{Sec:DyadApprox}, the standard assumption made about the convergence of the sequence $(\Omega_k)_{k\in\N}$ to $\Omega$ in Section~\ref{Sec:CV-LP}, all (two-sided) admissible domains, is that of a monotone convergence: $\Omega_k\nearrow\Omega$.
It seems natural to wonder if that monotonicity condition can be foregone, and at what cost.

To answer that question, the first step -- and perhaps the most important, see Figure~\ref{Fig:CV-Hilb} -- is to determine a convergence framework for the $\dot H^1$ spaces. The main purpose of the monotonicity assumption is to use the restrictions operators $\cdot|_{\Omega_k}$ to make the interior $\dot H^1$ spaces converge, see Lemma~\ref{Lem:CV-H1-i}. Using those operators, the convergence holds without any kind of regularity assumption on $\Omega$ or on $(\Omega_k)_{k\in\N}$. That assumption also arises for the convergence of the exterior $\dot H^1$ spaces in Corollary~\ref{Cor:CV-H1-e} for it allows to think of any element of $\dot H^1(\overline{\Omega}^c)$ as an element $\dot H^1(\overline{\Omega}_k^c)$ by means of $i_k$, the extension by $0$.

Let $\Omega$ be an admissible domain and $(\Omega_k)_{k\in\N}$ be a sequence of admissible domains such that $\Omega_k\to\Omega$ in the sense of characteristic functions.
Denoting by $\mathrm{Ext}_\Omega:\dot H^1(\Omega)\to \dot H^1(\R^n)$ any linear bounded $\dot H^1$ extension from $\Omega$ to $\R^n$, it holds $\dot H^1(\Omega_k)\to \dot H^1(\Omega)$ through $(\dot H^1(\Omega),((\mathrm{Ext}_\Omega\cdot)|_{\Omega_k})_{k\in\N})$.
Note that the extension property for $(\Omega_k)_{k\in\N}$ is not required for this result to hold.
However, as it was pointed out in Subsection~\ref{Subsec:CV-Ext}, such a construction does not preserve harmonicity, which is rather distasteful for the study of the harmonic layer potential operators.
For that matter, we use a convergence framework using double extension operators similar to the sequence $(\mathcal{E}_k)_{k\in\N}$ in Corollary~\ref{Cor:CV-H1-e}. The latter result relies on Proposition~\ref{Prop:ExtCV-i} on the convergence in $\dot H^1(\R^n)$ of harmonic extensions, which the following result generalizes under the assumption of domain convergence in the sense of compact sets, which implies that in the sense of characteristic functions for two-sided admissible domains~\cite[Exercise 2.7]{henrot_shape_2018}.

\begin{proposition}\label{A-Prop:Ext-CV}
Let $\Omega$ be a two-sided admissible domain and $(\Omega_k)_{k\in\N}$ be a sequence of two-sided admissible domains such that $\Omega_k\to\Omega$ in the sense of compact sets. Then, for all $u\in \dot H^1(\R^n)$,
\begin{equation*}
\dot{\mathrm{E}}_{\Omega_k}(u|_{\Omega_k})\xrightarrow[k\to\infty]{}\dot{\mathrm{E}}_\Omega (u|_\Omega)\quad\mbox{in }\dot H^1(\R^n).
\end{equation*}
If $\overline{\Omega}^c$ and $\overline{\Omega}_k^c$, $k\in\N$, are admissible domains, then 
\begin{equation*}
\dot{\mathrm{E}}_{\overline{\Omega}_k^c}(u|_{\overline{\Omega}_k^c})\xrightarrow[k\to\infty]{}\dot{\mathrm{E}}_{\overline{\Omega}^c} (u|_{\overline{\Omega}^c})\quad\mbox{in }\dot H^1(\R^n).
\end{equation*}
\end{proposition}

\begin{proof}
To prove the first convergence, we proceed as for Proposition~\ref{Prop:ExtCV-i}. As for~\eqref{Eq:Extensions-UpperBound}, energy minimization yields
\begin{equation*}
\forall k\in\N,\quad \lVert\dot{\mathrm{E}}_{\Omega_k}(u|_{\Omega_k})\rVert_{\dot H^1(\R^n)}\le \lVert u\rVert_{\dot H^1(\R^n)}.
\end{equation*}
Following the proof leads to
\begin{equation}\label{A-Eq:Ext-WCV}
\dot{\mathrm{E}}_{\Omega_k}(u|_{\Omega_k})\xrightharpoonup[k\to\infty]{\dot H^1(\R^n)}\dot{\mathrm{E}}_\Omega (u|_\Omega).
\end{equation}
Indeed, the convergence in the sense of compact sets implies that any weak subsequential limit of $(\dot{\mathrm{E}}_{\Omega_k}(u|_{\Omega_k}))_{k\in\N}$ is $\dot H^1(\R^n)$, equal to $u$ on $\Omega$ and harmonic on $\overline{\Omega}^c$.
All that is left to show is the convergence of the $\dot H^1(\R^n)$ norms.
The convergence in the sense of characteristic functions directly yields
\begin{equation*}
\lVert(\dot{\mathrm{E}}_{\Omega_k}(u|_{\Omega_k}))|_{\Omega_k}\rVert^2_{\dot H^1(\Omega_k)}=\int_\Omega|\nabla u|^2\,\dx\xrightarrow[k\to\infty]{}\int_{\Omega_k}|\nabla u|^2\,\dx=\lVert(\dot{\mathrm{E}}_\Omega (u|_\Omega))|_\Omega\rVert^2_{\dot H^1(\Omega)}.
\end{equation*}
Since $\dot{\mathrm{E}}_\Omega (u|_\Omega)$ is harmonic on $\overline{\Omega}^c$, it holds
\begin{equation*}
\int_{\overline{\Omega}^c}\nabla(\dot{\mathrm{E}}_\Omega (u|_\Omega))\cdot\nabla u\,\dx=\lVert(\dot{\mathrm{E}}_\Omega (u|_\Omega))|_{\overline{\Omega}^c}\rVert^2_{\dot H^1(\overline{\Omega}^c)}.
\end{equation*}
Similarly, it holds for all $k\in\N$,
\begin{equation*}
\int_{\overline{\Omega}_k^c}\nabla(\dot{\mathrm{E}}_{\Omega_k} (u|_{\Omega_k}))\cdot\nabla u\,\dx=\lVert(\dot{\mathrm{E}}_\Omega (u|_{\Omega_k}))|_{\overline{\Omega}_k^c}\rVert^2_{\dot H^1(\overline{\Omega}_k^c)}.
\end{equation*}
Since $(\nabla u)\mathds{1}_{\Omega_k}\to (\nabla u)\mathds{1}_\Omega$ in $L^2(\R^n)$ and by~\eqref{A-Eq:Ext-WCV}, the convergence of the norms follows. The proof of the second convergence is similar.
\end{proof}

The convergence results from Proposition~\ref{A-Prop:Ext-CV} allow to prove a variant of Proposition~\ref{Prop:ExtCV-e} (considering $u=\dot{\mathrm{E}}_\Omega v$ for $v\in\dot H^1(\Omega)$ and $u=\dot{\mathrm{E}}_{\overline{\Omega}^c} v$ for $v\in \dot H^1(\overline{\Omega}^c)$) and prove the following convergences of the $\dot V_0$ spaces. We only focus on the convergence of harmonic functions for only the harmonic part is relevant when considering harmonic layer potentials.

\begin{proposition}\label{A-Prop:CV-H1}
Let $\Omega$ be a two-sided admissible domain. Let $(\Omega_k)_{k\in\N}$ be a sequence of two-sided admissible domains such that $\Omega_k\to\Omega$ in the sense of compact sets. Then, it holds
\begin{align*}
\dot V_0(\Omega_k)\xrightarrow[k\to\infty]{}\dot V_0(\Omega)\quad&\mbox{through}\quad \big(\dot V_0(\Omega),\big(\big(\dot{\mathrm{E}}_{\overline{\Omega}_k^c}\big[(\dot{\mathrm{E}}_{\Omega}\cdot)|_{\overline{\Omega}_k^c}\big]\big)\big|_{\Omega_k})_{k\in\N}\big),\\
\dot V_0(\overline{\Omega}_k^c)\xrightarrow[k\to\infty]{}\dot V_0(\overline{\Omega}^c)\quad&\mbox{through}\quad \big(\dot V_0(\overline{\Omega}^c),\big(\big(\dot{\mathrm{E}}_{\Omega_k}\big[(\dot{\mathrm{E}}_{\overline{\Omega}^c}\cdot)|_{\Omega_k}\big]\big)\big|_{\overline{\Omega}_k^c})_{k\in\N}\big).
\end{align*}
\end{proposition}

\begin{remark}
In the framework of Lemma~\ref{Lem:CV-H1-i}, which only used restrictions, no regularity assumption was made on the domain $\Omega$, nor on the converging sequence $(\Omega_k)_{k\in\N}$. In the convergence framework of Proposition~\ref{A-Prop:CV-H1} however, the use of double extension operators leads to assuming the domains are two-sided admissible.
\end{remark}

Using the convergence framework of Proposition~\ref{A-Prop:CV-H1}, we can follow the steps taken in Subsection~\ref{Subsec:CV-Ext} replacing the monotone convergence $\Omega_k\nearrow\Omega$ with a convergence in the sense of compact sets.

\begin{proposition}\label{Prop:CV-B-K}
Let $\Omega$ be a two-sided admissible domain. Let $(\Omega_k)_{k\in\N}$ be a sequence of two-sided admissible domains such that $\Omega_k\to\Omega$ in the sense of compact sets. Then, it holds
\begin{equation*}
\dot \B(\bord_k)\xrightarrow[k\to\infty]{}\Hpo \quad \mbox{through }\big(\Hpo,(\trd^{\bord_k}\circ\,\dot{\mathrm{H}}_{\bord})_{k\in\N}\big).
\end{equation*}
\end{proposition}

\begin{proof}
Let $f\in\Hpo$.
Since $\Omega_k\to\Omega$ in the sense of compact sets, it holds $\mathds{1}_{\Omega_k}\to\mathds{1}_\Omega$ pointwise on $\nbord$.
Hence, by Proposition~\ref{A-Prop:Ext-CV} and dominated convergence, it holds
\begin{equation*}
\lVert(\dot{\mathrm{E}}_{\overline{\Omega}_k^c}(\dot{\mathrm{H}}_{\bord}f)|_{\overline{\Omega}_k^c})|_{\Omega_k}\rVert^2_{\dot H^1(\Omega_k)}\xrightarrow[k\to\infty]{}\lVert(\dot{\mathrm{E}}_{\overline{\Omega}^c}(\dot{\mathrm{H}}_{\bord}f)|_{\overline{\Omega}^c})|_{\Omega}\rVert^2_{\dot H^1(\Omega)}.
\end{equation*}
By Theorem~\ref{Th:Trace}, Point (iii), it follows that
\begin{equation*}
\lVert\trd^{\bord_k}\dot{\mathrm{H}}_{\bord}f\rVert_{\dot \B(\bord_k)}\xrightarrow[k\to\infty]{}\lVert f\rVert_{\Hpo}.\qedhere
\end{equation*}
\end{proof}

For instance, we prove the link between the convergence of harmonic functions and of their traces, that is a variant of Propositions~\ref{Prop:TrCV-i} and~\ref{Prop:TrCV-e} under those new hypotheses.

\begin{proposition}
Let $\Omega$ be a two-sided admissible domain and $(\Omega_k)_{k\in\N}$ be a sequence of two-sided admissible domains such that $\Omega_k\to\Omega$ in the sense of compact sets. Let $u\in \dot V_0(\Omega)$ and $(u_k\in \dot V_0(\Omega_k))_{k\in\N}$. Then, the following equivalence holds:
\begin{multline*}
u_k\xrightarrow[k\to\infty]{((\dot{\mathrm{E}}_{\overline{\Omega}_k^c}[(\dot{\mathrm{E}}_{\Omega}\cdot)|_{\overline{\Omega}_k^c}])|_{\Omega_k})}u\\
\iff\quad (\trd_{\mathrm{i}}^{\bord_k}u_k\in\dot \B(\bord_k))_{k\in\N}\xrightarrow[]{(\trd^{\bord_k}\circ\,\dot{\mathrm{H}}_{\bord})}\trd^{\bord}_{\mathrm{i}}u\in\Hpo.
\end{multline*}
If the Dirichlet harmonic extensions $(\dot{\mathrm{E}}_{\Omega_k})_{k\in\N}$ are uniformly bounded, then the following implication holds:
\begin{equation*}
(v_k\in \dot V_0(\overline{\Omega}_k^c))_{k\in\N}\xrightarrow[]{((\dot{\mathrm{E}}_{\Omega_k}[(\dot{\mathrm{E}}_{\overline{\Omega}^c}\cdot)|_{\Omega_k}])|_{\overline{\Omega}_k^c})}v\in \dot V_0(\overline{\Omega}^c),
\end{equation*}
implies
\begin{equation*}
(\trd_{\mathrm{e}}^{\bord_k}v_k\in\dot \B(\bord_k))_{k\in\N}\xrightarrow[]{(\trd^{\bord_k}\circ\,\dot{\mathrm{H}}_{\bord})}\trd_{\mathrm{e}}^{\bord}v\in\Hpo.
\end{equation*}
If the Dirichlet harmonic extensions $(\dot{\mathrm{E}}_{\overline{\Omega}_k^c})_{k\in\N}$ are uniformly bounded, then the converse implication holds instead.
\end{proposition}

In the same spirit, we can prove the link between the convergence of harmonic functions and of their interior normal derivatives (as in Proposition~\ref{Prop:CV-ddni}) as the domains converge in the sense of compact sets only.
The link with the exterior normal derivatives is proved as Proposition~\ref{Prop:CV-ddne}, by introducing another convergence framework for the duals of the trace spaces.

\begin{proposition}\label{Prop:CV-B'-K}
Let $\Omega$ be a two-sided admissible domain. Let $(\Omega_k)_{k\in\N}$ be a sequence of two-sided admissible domains such that $\Omega_k\to\Omega$ in the sense of compact sets. Then, it holds
\begin{equation*}
\dot \B'(\bord_k)\xrightarrow[k\to\infty]{}\Hpo \quad \mbox{through }\big(\Hmo,({\textstyle\ddno{}{}}\big|_{\bord_k}\circ\,\dot{\mathrm{N}}_{\bord})_{k\in\N}\big).
\end{equation*}
\end{proposition}

\begin{proof}
Let $g\in\Hmo$. Denote $u:=\dot{\mathrm N}_{\bord}g\in \dot V_{\D}(\nbord)$. For all $k\in\N$, denote by $u_k$ the unique element of $\dot V_{\D}(\nbord_k)$ such that $u_k|_{\Omega_k}=u|_{\Omega_k}$. As in the proof of Proposition~\ref{A-Prop:Ext-CV}, it holds
\begin{equation}\label{Eq:Neum-Ext-Bound}
\lVert u_k\rVert_{\dot H^1(\nbord_k)}\le \lVert u\rVert_{\dot H^1(\nbord)}.
\end{equation}
Consider an exhaustion by compact sets  $(K_m)_{m\in\N}$ of $\nbord$. Then, by the convergence in the sense of compact sets, there exists (up to a subsequence) $u_\infty\in\dot V_0(\nbord)$ such that $u_\infty|_\Omega=u|_\Omega$ and
\begin{equation*}
\forall m\in \N,\quad u_k|_{K_m}\xrightharpoonup[k\to\infty]{}u_\infty|_{K_m}\quad\mbox{in }\dot H^1(K_m).
\end{equation*}
By Lemma~\ref{Lem:DLP-ortho}, for all $k\in\N$ and $v\in\dot H^1(\R^n)\subset\ker\llbracket\Trd^{\bord}\rrbracket$, it holds $\langle u_k,v\rangle_{\dot H^1(\nbord_k)}=0$. In addition,
\begin{multline*}
\lvert\langle u_k,v\rangle_{\dot H^1(\nbord_k)}-\langle u_\infty,v\rangle_{\dot H^1(\nbord)}\rvert\le\left\lvert\int_{\R^n}\mathds{1}_{K_m}\nabla(u_k-u_\infty)\cdot\nabla v\,\dx\right\rvert\\
+\left\lvert\int_{\R^n}(\mathds{1}_{\nbord}-\mathds{1}_{K_m})\nabla u_\infty\cdot\nabla v\,\dx\right\rvert+\left\lvert\int_{\R^n}(\mathds{1}_{\nbord_k}-\mathds{1}_{K_m})\nabla u_k\cdot\nabla v\,\dx\right\rvert.
\end{multline*}
Taking successively the limits in $k$ and in $m$ (and using~\eqref{Eq:Neum-Ext-Bound} for the third term) yields $\langle u_\infty,v\rangle_{\dot H^1(\nbord)}=0$, for all $v\in\dot H^1(\R^n)$. It follows that $u_\infty\in\dot V_\D(\nbord)$, hence $u_\infty=u$. Along with~\eqref{Eq:Neum-Ext-Bound}, this yields $u_k|_{K_m}\to u|_{K_m}$ for all $m\in\N$. Hence,
\begin{equation*}
\forall m\in\N,\quad \lVert u_k|_{\Omega_k}\rVert_{\dot H^1(\Omega_k)}\xrightarrow[k\to\infty]{}\lVert u\rVert_{\dot H^1(\Omega)},
\end{equation*}
and the result follows by Proposition~\ref{Prop:ddn-isom}.
\end{proof}

In the monotone case, it was pointed out in Remark~\ref{Rem:B'-Change} that both frameworks are equivalent for the operators $\Xi_k$ and $\ddno{}{}|_{\bord_k}\circ\dot{\mathrm N}_{\bord}$ are the same. Under the assumption of a convergence in the sense of compact sets alone, those operators may differ, but still induce equivalent notions of convergence.

\begin{lemma}\label{Lem:Equiv-B'-Frameworks}
Let $\Omega$ be a two-sided admissible domain. Let $(\Omega_k)_{k\in\N}$ be a sequence of two-sided admissible domains such that $\Omega_k\to\Omega$ in the sense of compact sets. Then, for all $g\in\Hmo$, it holds
\begin{equation*}
\left\lVert\Xi_kg-\ddno{}{}\Big|_{\bord_k}\dot{\mathrm N}_{\bord}g\right\rVert_{\dot\B'(\bord_k)}\xrightarrow[k\to\infty]{}0.
\end{equation*}
\end{lemma}

\begin{proof}
Let $(u_k\in\dot V_0(\Omega_k))_{k\in\N}$ and $u\in \dot V_0(\Omega)$. By Proposition~\ref{Prop:ddn-isom}, it holds
\begin{equation*}
\left\lVert \ddno{u_k}{\mathrm i}\Big|_{\bord_k}-\ddno{}{\mathrm i}\Big|_{\bord_k}\dot{\mathrm N}_{\bord}\ddno{u}{\mathrm i}\Big|_{\bord}\right\rVert_{\dot\B'(\bord_k)} = \lVert u_k - (\dot{\mathrm E}^{\mathrm N}_{\overline{\Omega}_k^c}(\dot{\mathrm E}^{\mathrm N}_\Omega u)|_{\overline{\Omega}_k^c})|_{\Omega_k}\rVert_{\dot H^1(\Omega_k)},
\end{equation*}
involving the Neumann harmonic extensions as in the proof of Proposition~\ref{Prop:CV-ddne}.
In addition, by Proposition~\ref{A-Prop:Ext-CV} and proceeding as in the proof of Proposition~\ref{Prop:CV-B'-K}, it holds
\begin{equation*}
\lVert (\dot{\mathrm E}_{\overline{\Omega}_k^c}(\dot{\mathrm E}_\Omega u)|_{\overline{\Omega}_k^c})|_{\Omega_k} - (\dot{\mathrm E}^{\mathrm N}_{\overline{\Omega}_k^c}(\dot{\mathrm E}^{\mathrm N}_\Omega u)|_{\overline{\Omega}_k^c})|_{\Omega_k}\rVert_{\dot H^1(\Omega_k)}\xrightarrow[k\to\infty]{}0,
\end{equation*}
both converging in $\dot H^1$ on (every compact subset of) $\Omega$ to $u$. It follows that
\begin{equation*}
\ddno{u_k}{\mathrm{i}}\Big|_{\bord_k}\xrightarrow[k\to\infty]{(\Xi_k)}\ddno u{\mathrm{i}}\Big|_{\bord}\quad
\iff\quad\ddno{u_k}{\mathrm{i}}\Big|_{\bord_k}\xrightarrow[k\to\infty]{(\ddno{}{}|_{\bord_k}\circ\dot{\mathrm{N}}_{\bord})}\ddno u{\mathrm{i}}\Big|_{\bord}.
\end{equation*}
Since $\ddno{}{}|_{\bord}(\dot V_0(\Omega))=\Hmo$, the result follows by Lemma~\ref{Lem:EquivCVFramework}.
\end{proof}

Using those results, we follow the method from Subsection~\ref{Subsec:CV-Rn} (notably Theorem~\ref{Th:CV-LP} and Proposition~\ref{Prop:Link-CV-H1Rn}, Point (ii)) to generalize Theorem~\ref{Th:CV-LP-H1-Rn}.

\begin{theorem}\label{A-Th:CV-LP-H1-Rn}
Let $\Omega$ be a two-sided admissible domain. Assume there exists a sequence $(\Omega_k)_{k\in\N}$ of two-sided admissible domains such that $\Omega_k\to\Omega$ in the sense of compact sets, and that the extension operators $(\dot{\mathrm{E}}_{\Omega_k})_{k\in\N}$ and $(\dot{\mathrm{E}}_{\overline{\Omega}_k^c})_{k\in\N}$ from~\eqref{Eq:Dir-Int-Ext} and~\eqref{Eq:Dir-Ext-Ext} are uniformly bounded. Then,
\begin{enumerate}
\item[(i)] for every $(f_k\in\dot \B(\bord_k))_{k\in\N}\xrightarrow[k\to\infty]{(\trd^{\bord_k}\circ\,\dot{\mathrm{H}}_{\bord})} f\in\Hpo$, it holds
\begin{align*}
&& \dot{\mathrm{E}}_{\Omega_k}(\Dd_{\bord_k}f_k|_{\Omega_k})&\xrightarrow[k\to\infty]{}\dot{\mathrm{E}}_\Omega(\Dd_{\bord}f|_\Omega) &&\mbox{in }\dot H^1(\R^n),\\
&\text{and} &
\dot{\mathrm{E}}_{\overline{\Omega}_k^c}(\Dd_{\bord_k}f_k|_{\overline{\Omega}_k^c})&\xrightarrow[k\to\infty]{}\dot{\mathrm{E}}_{\overline{\Omega}^c}(\Dd_{\bord}f|_{\overline{\Omega}^c}) &&\mbox{in }\dot H^1(\R^n),
\end{align*}
using the convergence framework from Lemma~\ref{Lem:CV-B};

\item[(ii)] for every $(g_k\in\dot\B'(\bord_k))_{k\in\N}\xrightarrow[k\to\infty]{(\Xi_k)} g\in\Hmo$, it holds
\begin{equation*}
\Sd_{\bord_k}g_k\xrightarrow[k\to\infty]{}\Sd_{\bord}g\qquad\mbox{in }\dot H^1(\R^n),
\end{equation*}
using the convergence framework from Lemma~\ref{Lem:CV-B'}.
\end{enumerate}
\end{theorem}

\section*{Acknowledgments}
The authors are most grateful to Simon N. Chandler-Wilde, David Hewett and Michael Hinz for insightful discussions on the topic, as well as to Irina Mitrea and Luke Rogers, for their work has motivated the present paper.

\section*{Fundings}
Research supported in part by CNRS INSMI IEA (International Emerging Actions 2022) “Functional and applied analysis with fractal or non-Lipschitz boundaries”.

A.T.: Research was supported in part by NSF Grant DMS-2349433, the Simons Foundation and the Fulbright Program. The project was completed during a stay at CentraleSup\'elec, Universit\'e Paris-Saclay, whose kind hospitality is gratefully acknowledged.

\def\refname{References}
\bibliographystyle{siam}
\bibliography{BibGC.bib}

\end{document}